\documentclass[11pt, a4paper, USenglish]{article}
\pdfoutput=1
\usepackage{enumitem}
\usepackage[utf8]{inputenc}
\usepackage{amssymb,amsthm,amsmath}
\usepackage[usestackEOL]{stackengine} 
\usepackage{mathtools} 
\usepackage{mathrsfs} 
\usepackage{breqn}
\usepackage{microtype}
\usepackage{float}
\usepackage[USenglish]{babel}
\usepackage[T1]{fontenc}
\usepackage{lmodern}
\usepackage{caption}
\usepackage{physics} 
\usepackage{tikz}
\usepackage{titling} 
\usepackage[left=3.25cm, right=3.25cm]{geometry}
\usepackage{cite}
\usepackage[colorlinks=true,linkcolor=blue,citecolor=blue]{hyperref}

\theoremstyle{plain}
\newtheorem{prop}{Proposition}[section]
\newtheorem{lemma}[prop]{Lemma}
\newtheorem{thm}[prop]{Theorem}
\newtheorem{cor}[prop]{Corollary}

\theoremstyle{definition}
\newtheorem{defi}[prop]{Definition}
\newtheorem*{ack}{Acknowledgements}
\theoremstyle{remark}
\newtheorem{remark}[prop]{Remark}

\let\temp\phi
\let\phi\varphi
\let\varphi\temp
\newcommand*\dif{\mathop{}\!\mathrm{d}}

\DeclareMathOperator{\cqq}{\coloneqq}

\DeclareMathOperator{\wto}{\rightharpoonup}

\newcommand{\vertiii}[1]{{\left\vert\kern-0.25ex\left\vert\kern-0.25ex\left\vert #1 
    \right\vert\kern-0.25ex\right\vert\kern-0.25ex\right\vert}}
\DeclareMathOperator{\embeds}{\hookrightarrow}

\newcommand\N{\mathbb{N}}

\newcommand\R{\mathbb{R}}

\newcommand\C{\mathbb{C}}

\newcommand{\mc}[1]{\mathcal #1}

\newcommand{\mb}[1]{\mathbb #1}

\DeclareMathOperator{\id}{Id}

\newcommand\Vector[1]{\setstackEOL{,}\parenVectorstack{#1}}

\newcommand{\parpar}[2]{\partial_{#1} #2}

\newcommand{\eps}{\varepsilon}

\newcommand{\limn}{\lim _{n\to \infty}}
\DeclareMathOperator{\im}{Im}

\DeclareMathOperator{\Div}{div}
\DeclareMathOperator{\loc}{loc}

\def\dashint{\let\mathchoice\oldmathchoice\,\ThisStyle{\ensurestackMath{%
            \stackinset{c}{.2\LMpt}{c}{.5\LMpt}{\SavedStyle-}{%
            \SavedStyle\phantom{\int}}}%
        \setbox0=\hbox{$\SavedStyle\int\,$}\kern-\wd0}\int%
        \let\mathchoice\newmathchoice}

	\newcommand{\mres}{\mathbin{\vrule height 1.6ex depth 0pt width
0.13ex\vrule height 0.13ex depth 0pt width 1.3ex}}

\DeclareMathOperator{\diam}{diam}
\DeclareMathOperator{\genus}{genus}

\DeclareMathOperator{\vol}{vol}
\DeclareMathOperator{\iso}{\mc{I}}

\DeclareMathOperator{\cat}{cat}
\DeclareMathOperator{\Capp}{cap}

\DeclareMathOperator\dom{dom}

\numberwithin{equation}{section}

\usepackage[normalem]{ulem} 

\usepackage{tocbibind}
\author{\textsc{Christian Scharrer} \and \textsc{Alexander West}}
\date{March 21, 2024}
\title{On the minimization of the Willmore energy under a constraint on total mean curvature and area}

\begin{document}
\maketitle
\begin{abstract}
    Motivated by a model for lipid bilayer cell membranes, we study the minimization of the Willmore functional in the class of oriented closed surfaces with prescribed total mean curvature, prescribed area, and prescribed genus. Adapting methods previously developed by  Keller-Mondino-Rivière, Bauer-Kuwert, and Ndiaye-Schätzle, we prove existence of smooth minimizers for a large class of constraints. Moreover, we analyze the asymptotic behaviour of the energy profile close to the unit sphere and consider the total mean curvature of axisymmetric surfaces.
\end{abstract}
\section{Introduction}\label{sec:Introduction}
In this paper, we will focus on surfaces immersed into $\R^3$. For a given immersion, we will investigate its Willmore energy, which can be described as a quantity measuring the amount of curvature present in the surface. Let us first give a definition. Given an immersion $f:\Sigma \to \R^3$ where $\Sigma$ is a 2-dimensional manifold, the \emph{Willmore energy $\mc W$ of $f$} is defined as
\begin{equation}
\mc W(f) \cqq \frac{1}{4} \int _{\Sigma} H^2 \dif\vol_{g_f}.\label{eq:Willmore energy}
\end{equation}
Here, $H= \kappa_1 + \kappa_2 $ is the sum of the principal curvatures $\kappa_1,\kappa_2$ and $\vol_{g_f}$ denotes the volume form induced by $g_f = f^* \langle \cdot, \cdot \rangle$, the pull back metric of the Euclidean metric $\langle \cdot, \cdot \rangle$ along $f$. The definition $H=\frac{\kappa_1+\kappa_2}{2}$ is also very common in literature. However, in this paper, we will stick to the first definition.

The Willmore energy was already considered by \textsc{Poisson} \cite{Poisson} in 1814 and \textsc{Germain} \cite{Germain} in 1821. In their work, they studied the elastic energy of thin elastic plates under vibrations. In the 1920s, it was studied by \textsc{Blaschke} and \textsc{Thomsen} \cite{BlaschkeThomsen} in the context of conformal differential geometry. They established the crucial fact that for closed surfaces, the Willmore energy is invariant under conformal transformations of the ambient space, i.e., 
\[\mc W(\Psi \circ f) = \mc W(f),\]
where $\Psi$ is a conformal diffeomorphism of an open subset in $\R^3$ containing the image of $f$. This includes translations and dilations, rotations and sphere inversions whose centers do not lie on the surface. Exploiting especially the latter invariance will be indispensable in this paper. 

In 1965, \textsc{Willmore} \cite{Willmore} studied the energy in \eqref{eq:Willmore energy}, which now bears his name. He discovered that for immersions $f:\mb S^2\to \R^3$, the Willmore energy is bounded from below by $4\pi$, with equality if and only if $f$ immerses a round sphere. This can be seen by the following argument: It holds that $H^2 = (\kappa_1 + \kappa_2)^2 \geq 4\kappa_1 \kappa_2 = 4K$, where $K=\kappa_1\kappa_2$ denotes the Gauss curvature. Thus
\begin{equation}
\frac{1}{4} \int _{\mb S^2}H^2 \dif\vol_{g_f}  \geq \int _{\mb S^2} K \dif\vol_{g_f} = 4\pi .\label{intro:Willmore inequality}
\end{equation}
The last equality follows from the Gauss-Bonnet Theorem. Equality holds if and only if $\kappa_1 = \kappa_2$ everywhere, i.e., every point of $f(\mb S^2)$ is umbilic, which implies that $f(\mb S^2)$ is a round sphere. 

 The inequality $\mc W(f)\geq 4\pi$ holds true for immersions $f:\Sigma \to \R^3$ for any closed surface $\Sigma$ where the equality case remains as before. Willmore also calculated the energy of rotationally symmetric tori with minor radius $r>0$ and major radius $R>r$, i.e., the set given by
\[\left (\sqrt{x^2+y^2}-R\right )^2 + z^2 = r^2.\]
He obtained that, in this special class, the Willmore energy was bounded by $2\pi^2$ from below, with equality if and only if $R=\sqrt{2}r$. This torus is known as the \emph{Clifford torus}. He conjectured that for any immersion $f:\mb T^2\to \R^3$, the Willmore energy is bounded from below by $2\pi^2$, and that the Clifford torus is a minimizer of the Willmore energy. This inequality became known as the \emph{Willmore conjecture} and was finally proved in 2012 by \textsc{Marques} and \textsc{Neves} \cite{MarquesNeves}.  

In 1982, \textsc{Li} and \textsc{Yau} \cite{LiYau} proved the following inequality. If $f:\Sigma \to \R^3$ is an immersion and $p\in \R^3$ such that $f^{-1}(p) = \{a_1,\ldots, a_k\}$ for distinct points $a_1, \ldots, a_k \in \Sigma$, then
\begin{equation}
\mc W(f)\geq 4\pi k.\label{intro:LiYau}
\end{equation}
In particular, if $\mc W(f)<8\pi$, then $f$ is automatically injective and thus an embedding. 

The Willmore energy naturally shows up in a lot of applications due to its simplicity. In cellular biology, the elastic energy of lipid bilayers, which are polar cell membranes made of two layers of lipid molecules, is closely related to the Willmore energy. If we model the cell membrane as a surface given by an immersion $f:\Sigma\to \R^3$, \textsc{Helfrich} \cite{Helfrich} proposed that the elastic energy could be modeled by the following expression:
\[ \int _{\Sigma} \frac{1}{2} k_c (H-c_0)^2 + \overline{k_c} K \dif\vol_{g_f}.\] 
Here, $k_c$ and $\overline{k_c}$ are curvature elastic moduli and $c_0$ is the spontaneous curvature. By the Gauss-Bonnet Theorem, the integral over $K$ is a topological constant; more precisely
\[\int _\Sigma K \dif \vol_{g_f} = 2\pi \chi(\Sigma ) = 4\pi(1-g), \]
where $\chi(\Sigma)$ is the Euler characteristic of $\Sigma$ and $g$ is its genus. Thus, for our mathematical model we set
\begin{equation}
\mc H^{c_0}(f) \cqq \int _\Sigma \left (\frac{H}{2}-c_0\right )^2 \dif \vol_{g_f}.\label{intro:CanhamHelfrichenergy}
\end{equation}
This will be called the \emph{Helfrich energy}. If $c_0=0$, then we recover the Willmore energy.

\subsection{The unconstrained problem} \label{subsec:IntroductionUnconstrainedProblem}
Let us first consider the minimization of the Willmore energy in the class of smooth surfaces of genus $g$. Let $\Sigma$ be any smooth, closed, oriented and connected genus $g$ surface\footnote{By the classification theorem for surfaces, all such surfaces are diffeomorphic to each other.}. We define 
\[\mc S_g \cqq \{f:\Sigma \to \R^3,\; \text{$f$ is a smooth immersion}\}.\]
Then, the minimal Willmore energy among $\mc S_g$ is defined by
\begin{equation}
\beta_g \cqq \inf _{f\in \mc S_g} \mc W(f).\label{intro:betag}
\end{equation}
It is natural for such a minimization problem to ask whether the infimum in \eqref{intro:betag} is attained. A pioneering result in that regard was provided by \textsc{Simon} \cite{Simon}, as he showed that under the condition 
\begin{equation}
\beta_g < \min\{8\pi, \omega_g\},\label{intro:betag<8piomegag}
\end{equation}
the infimum in \eqref{intro:betag} is attained. Here,
\begin{equation}
\omega_g   \cqq \min \left \{4\pi + \sum _{i=1}^p (\beta _{g_i} - 4\pi) , \;  g= g_1 + \ldots + g_p, \; 1\leq g_i < g\right \},\label{intro:omegag}
\end{equation}
where $\omega_1 \cqq \infty$. By the estimate $\beta_1\leq 2\pi^2 < 8\pi$ proved by Willmore, Simon's result already shows that a minimizer in the genus one case exists.  

It remained to verify \eqref{intro:betag<8piomegag} for $g>1$. The $8\pi$ bound was proved by \textsc{Kusner} \cite{Kusner1989} who estimated the Willmore energy of the stereographic projection of certain minimal surfaces in $\mb S^3$ found by \textsc{Lawson} \cite{Lawson1970}. To show that 
\begin{equation}
\beta_g < \omega_g,\label{intro:betag<omegag}
\end{equation}
the following approach was used. Suppose that for two embeddings $f_1:\Sigma_1\to \R^3$, $f_2:\Sigma_2\to \R^3$, none of which immerse round spheres, one can construct an embedding $f:(\Sigma_1 \# \Sigma_2) \to \R^3$ where $(\Sigma_1 \# \Sigma_2)$ denotes the connected sum of $\Sigma_1$ and $\Sigma_2$, such that 
\begin{equation}
\mc W(f) < \mc W(f_1) + \mc W(f_2) - 4\pi. \label{intro:BauerKuwertEstimate}
\end{equation}
Notice that $\genus(\Sigma_1 \# \Sigma_2) = \genus(\Sigma_1) + \genus(\Sigma_2)$. Then, \eqref{intro:betag<omegag} follows by inductively applying \eqref{intro:BauerKuwertEstimate} to minimizers of the Willmore energy for genus $g_i<g$. To prove \eqref{intro:BauerKuwertEstimate}, \textsc{Kusner} \cite{Kusner} suggested inverting $f_1$ at a nonumbilic point\footnote{Nonumbilic means that the two principal curvatures $\kappa_1,\kappa_2$ differ.} in the image of $f_1$, which always exists for a genus of at least 1, resulting in a noncompact surface of Willmore energy $\mc W(f_1) - 4\pi$. After that, the inverted surface is glued to a blown-up version of the $f_2$, again at another nonumbilic point. If one does the gluing process using the graph of a biharmonic function, it is possible to save energy in the gluing process, proving \eqref{intro:BauerKuwertEstimate}. This construction was done by \textsc{Bauer} and \textsc{Kuwert} \cite{BauerKuwert}. Later, \textsc{Marques} and \textsc{Neves} \cite{MarquesNeves} showed that $\beta_g > \beta_1 = 2\pi^2 > 6\pi$ for $g>1$, which also implies $\omega_g > 8\pi$ for $g\geq 1$\footnote{Notice however that the result in \cite{BauerKuwert} is more general, since their result holds for any codimension, not just for surfaces in $\R^3$. We omit the dependence on the dimension of the ambient space since we only work in $\R^3$.}.

\subsection{The isoperimetrically constrained problem}
\label{subsec:IntroductionIsoConstrainedProblem}
We have seen that the main term in the Helfrich energy \eqref{intro:CanhamHelfrichenergy} used to study the bending energy of cell membranes comes from the Willmore energy. Thus, in the following, we will restrict ourselves to the case $c_0=0$. 

According to \textsc{Helfrich} \cite{Helfrich}, the area and volume of a vesicle should be treated as fixed and therefore, the shape of the cell membrane should minimize the Willmore energy under a constraint on the area and the enclosed volume. This problem was first considered for spherical immersions by \textsc{Canham} \cite{Canham1970} in 1970 in an attempt to explain the biconcave shape of red blood cells.
%

For an embedding $f:\Sigma \to \R^3$, we define the area and volume by
\begin{equation}
\mc A(f) \cqq \int _\Sigma 1 \dif \vol_{g_f}, \quad \mc V(f) \cqq  -\frac{1}{3} \int _\Sigma \langle f, \vec n \rangle \dif \vol_{g_f}, \label{intro:AreaandVolume}
\end{equation}
where $\vec n$ denotes the inward-pointing unit normal vector field to $f$. By the divergence theorem, $\mc V(f)$ coincides with the volume of the bounded component of $\R^3 \setminus \im (f)$. The problem is also relevant for surfaces of higher genus, see \cite{Michaletgenusone,Michaletgenustwo}. Thus, fixing a genus $g$ and a given area $A_0$ and a volume $V_0$, we ask to find an immersion $f\in \mc S_g$ minimizing the Willmore energy under the constraint on area and enclosed volume, i.e., find $f_0\in \mc S_g$ such that
\[\mc A(f_0) = A_0,\quad \mc V(f_0) = V_0, \quad \mc W(f_0) = \inf _{\substack{ f\in \mc S_g,\\ \mc A(f) = A_0,\\  \mc V(f) = V_0 }} \mc W(f).\]
Because the Willmore energy is scaling-invariant, fixing both area and volume is equivalent to fixing the \emph{isoperimetric ratio}
\[\iso(f) \cqq \frac{\mc A(f)}{\mc V(f) ^{\frac{2}{3}}}.\]
Notice that $\iso(f)$ is scaling-invariant and that by the isoperimetric inequality, it follows that
\begin{equation}
\iso(f) \geq \iso(\mb S^2) = \sqrt[3]{36\pi}\label{isoperimetric inequality}
\end{equation} 
with equality if and only if $f$ parametrizes a round sphere. We define for $R > \sqrt[3]{36\pi}$ and $g\in \N_0$
\[\beta^{\iso}_g(R) \cqq \inf _{\substack{ f\in \mc S_g,\\ \iso(f) = R}} \mc W(f).\]
\textsc{Schygulla} \cite{Schygulla} showed that for $g=0$ and $R \in [\sqrt[3]{36\pi},\infty)$, there exists a minimizer $f^{\mc I,0}_R \in \mc S_0$ of the Willmore energy under the isoperimetric constraint, i.e., $\beta_0^\iso(R ) = \mc W(f^{\mathcal I,0}_R)$ and $\iso(f^{\mathcal I,0}_R) = R$. Furthermore, he showed that the map $R \mapsto \beta_0(R)$ is continuous.  

In 2014, \textsc{Keller}, \textsc{Mondino}, and \textsc{Rivière} \cite{KMR} proved an existence result for arbitrary genus. They showed that for any $R \in (\sqrt[3]{36\pi},\infty)$ and $g\in \N_0$ such that 
\begin{align}
\beta_g^{\iso}(R) &<  8\pi,\label{intro:KMR8pi}\\
\beta_g^{\iso}(R) &< \omega_g,\label{intro:KMRomegag}\\
\beta_g^{\iso}(R) &<  \beta_g + \beta_0^{\iso}(R) - 4\pi\label{intro:KMRBauerKuwert},
\end{align}
there exists an embedding $f_R^{\iso, g}$ which minimizes the Willmore energy under the isoperimetric constraint $\iso(f_R^{\iso, g}) = R$. Thanks to the result by \textsc{Marques} and \textsc{Neves} \cite{MarquesNeves}, we have $\omega_g > 8\pi$ and so \eqref{intro:KMR8pi} implies \eqref{intro:KMRomegag}. 

The inequality \eqref{intro:KMRBauerKuwert} was proved to be true for all $R \in (\sqrt[3]{36\pi},\infty)$. Indeed, similarly to the result by \textsc{Bauer} and \textsc{Kuwert} \cite{BauerKuwert}, \textsc{Mondino} and \textsc{Scharrer} \cite{MondinoScharrerInequality} showed that for two embedded surfaces $f_1:\Sigma_1\to \R^3$, $ f_2:\Sigma_2\to \R^3$, none of which immerse round spheres, one can construct an embedding $f:(\Sigma_1 \# \Sigma_2) \to \R^3$ such that 
\begin{equation}
\mc W(f) < \mc W(f_1) + \mc W(f_2) - 4\pi \quad \text{and} \quad \iso(f)= \iso(f_2).
\end{equation}
Then, if $f_1$ is a minimizer of the unconstrained problem so that $\mc W(f_1 ) = \beta_g$ and $f_2 = f_R^{\iso, 0}$ is a minimizer of the isoperimetrically constrained problem for $g=0$ such that $\iso(f_R^{\iso, 0}) = R$, $\mc W(f_R^{\iso, 0} ) = \beta^\iso _0(R)$, \eqref{intro:KMRBauerKuwert} follows.  

Finally, the inequality \eqref{intro:KMR8pi} for $R \in (\sqrt[3]{36\pi},\infty)$ was proved by \textsc{Schygulla} \cite{Schygulla} for $g=0$, by \textsc{Scharrer} \cite{ScharrerDelaunay} for $g=1$ and by \textsc{Kusner} and \textsc{McGrath} \cite{KusnerMcGrath} for $g>1$. In this paper, we will also prove \eqref{intro:KMR8pi} in Proposition \ref{prop:8piconstruction} using a method from \textsc{Ndiaye} and \textsc{Schätzle} \cite{Ndiaye}, see Remark \ref{rem:iso constrained 8pi}.
\subsection{The total mean curvature constrained problem}
In this work, we will focus our attention on a different minimization problem which we will introduce now. As a motivation, we consider again the Helfrich energy. 
Suppose we are given an immersion $f$ and are interested in finding $c_0$ such that the Helfrich energy $\mc H^{c_0}(f)$ is minimized among $c_0\in \R$. As $\mc H^{c_0}(f)$ is a second-degree polynomial in $c_0$, this results in 
\[c_0 = \frac{\int _{\mb S^2} H\dif \vol_{g_f}}{2\mc A (f)},\quad \inf _{c_0 } \mc H^{c_0}(f) = \mc W(f) - \frac{\left (\int _{\mb S^2} H \dif \vol _{g_f}\right )^2}{4\mc A (f)}.\]
Thus, the minimal Helfrich energy is the difference of the Willmore energy and the ratio between total mean curvature and area.  We define the \emph{total mean curvature ratio} $\mc T$ as
\begin{equation}
\mc T(f) \cqq \frac{\int _{\mb S^2} H\dif \vol_{g_f}}{\sqrt{\mc A (f)}}. \label{definition T}
\end{equation}
Then, we can rewrite the infimum as
\begin{equation}
\inf _{c_0} \mc H^{c_0}(f) = \mc W(f) - \frac{1}{4}\mc T(f)^2. \label{W and T relation with inf H}
\end{equation}
Notice that $\mc T$ is scaling-invariant. In view of \eqref{W and T relation with inf H} and the previously discussed isoperimetrically constrained minimization of the Willmore energy, it is natural to ask the following question:
\begin{center}
\emph{What are the minimizers of the Willmore energy with prescribed total mean curvature ratio $\mc T$?}
\end{center}

This question is equivalent to fixing both area and total mean curvature as both $\mc W$ and $\mc T$ are scaling-invariant. Hence, under these constraints, minimizing the Willmore energy is equivalent to minimizing the Helfrich energy. We define for $g\in \N_0$ and $R \in \R$ 
\[\beta_g (R ) \cqq \inf _{\substack{f\in \mc S_g\\ \mc T(f) = R}} \mc W(f).\]
The first observation is that by \eqref{W and T relation with inf H}, it holds for $f:\Sigma \to \R^3$
\begin{equation}
\frac{1}{4} \mc T(f)^2 \leq \mc W(f), \label{CS for W-T relation}
\end{equation}
implying $\beta_g(R) \geq \frac{R^2}{4}$. In particular, in the region $R\in \R\setminus (-\sqrt{32 \pi}, \sqrt{32\pi})$, all immersions $f$ satisfying $\mc T(f) = R$ will also have a Willmore energy of at least $8\pi$. 

Let us state the main result of this paper. We will prove that the total mean curvature constrained problem admits smooth minimizers in $\mc S_g$ for a certain set of constraints and arbitrary genus. The result is actually slightly better as we do not minimize over $\mc S_g$, but the larger class of Lipschitz immersions $\mc E_\Sigma$, defined in Section \ref{subsec:WeakImmersions}.
\begin{thm}[See Theorem \ref{thm:KMR}, Corollary \ref{cor:Monotonicity}, Corollary \ref{cor:Continuity at sphere}] \label{intro:thm:KMR}
Suppose that $\Sigma$ is a closed, orientable and connected surface of genus $g$ and let $\mc E_\Sigma$ be the class of Lipschitz immersions defined in Section \ref{subsec:WeakImmersions}. Consider the set
\begin{equation}
I \cqq (0,\sqrt{2} \mc T(\mb S^2))\setminus \{\mc T(\mb S^2)\} = (0, \sqrt{32\pi})\setminus \{4\sqrt{\pi}\}.\label{intro:I}
\end{equation}
Then, for any $R\in I$ there exists a smooth embedding $f^g_R$ of $\Sigma$ into $\R^3$, with $\mc T(f^g_R) = R$ and 
\[\mc W(f_R^g) = \inf _{\substack{f \in \mc E_\Sigma \\ \mc T(f) = R}} \mc W(f)<8\pi,\]
i.e., $f_R^g$ minimizes the Willmore energy among all Lipschitz immersions of $\Sigma$ with second fundamental form bounded in $L^2$ and fixed total mean curvature constraint equal to $R$ and the minimizers $f^g_R$ are smooth embedded surfaces of genus $g$. Moreover, $\beta_g \vert _{I \cup \{\mc T(\mb S^2)\}}$ is continuous and $\beta_g(\mc T(\mb S^2)) = \beta_g$. Finally, the map $R\mapsto \beta_g(R)$ is monotonically increasing in the region $[\mc T(\mb S^2),\infty)$ and monotonically decreasing in the region $(-\infty, \mc T(\mb S^2)]$.
\end{thm}
In large parts, we follow the same ideas used to prove existence of solutions for the isoperimetrically constrained problem. In Chapter \ref{sec:Preliminaries}, we will introduce notation from differential geometry. In Chapter \ref{sec:VariationOfTandW}, we will study variations of $\mc T$ and its behavior under sphere inversions which will be needed throughout the remaining chapters. In Chapter \ref{sec:ExistenceOfMinimizers}, we will show that the method of \textsc{Keller}, \textsc{Mondino}, and \textsc{Rivière} \cite{KMR} can be adapted to minimize the Willmore energy under fixed $\mc T$ instead of $\iso$. This will be done under the additional assumptions that for $R \in I$
\begin{align}
\beta_g(R) &<  8\pi\label{intro2:KMR8pi}\\
\beta_g(R) &< \omega_g\label{intro2:KMRomegag}\\
\beta_g(R) &<  \beta_g + \beta_0(R) - 4\pi\label{intro2:KMRBauerKuwert}.
\end{align}
Compare this to \eqref{intro:KMR8pi}, \eqref{intro:KMRomegag} and \eqref{intro:KMRBauerKuwert}. As before, \eqref{intro2:KMRomegag} is implied by \eqref{intro2:KMR8pi}. In Chapter \ref{sec:BauerKuwert}, we will prove \eqref{intro2:KMRBauerKuwert} by proving the following Theorem:
\begin{thm}[See Theorem \ref{thm:connectedsum}] \label{intro:thm:connectedsum}
Suppose that $\Sigma_1$, $ \Sigma_2$ are two closed, oriented surfaces, ${f_1:\Sigma_1 \to \R^3}$ is a smooth embedding, ${f_2:\Sigma_2 \to \R^3}$ is a smooth immersion and neither $f_1$ nor $f_2$ parametrize a round sphere and $\mc T(f_2) \neq \mc T(\mb S^2)$. Denote with $\Sigma$ the connected sum $\Sigma_1 \# \Sigma_2$. Then, there exists a smooth immersion $f:\Sigma \to \R^3$ such that
\[\mc T(f) = \mc T(f_2)\]
and 
\[\mc W(f) < \mc W(f_1) + \mc W(f_2) - 4\pi.\]
\end{thm}
In Chapter \ref{sec:8piconstruction}, we will construct embedded surfaces of arbitrary genus with Willmore energy strictly below $8\pi$ such that their total mean curvature ratio converges either to $0$ or to $\sqrt{2} \mc T(\mb S^2) = \sqrt{32\pi}$. More precisely, we show
\begin{prop}[See Proposition \ref{prop:8piconstruction}] \label{intro:prop:8piconstruction}
There exist smooth genus $g$ surfaces $\Sigma^{1,g}_n$, $ \Sigma^{2,g}_n$ such that $\mc W(\Sigma^{1,g}_n)< 8\pi$, $ \mc W(\Sigma^{2,g}_n) < 8\pi$ and
\[\mc T(\Sigma^{1,g}_n) \to 0, \quad \mc T(\Sigma^{2,g}_n) \to\sqrt{2} \mc T(\mb S^2) \quad \text{as $n\to \infty$}.\]
\end{prop}
Composing these surfaces with a suitable Möbius transformation, see Lemma \ref{lem: Möbius blow down} and Lemma \ref{lem: Möbius blowup}, will prove \eqref{intro2:KMR8pi}.  

\textsc{Minkowski} \cite{Minkowski} showed that for surfaces $\Sigma$, which are the boundary of convex sets in $\R^3$, it holds that $\mc T(\Sigma)\geq \mc T(\mb S^2)$. In particular, the behavior of $\beta_0$ and consequently $\beta_g$ at $\mc T(\mb S^2)$ is non-trivial. In Chapter \ref{sec:CloseToSphere}, we show that $\beta_g$ is continuous at $\mc T(\mb S^2)$ using certain surfaces constructed in \cite{ChodoshEichmairKoerber}. We also prove the following asymptotic lower bound for $\beta_0$ around $\mc T(\mb S^2)$.
\begin{thm}[See Theorem \ref{thm:asymptotic estimate at T(S2)}] \label{intro:thm:asymptotic estimate at T(S2)}
Suppose that $\tilde{f_n}:\mb S^2\to \R^3$ are immersions such that $\mc T(\tilde{f_n})\to \mc T(\mb S^2)$. Then as $n\to \infty$
\[\mc T(\tilde{f_n})-\mc T(\mb S^2) \geq  - o(1) (\mc W(\tilde{f_n}) - 4\pi).\]
\end{thm}
Notice that consequently, the function $\beta_0$ cannot be differentiable at $\mc T(\mb S^2)$, which can also be seen directly from \eqref{CS for W-T relation}.

Finally, in Chapter \ref{sec:OtherResults}, we will answer a question posed in \cite{Dalphin2016} regarding the total mean curvature ratio for rotationally symmetric surfaces. More precisely, we will show that there exists an axisymmetric embedded surface whose total mean curvature is negative. Furthermore, we will show that under an additional assumption, the total mean curvature is non-negative, see Theorem \ref{rot sym surface}. Finally, we will prove a lower bound of $6\pi$ for the Willmore energy of axisymmetric surfaces with a negative total mean curvature, see Corollary \ref{cor:Willmore 6pi}.
\subsection{Challenges and open questions}
Let us point out some key differences between the isoperimetrically constrained problem and the total mean curvature constrained problem. The most striking difference is the isoperimetric inequality \eqref{isoperimetric inequality}. Such an inequality does not exist for $\mc T$, as we will see in Lemma \ref{lem:im T is R}. To be precise, there exist embedded surfaces with arbitrary total mean curvature ratio. In particular, the minimization problem for the isoperimetric ratio is only considered for $R\in (\sqrt[3]{36\pi},\infty)$, while in this paper, the constraint of $\mc T$ is considered for $R\in \R$. Additionally, the constraint $R = \mc T(\mb S^2)$ poses a problem for $g\geq 1$ and must be excluded in the existence of minimizers via the direct method in the calculus of variations. The reason for this is the following: When deducing the existence of minimizers of the total mean curvature constrained problem, we use a minimizing sequence $(f_n)_n$ of embeddings from a genus $g$ surface into $\R^3$ and want to use weak $W^{2,2}$-compactness to find a weak limit, see Section \ref{subsubsec:Case1}. To ensure the compactness, we need to exclude a bubbling phenomenon, i.e., a small portion of the surface is blown up to a surface which looks like a minimizer of the genus $0$ problem with the same total mean curvature ratio $R$, to which a shrinking portion of the surface carrying the genus is attached, see Section \ref{subsubsec:Case2}. Then, we can show that the Willmore energy of this surface is bounded from below by $\beta_g+\beta_0(R) - 4\pi$. In the case $R\neq \mc T(\mb S^2)$, this contradicts \eqref{intro2:KMRBauerKuwert}. However, for $R=\mc T(\mb S^2)$, \eqref{intro2:KMRBauerKuwert} turns into the false inequality $\beta_g(R) <\beta_g$. This prevents us from excluding the bubbling phenomenon and establishing compactness. 

Notice however, that since $\beta_g(\mc T(\mb S^2)) = \beta_g$ by Corollary \ref{cor:Continuity at sphere}, the existence of minimizers of $\beta_g(\mc T(\mb S^2))$ depends on whether there exist minimizers of $\beta_g$ with a total mean curvature ratio equal to $\mc T(\mb S^2)$. In the case $g=1$, it is known from \cite{MarquesNeves} that all minimizers are conformal images of the Clifford torus and so, in order to analyze the existence of minimizers for $R=\mathcal T(\mathbb S^2)$, the total mean curvature ratio of these conformal images needs to be calculated.

We show in Proposition \ref{prop:8piconstruction} that 
\begin{align}
\limsup_{R\to 0^+} \beta_g(R) &\leq 8\pi, \label{intro:beta at left end}\\
\limsup_{R\to \sqrt{2}\mc T(\mb S^2)^-} \beta_g(R) &\leq 8\pi. \label{intro:beta at right end}
\end{align}
While \eqref{CS for W-T relation} shows that equality holds in \eqref{intro:beta at right end}, it is not known whether equality also holds for \eqref{intro:beta at left end}. The first lower bound aiming towards \eqref{intro:beta at left end} is Corollary \ref{cor:Willmore 6pi} in which we show a $6\pi$ lower bound in the class of axisymmetric surfaces. Notice that for the isoperimetrically constrained problem, the analogous inequality $\lim _{R\to \infty} \beta_0^\iso	(R) = 8\pi$ was proved by \textsc{Schygulla} \cite[Theorem 1]{Schygulla}, using the fact that if we fix the area, the volume has to approach 0. Then any limit surface has a volume of 0 and must therefore have a density of at least 2 everywhere. Such an argument is however not applicable here.

While the isoperimetrically constrained problem admits smooth minimizers for all $R\in (\sqrt[3]{36\pi},\infty)$ as $\beta_g ^{\iso}(R) <8\pi$ always holds, the same cannot be said about the total mean curvature constrained problem. Consider $R = \sqrt{2}\mc T(\mb S^2)$. Then \eqref{CS for W-T relation} shows that $\beta_g(R)\geq 8\pi$ and that equality holds if and only if $H$ is constant which is impossible as \cite{Alexandrov} would imply that the surface is a round sphere. By the same argument as in \cite[Example 1.2]{RuppScharrer}\footnote{We use Lemma \ref{lem:First variation of T} instead of \cite[Lemma 2.1]{MondinoScharrerInequality}}, we can show that $\beta_0(\sqrt{2} \mc T(\mb S^2)) = 8\pi$. Thus, no smooth minimizer exists for $R=\sqrt{2}\mathcal T(\mathbb S^2)$.

Another issue that becomes apparent when trying to prove Theorem \ref{thm:connectedsum} is that $\mc T$ is a second-degree differential expression, while the isoperimetric ratio is only of first order. This worsens the asymptotic behavior of the total mean curvature of the connected sum constructed in Chapter \ref{sec:BauerKuwert}. We resolve this by gluing the inverted surface from either the inside or the outside, depending on the sign of the total mean curvature of the inverted surface.

\begin{ack}
This paper was submitted as the second authors' master thesis at the University of Bonn. The authors want to thank Stefan Müller for his interest in the work and for reviewing the thesis.
\end{ack}
\section{Preliminaries} \label{sec:Preliminaries}
\subsection{Definitions from differential geometry}
 \label{subsec:DefinitionsFromDifferentialGeometry}
In this paper, $C$ will always stand for an arbitrary constant which may change from line to line. We denote by $\dom f$ and $\im f$ the domain and image of a function $f$.

Let us first recall some definitions from differential geometry. Suppose that $(\Sigma, \mc A)$ is a 2-dimensional smooth manifold and $f: \Sigma \to \R^3$ an immersion. Consider a local chart $x\in \mc A$. Given an open subset $U\subset \Sigma$ with $U\subset \dom(x)$, we define for $\phi\in C^1(U,\R^n)$ and a point $p\in U$ the \emph{differential of $\phi$ at $p$ in direction $i$} by
\[\parpar{x^i}{\phi}(p)\cqq (D_i(\phi\circ x^{-1})\circ x)(p).\]
We also write $\parpar{i}{\phi}$ instead of $\parpar{x^i}{\phi}$.
We denote by $g$ or $g_f$ the \emph{first fundamental form}, i.e., the pullback of the canonical, flat metric $\langle \cdot, \cdot \rangle$ on $\R^3$ by $f$. In a local chart $x$, this corresponds to the $2\times 2$ matrix
\[g_{ij}(p) = \langle \parpar{x^i}{f}(p),\parpar{x^j}{f}(p)\rangle.\]
We usually suppress the dependency on the chart. The coefficients of the inverse matrix of $g$ at $p$ are denoted by $g^{ij}(p)$. We define the \emph{gradient} of a function $\phi\in C^1(\Sigma, \R)$ in a local chart $x$ by $\nabla \phi \cqq g ^{ij} \parpar{x^i}{\phi} \parpar{x^j}{f}$. Here we use the Einstein convention, i.e., whenever an index appears twice in the same term, we sum over all possible values for this index. The \emph{divergence} of a function $X\in C^1(\Sigma, \R^3)$ is defined by $\Div X \cqq g^{ij}\langle\parpar{x^i}{X}, \parpar{x^j}{f}\rangle$.

We suppose now that $\Sigma$ is an oriented surface and $f:\Sigma \to \R^3$ is again an immersion. We define the \emph{Gauss map} $\vec n$ in any local, oriented chart $x$ by
\begin{equation}
\vec n\cqq \frac{\parpar{x^1}{f}\times \parpar{x^2}{f}}{|\parpar{x^1}{f}\times \parpar{x^2}{f}|}.\label{normal}
\end{equation}
The map $\vec n$ takes values in $\mb S^2$. If $f$ is an embedding, then the orientation is chosen such that $\vec n$ represents the inner normal vector field. We define the \emph{second fundamental form of $f$} in an oriented chart $x$ by the $2\times 2$ matrix
\begin{equation}
\mb I_{ij}\cqq \langle \partial_i \partial_j f, \vec n\rangle.\label{second fundamental form}
\end{equation}
Also notice that due to the fact that $\langle\partial_j f, \vec n\rangle=0$, the product rule implies
\begin{equation}
\mb I_{ij} = - \langle \partial _i f, \partial_j \vec n\rangle .  \label{second fundamental form less regularity}
\end{equation}
The \emph{mean curvature vector} $\vec H$ and the \emph{scalar mean curvature} $H$ are defined by
\begin{equation}
\begin{split}
\vec H&\cqq g^{ij} \mb I_{ij}\vec n=-g^{ij} \langle \partial _i f, \partial_j \vec n\rangle \vec n,\\
 \quad H &\cqq \langle\vec H,\vec n\rangle.
\end{split} \label{H}
\end{equation}
Furthermore, the \emph{Gauss curvature} is defined as 
\begin{equation}
K \cqq \det((g^{ik} \mb I_{kj})_{ij}).\label{K}
\end{equation}
The canonical \emph{volume form} corresponding to the first fundamental form $g$ is the 2-form $\vol_g$, which in a local chart is defined as
\[\vol_g \cqq \sqrt{\det (g_{ij})_{ij}} \dif x^1 \land \dif x^2.\]
If the volume form is induced by an immersion $f$, we also use the notation $\vol_f$. The \emph{Willmore energy} $\mc W(f)$ of the immersion $f$ is defined as
\begin{equation}
\mc W(f) \cqq \frac{1}{4} \int _{\Sigma} H^2 \dif \vol_g.\label{Willmore energy}
\end{equation}
When $\Sigma$ denotes the immersed surface itself, we may also write $\mc W(\Sigma)$. The \emph{area} of an immersion is defined by $\mc A (f) \cqq \int_\Sigma 1 \dif \vol_g$ and the \emph{total mean curvature ratio} is defined for embedded surfaces  as
\begin{equation}
\mc T(f) \cqq \frac{\int_\Sigma H \dif \vol_g }{\sqrt{\mc A (f)}}.\label{mean curvature ratio}
\end{equation}
Finally, we also define the \emph{volume} of an embedded surface as
\begin{equation}
\mc V(f) \cqq -\frac{1}{3} \int \langle f, \vec n\rangle \dif \vol_g.\label{volume}
\end{equation}
For non-embedded immersions $f$, we set $\mc T(f) = \mc V(f) = \infty$, as there is no canonical choice of an inner normal vector field. For an embedded surface $\Sigma \subset \R^3$, we also use the notation  $\mc T(\Sigma)$, $\mc V(\Sigma)$.
\subsection{The constrained minimization problem}\label{subsec:IntroductionToConstrainedProblem}
Let us now introduce the main minimization problem we want to investigate in this work. Consider a fixed oriented closed and connected surface $\Sigma$ of genus $g$. We label the class of smooth embeddings from $\Sigma$ into $\R^3$ by $\mc S_g$, i.e.,
\[\mc S_g\cqq \{f \in C^{\infty}(\Sigma, \R^3), \; \text{$f$ is an immersion}\}. \]
Fix $R \in \R$. We define 
\begin{equation}
\beta_g (R) \cqq \inf _{\substack{f\in \mc S_g\\ \mc T(f) = R}} \mc W(f).\label{beta g tau for smooth surfaces}
\end{equation}
We also set 
\begin{equation}
\beta_g \cqq \inf _{f\in \mc S_g} \mc W(f),\label{beta g for smooth surfaces}
\end{equation}
the infimum of the Willmore energy among all smooth embeddings from $\Sigma$ into $\R^3$. This is independent of the choice of the genus $g$ surface $\Sigma$ by the classification of closed surfaces. It was proved by \textsc{Simon} \cite{Simon} that this problem admits smooth minimizers under certain constraints, see Section \ref{subsec:IntroductionUnconstrainedProblem}.  

The reason why we work with embeddings instead of immersions is to ensure that there is a natural choice of an inner normal field. Notice however that if $\beta_g(R)<8\pi$, this minimization problem is the same as minimizing over smooth \emph{immersed} surfaces by the Li-Yau inequality \eqref{intro:LiYau}, which implies that for any non-embedded immersion $f$, we have
\[\mc W(f)\geq 8\pi.\]
First, we need to show that the minimization problem is well-posed, i.e., that for any $R \in \R$ and $g\in \N_0$, there exists an embedding in $\mc S_g$ whose total mean curvature ratio equals $R$. This is done by the following Lemma, which is a slight adaption of the argument from \cite[Theorem 1.2]{Dalphin2016}.
\begin{lemma} \label{lem:im T is R}
For every $g\in \N_0$ and $R\in \R$, there exists a smooth embedded genus $g$ surface $\Sigma\subset \R^3$ with $\mc T(\Sigma) = R$.
\end{lemma}
\begin{proof}
Let us first fix $u\in C_c^\infty ((0,1)^2,\R)$ whose graph $\{(a,u(a)), \; a\in (0,1)^2\}$ has nonzero total mean curvature. We periodically extend $u$ to $\R^2$ and define $\Gamma_u:\R^2 \to \R^3$, $\Gamma_u(x)\cqq (x, u(x))$. For $n\in \N$ and $t\in [-1,1]$, we define $u_{n,t}:(0,1)^2\to \R$, $u_{n,t}(x) \cqq \frac{t}{n}u(nx)$. Then $\Gamma_{u_{n,1}}|_{(0,1)^2}$ consists of $n^2$ rescaled copies of $\Gamma_u|_{(0,1)^2}$, each contributing $1/n$ times the total mean curvature, so that
\begin{equation}
\int _{(0,1)^2} H_{\Gamma_{u_{n,1}} }\dif \vol_{\Gamma_{u_{n,1}}} = n \int _{(0,1)^2} H_{\Gamma_u} \dif \vol_{\Gamma_{u}}\neq 0.\label{periodically extend bump}
\end{equation}
Notice that 
\begin{equation}
\int _{(0,1)^2} H_{\Gamma_{u_{n,1}}} \dif \vol_{\Gamma_{u_{n,1}}} =- \int _{(0,1)^2} H_{\Gamma_{u_{n,-1}}} \dif \vol_{\Gamma_{u_{n,-1}}}.\label{inverted bump}
\end{equation}
Furthermore, for $t\in [-1,1]$
\begin{equation}
\mc A(\Gamma_{u_{n,t}})\leq \mc A(\Gamma_{u_{n,1}}) = \mc A(\Gamma_{u}) = \mc A (\Gamma_{u_{n,-1}})\label{area remains the same}
\end{equation}
and $\int _{(0,1)^2} H_{\Gamma_{u_{n,t}}} \dif \vol_{\Gamma_{u_{n,t}}}$ varies continuously in $t$. 
Consider now a smooth genus $g$ surface $\Sigma\subset \R^3$ containing a flat patch $D$, where we may assume that $D = (0,1)^2 \times \{0\}$. Now, we simply glue the graph of $\Gamma_{u_{n,t}}|_{(0,1)^2}$ to $\Sigma$ instead of $D$ to obtain a surface $\Sigma_{n,t}$. Because the total mean curvature and area of $\Sigma \setminus D$ are fixed, we obtain by \eqref{periodically extend bump}, \eqref{inverted bump}, \eqref{area remains the same} and the intermediate value theorem that
\[\bigcup _{n\in \N} \bigcup _{t\in [-1,1]} \mc T(\Sigma_{n,t}) = \R. \]
Notice that for $n$ sufficiently large, the immersion is already an embedding.
\end{proof}
\section{Monotonicity of \texorpdfstring{$\beta_g$}{} and variations of \texorpdfstring{$\mc T$}{}}\label{sec:VariationOfTandW}

\subsection{Monotonicity of \texorpdfstring{$\beta_g$}{}}\label{subsec:Monotonicity} In this section, we want to establish a monotonicity property for the map $R\mapsto \beta_g (R)$ given a fixed $g\in \N_0$. Crucial in our procedure will be the invariance of the Willmore functional under composition of conformal maps in the ambient space $\R^3$. Let us start by defining the notion of a conformal map.
\begin{defi}\label{Conformality}
Suppose that $(M,g)$ and $(N,k)$ are two Riemannian manifolds. A smooth map $f:M\to N$ is called \emph{conformal} if it holds that
\[ e^{2\lambda}g = f^ * k,\]
where $f ^* k $ denotes the pull back metric of $k$ under $f$ and $\lambda$ is a smooth map from $\Sigma$ to $\R$. In other words for all $p\in \Sigma$, it holds 
\[e^{2\lambda}	g(v,w) =  k(df(v),df(w)) \quad \forall v,w\in T_p\Sigma.\]
$\lambda$ is called the \emph{conformal factor of $f$}.
\end{defi}
\begin{prop}[\cite{Chen1974}] \label{prop: Willmore invariance} 
Suppose that $\Sigma$ is a closed surface, $f:\Sigma \to \R^n$ is a smooth immersion. Let $U\subset \R^n$ be an open subset containing $f(\Sigma)$ and $\psi:U\to \R^n$ be a conformal map. Then, it holds that
\[\mc W(f) = \mc W(\psi \circ f).\]
\end{prop}
We will apply Proposition \ref{prop: Willmore invariance} to the sphere inversion $I_{x_0}: \R^3\setminus \{x_0\} \to \R^3$ centered at a point $x_0\in \R^3$, defined by 
\begin{equation}
I_{x_0}(x) \cqq \frac{x-x_0}{|x-x_0|^2}.\label{sphere inversion}
\end{equation}
It is easy to check that these are conformal maps. We will exploit the fact that sphere inversions in general do not fix the total mean curvature ratio. In fact, using a blow-down procedure in Lemma \ref{lem: Möbius blow down} and a blow-up procedure in Lemma \ref{lem: Möbius blowup}, we will show that for a given immersion $f$	
\[\bigcup _{x_0 \in \R^3 \setminus \im f} \mc T(I_{x_0} \circ f) \supset (\min\{\mc T(f), \mc T(\mb S^2)\}, \max\{\mc T(f), \mc T(\mb S^2)\}).\]
From here, the monotonicity of $\beta_g$ is an easy consequence, see Corollary \ref{cor:Monotonicity}. 

The following lemma is an adaption of \cite[Lemma 3.7]{ScharrerPhD} and we will follow the proof closely.
\begin{lemma} \label{lem: Möbius blow down}
Suppose that $f:\Sigma \to \R^3$ is a smooth embedding. For $I_a$ defined as in \eqref{sphere inversion}, it holds that
\[\lim _{|a|\to \infty} \mc T(I_a \circ f  ) = \mc T(f ).\]
\end{lemma}
\begin{proof}
Denote $f^a = I_a \circ f$, $ f^a _i = \partial_i f^a$. Then
\begin{equation}
DI_a(x) = \frac{1}{|x-a|^2} \left (\id - 2 \frac{(x-a) \otimes (x-a)}{|x-a|^2}\right )\label{DI a x}
\end{equation}
and
\[f^a_i = \frac{1}{|f-a|^2} \left (\partial_i f - 2(f-a) \frac{\langle f-a , \partial _i f\rangle }{|f-a|^2}\right ).\]
Furthermore, denoting by $g^a $ the first fundamental form of $f^a$, we have
\begin{equation}
g_{ij}^a = \frac{1}{|f-a|^4} g_{ij},\quad (g ^a)^{ij} = |f-a|^4 g^{ij}.\label{metric tensor g of a}
\end{equation}
In particular 
\begin{equation}
\dif\vol_{g^a} = \frac{1}{|f-a|^4} \dif\vol_g. \label{vol of I a}
\end{equation} As $f$ is bounded, it holds 
\begin{equation}
\frac{1}{|f-a|^p} = \frac{1}{|a|^p} + \mc O\left (\frac{1}{|a|^{p+1}}\right ) \quad \text{as} \quad |a|\to \infty\label{f-a asymptotic}
\end{equation}
 for $p\geq 1$ by \cite[(3.4)]{ScharrerPhD}. We calculate $\partial_i \partial_j f^a$.
\begingroup
\allowdisplaybreaks
\begin{align}
\partial_i \partial_j f^a &= 
\partial_i \left ( \frac{1}{|f-a|^2} \left ( \partial_j f - 2 \frac{(f-a)}{|f-a|^2} \langle f-a, \partial_j f\rangle \right )\right )
\notag\\
&=\frac{1}{|f-a|^2} \Bigg [\partial_i \partial_j f - 2(f-a) \frac{\langle f-a,  \partial_i \partial_j f\rangle }{|f-a|^2}\notag\\
&\quad  - \frac{2}{|f-a|^2} \Bigg ( \langle f-a, \partial_i f\rangle  \partial_j f + \langle f-a, \partial_j f\rangle  \partial_i f\Bigg )\notag \\
   &\quad + 8 \frac{f-a}{|f-a|^4} \langle f-a, \partial_i f\rangle\langle f-a, \partial_j f\rangle - 2 \frac{f-a}{|f-a|^2} \langle \partial_i f, \partial_j f\rangle \Bigg ]
\notag\\ 
&=DI_a(f) (\partial_i \partial_j f) + \mc O\left (\frac{1}{|f-a|^3}\right  ) \label{explicit d_id_j f for later}
\\
&=DI_a(f)( \partial_i \partial_j f) + \mc O\left (\frac{1}{|a|^3}\right )\label{explicit d_id_j f}
\end{align}
\endgroup
as $|a|\to \infty$, where the last equality follows by \eqref{f-a asymptotic}. The normal vector of the immersion $f^a$ respectively $f$ is given by
\[-\vec n_a =\frac{f_1 ^a \times f_2 ^a}{|f_1 ^a \times f_2^a|}, \quad \vec n = \frac{f_1 \times f_2}{|f_1 \times f_2|}.\]
Here we assume $\vec n$ is an inward-pointing normal vector field. The minus sign for $\vec n_a$ is to ensure that $\vec n_a$ is still inward-pointing as $I_a$ changes orientation. Then, by \cite[Lemma 3.7]{ScharrerPhD}
\[\langle f_1^a \times f^a _2,f-a\rangle =\frac{1}{|f-a|^4} \langle f_1\times f_2,f-a\rangle \quad \text{and}\quad |f_1 ^a\times f_2 ^a| = \frac{|f_1 \times f_2|}{|f-a|^4}.\]
Thus,
\[ \langle \vec n_a ,f-a\rangle  =  -\frac{|f-a|^4}{|f_1 \times f_2|}\langle f_1 ^a \times f_2 ^a,f-a\rangle  =- \langle \vec n , f-a\rangle. \]
By \eqref{DI a x}, it follows that
\[\langle DI_a (f)( \vec n),f-a\rangle  = -\frac{1}{|f-a|^2} \langle f-a,  \vec n\rangle  = \left \langle \frac{1}{|f-a|^2} \vec n_a ,f-a\right \rangle.\]
Since $I_a$ is conformal, we have $\langle DI_a (f) (\vec n), f_i ^a \rangle = 0 =\left \langle \frac{1}{|f-a|^2} \vec n_a , f_i ^a\right \rangle  $, $i=1,2$. If it holds $\langle f-a,\vec n\rangle \neq 0$, then $\{f_{1}^a, f_2 ^a, f-a\}$ form a basis of $\R^3$ and thus 
\begin{equation}
DI_a (f) (\vec n) = \frac{1}{|f-a|^2} \vec n_a.\label{Normal vector}
\end{equation}
If $\langle f-a , \vec n \rangle = 0$, \eqref{Normal vector} follows by a simple approximation of $a$ by a sequence $(a_n)_n$ such that $\langle f-a_n, \vec n \rangle \neq 0$. \eqref{metric tensor g of a}, \eqref{explicit d_id_j f} and \eqref{Normal vector} now imply
\begin{align}
H_a &= \langle (g^a) ^{ij}\partial_i \partial_j f^a  ,  \vec n_a\rangle  \notag\\
&= |f-a|^6 g ^{ij}\left \langle DI_a(f) ( \partial_i \partial_j f) + \mc O\left (\frac{1}{|a|^3}\right ) , DI_a (f)(\vec n)\right \rangle   \notag\\
&=|f-a|^2 g ^{ij}\langle  \partial_i \partial_j f   , \vec n\rangle  + \mc O\left (|a|\right )\notag\\
&= |f-a|^2 H + \mc O \left( |a|\right )\label{explicit H}
\end{align}
as $|a|\to \infty$.
Integrating \eqref{explicit H} and using \eqref{vol of I a} yields
\begin{align*}
 \int _{\Sigma}H_a \dif \vol_{g^a} &= \int _{\Sigma} \frac{1}{|f-a|^2} H + \mc O\left ( \frac{1}{|a|^3}\right ) \dif \vol_{g} \\
&= \frac{1}{|a|^2} \int _{\Sigma} H \dif \vol_{g} + \mc O\left ( \frac{1}{|a|^3}\right )\quad \text{ as $|a|\to \infty$}.
\end{align*}
The asymptotic behavior of the area was calculated in \cite[Lemma 3.7]{ScharrerPhD}, such that
\[\sqrt{\mc A  (f^a)} = \frac{1}{|a|^2}\sqrt{ \mc A  (f)} + \mc O\left ( \frac{1}{|a|^3}\right )\quad \text{ as $|a|\to \infty$}.\]
It follows that
\[\mc T(f^a) = \frac{\int _{\Sigma}H_a \dif \vol_{g^a}}{\sqrt{\mc A  (f_a)}} = \frac{\frac{1}{|a|^2} \int _{\Sigma} H \dif \vol_{g} + \mc O\left ( \frac{1}{|a|^3}\right )}{\frac{1}{|a|^2}\sqrt{ \mc A  (f)} + \mc O\left ( \frac{1}{|a|^3}\right )} = \frac{ \int _{\Sigma} H \dif \vol_{g} + \mc O\left ( \frac{1}{|a|}\right )}{\sqrt{ \mc A  (f)} + \mc O\left ( \frac{1}{|a|}\right )} \to \mc T(f) \]
as $|a|\to \infty$.
\end{proof}
After having seen the behavior of $\mc T$ under a blow-down when $a$ diverges to infinity, we want to investigate the behavior of the total mean curvature ratio as the center of the sphere inversion approaches the surface itself. Intuitively, the region close to the center is blown up by the inversion which maps planes to spheres. As the immersion $f$ locally looks like a plane, one should expect that the blown-up region resembles a sphere. Thus it is reasonable that the resulting surface has a mean curvature constraint close to $\mc T(\mb S^2)$. This is exactly what the following Lemma states. This lemma is an adaption of \cite[Lemma 3.8]{ScharrerPhD}.
\begin{lemma} \label{lem: Möbius blowup}
Suppose that $f:\Sigma \to \R^3$ is a smooth embedding. Let $p\in  \im f$ and $\vec n$ be defined as in \eqref{normal}. Set ${\gamma:[0,\infty) \to \R^3, \gamma(t) = p - t \vec n(p)}$. Then, 
\[\lim _{t\to 0} \mc T(I_{\gamma(t)}\circ f) = \mc T(\mb S^2)=4\sqrt{\pi}.\]
\end{lemma}
\begin{proof}
After a translation and a rotation, we may assume without loss of generality that $p=0$ and that $f$ has a local graph representation
\[f:D_R \cqq \{z\in \R^2,\; |z|<R\} \to \R^3, \quad f(z) = (z,u(z)),\quad \vec n = \frac{1}{\sqrt{|\nabla u|^2+1}} \Vector{\nabla u, -1}.\]
Here, $u:D_R\to \R$ is smooth and $u(0)=0$, $ Du(0)=0$, and thus $|u(z)|\leq C|z|^2$ for some $C\in\R$. Then, $\gamma(t) = (0,0,t)$. Let $\alpha \in (1/2,1)$ and set $f^t = I_{\gamma(t)}\circ f$. We have
\[\partial_1 f = (1,0,\partial_1 u),\quad \partial_2 f = (0,1,\partial_2 u).\]
We consider $z\in D_{t^\alpha}$ as $t\to 0$. Then, 
\begin{align}
|z| \in \mc O(t^\alpha),\quad  u(z) \in \mc O(t^{2\alpha})\quad \text{and} \quad \parpar{i}{u} \in \mc O(t^{\alpha}). \label{z and u(z) growth}
\end{align}
 We first calculate $\partial_i \partial_j (f^t)$, using \eqref{DI a x}:
\begin{align*}
 \partial_i \partial_j (f^t)(z)&=\frac{1}{|z|^2 + (u(z)-t)^2} \Bigg [\Vector{0,\partial_i \partial_j u}  - 2 \Vector{z,u(z)-t} \frac{(u(z)-t) \partial_i \partial_j u}{|z|^2 + (u(z)-t)^2}\\
&\quad  - \frac{2}{|z|^2 + (u(z)-t)^2} \Bigg ( (z_i + \partial_i u (u(z)-t)) \partial_j f  + (z_j + \partial_j u (u(z)-t)) \partial_i f\Bigg ) \\
   &\quad + 8 \frac{\Vector{z,u(z)-t}}{(|z|^2 + (u(z)-t)^2)^2} (z_i + \partial_i u (u(z)-t))(z_j + \partial_j u (u(z)-t)) \\
&\quad   - 2 \frac{\Vector{z,u(z)-t}}{|z|^2 + (u(z)-t)^2} \langle \partial_i f , \partial_j f\rangle \Bigg ].
\end{align*}
We denote by $g^{\gamma(t)}$ the first fundamental form. Multiplying the previous equation with $(g^{\gamma(t)})^{ij} = (|z|^2+(u(z)-t)^2)^2 g^{ij}$, see \eqref{metric tensor g of a}, results in
\begin{align*}
(g^{\gamma(t)})^{ij} \partial_i \partial_j (f^t)(z)&= \Bigg [(|z|^2+(u(z)-t)^2)\Vector{0,\partial_i \partial_j u}  - 2 \Vector{z,u(z)-t} (u(z)-t) \partial_i \partial_j u\\
&\quad  - 2 \Bigg ( (z_i + \partial_i u (u(z)-t)) \partial_jf  + (z_j + \partial_j u (u(z)-t)) \partial_if\Bigg ) \\
   &\quad + 8 \Vector{z,u(z)-t} \frac{(z_i + \partial_i u (u(z)-t))(z_j + \partial_j u (u(z)-t))}{|z|^2+(u(z)-t)^2} \\
&\quad   - 2 \Vector{z,u(z)-t} \langle \partial_if, \partial_jf\rangle  \Bigg ]g^{ij}
\\
&=-4 \Vector{z,0} + 4 \Vector{z,u(z)-t} \frac{|z|^2-t^2}{|z|^2+t^2}+ \mc O(t^{2\alpha}).
\end{align*}
The last equality follows by \eqref{z and u(z) growth} and careful consideration of each term. Now, the inner normal vector $\vec n_{\gamma(t)}$ of $f^t$ at $I_{\gamma(t)}(f(p))$ is due to \eqref{Normal vector} given by
\begin{align*}
\vec n_{\gamma(t)} = \vec n - 2\Vector{z,u(z)-t} \frac{ \left \langle \Vector{z,u(z)-t}, \Vector{\nabla u,-1}\right \rangle }{\sqrt{|\nabla u|^2+1} (|z|^2+(u(z)-t)^2)}.
\end{align*}
It follows that 
\begin{align*}
H_{\gamma(t)}  &=-4 \langle z, \nabla u\rangle  + 8 \frac{|z|^2}{ |z|^2 +(u(z)-t)^2}(\langle z, \nabla u\rangle  - (u(z)-t)) \\
&\quad + 4\frac{|z|^2-t^2}{|z|^2+t^2}\Bigg ( \left \langle \Vector{z,u(z)-t} , \vec n\right \rangle - 2\left \langle \Vector{z,u(z)-t}, \vec n\right \rangle \Bigg) + \mc O(t^{2\alpha})\\
&=4t + \mc O(t^{2\alpha}).
\end{align*}
From \cite[(3.10)]{ScharrerPhD}, we obtain the asymptotic behavior 
\begin{equation}
\int _{D_{t^\alpha}} H_{\gamma(t)} \dif \vol_{g^{\gamma(t)}} =  (1+ o(1))\frac{\pi}{t^2} (4t+\mc O(t^{2\alpha})) = \frac{1}{t}(4\pi +o(1)).\label{integral H on disk}
\end{equation}
Furthermore, in the region $D_R\setminus D_{t^\alpha}$, we obtain similarly to \eqref{explicit d_id_j f for later} and \eqref{explicit H} that 
\[H_{\gamma(t)} = |f-\gamma(t)|^2 H + \mc O(|f-\gamma(t)|) \quad \text{as $t\to 0$.}\]
Since $Du$ and $H$ are bounded in $D_R$ and $z^2 + (u(z)-t)^2 = |f-\gamma(t)|^2 \geq z^2$, this implies together with \cite[(3.7)]{ScharrerPhD} that
\begin{equation}
\begin{split}
&\quad \,\int _{D_R \setminus D_{t^{\alpha}}} H_{\gamma(t)} \dif \vol _{g^{\gamma(t)}}\\
&= \left |\int _{D_R\setminus D_{t^\alpha}} \frac{\sqrt{1+|\nabla u|^2}}{(z^2+(u(z) -t)^2)^2} \bigg(|f-\gamma(t)|^2 H + \mc O(|f-\gamma(t)|)\bigg) \dif z\right | \\
&\leq C \int _{t^\alpha}^R \frac{1}{x} + \frac{1}{x^2} \dif x \\
&\leq C + \mc O(t^{-\alpha}).
\end{split}
\label{integral H on annulus}
\end{equation}
Finally, on the remaining surface $I_{\gamma(t)}(f(\Sigma)\setminus f(D_R))$, both area and $H$ are uniformly bounded in $t$. Thus, \eqref{integral H on disk} and \eqref{integral H on annulus} imply
\[\int _{\Sigma} H_{\gamma(t)} \dif \vol _{g^{\gamma(t)}} = \frac{1}{t} ( 4\pi + o(1)). \]
Using also \cite[(3.10)]{ScharrerPhD} to estimate the area of $f^t$, we finally obtain
\[\frac{\int _{\Sigma} H_{\gamma(t)} \dif \vol_{g^{\gamma(t)}} }{\sqrt{\mc A  (f^t)}} = \frac{ \frac{1}{t} (4\pi+o(1))}{\frac{1}{t}(\sqrt{\pi} + o(1))} \to 4\sqrt{\pi} =\mc T(\mb S^2)\quad \text{as $t\to 0$.}\qedhere\]

\end{proof}
Having established these two lemmas, the following is now an easy consequence:
\begin{cor}[Monotonicity of $\beta_g$] \label{cor:Monotonicity}
The map $R\mapsto \beta_g(R)$ is monotonically increasing in the region $(\mc T(\mb S^2),\infty)$ and monotonically decreasing in the region $(-\infty, \mc T(\mb S^2))$.
\end{cor}
\begin{proof}
We consider $R_1,R_2\in \R$ such that $\mc T(\mb S^2)<R_1 <R_2$. Consider $f:\Sigma \to \R^3$ an embedding of a genus $g$ surface with $\mc T(f) = R_2$. Choose a continuous path $\gamma:[0,\infty)\to \R^3$ with $\gamma(0) = p \in \im f$, $\gamma\vert _{(0,\infty)}\cap \im f = \emptyset$, $\gamma(t) = p- t \vec n(p)$ for $t<\eps$ for some sufficiently small $\eps > 0$ such that $|\gamma(t)|\to \infty$ as $t\to \infty$. Because $\mc T(I_{\gamma(t)}\circ f) $ depends continuously on $t$ for $t>0$, using Lemma \ref{lem: Möbius blow down} and Lemma \ref{lem: Möbius blowup} and the intermediate value theorem shows the existence of some $t_0>0$ such that $\mc T(I_{\gamma(t_0)}\circ f) = R_1$. Thus, $\beta_g(R_1)\leq \mc W(f)$ and by taking the infimum over all immersions $f$ with $\mc T(f) = R_2$, we obtain $\beta_g(R_1)\leq \beta_g(R_2)$. The proof for the other case is similar.
\end{proof}
\subsection{Variations of \texorpdfstring{$\mc T$}{}}\label{subsec:VariationsOfT}
So far, we have considered global variations of our embeddings by using sphere inversions to preserve the Willmore energy. Now we want to consider local variations, i.e., variations $f_t$ of $f$ supported in a neighborhood of some point $p\in \Sigma$ such that 
\[\dv{}{t}\bigg\vert _{t=0} \mc T(f_t) \neq 0.\]
We begin by stating the following proposition which will be useful when calculating the first variation of $\mc T$.
\begin{prop}\label{prop: Alexandrov-corollary}
Suppose $f:\Sigma \to \R^3$ is a smoothly immersed, closed and connected surface. Then
\begin{equation}
4K(p) \mc A  (f) = H(p) \int_\Sigma H \dif \vol_g \quad \forall p\in \Sigma\label{H proportional to K}
\end{equation}
holds if and only if $f$ immerses a round sphere. 
\end{prop}
\begin{proof}
If $f$ immerses a round sphere of radius $r$, then $K=\frac{1}{r^2}$ and $H = \frac{2}{r}$, so that the converse implication holds. We prove the first implication.

As both sides of \eqref{H proportional to K} are scaling-invariant, we may assume that $\mc A (f) = 4\pi$. If $\int _{\Sigma} H \dif \vol_g=0$, then $K(p) = 0$ everywhere, which is a contradiction to the fact that $\int_\Sigma \max \{K,0\} \dif \vol_g \geq 4\pi$, see \cite[Lemma 7.2.1]{Willmore1993}. From the Gauss-Bonnet theorem, we know that
\[4\pi(1-g) = \int_\Sigma K \dif \vol_g = \frac{\left (\int_\Sigma H\dif \vol_g\right ) ^2}{16\pi }.\]
Since the right side is positive, we must have $g=0$. 
It follows up to a possible sign change that $\int_\Sigma H \dif \vol _g= 8\pi$ and so $2K = H$. Denoting with $\kappa_1, \kappa_2$ the principal curvatures, this means $2\kappa_1 \kappa_2 = \kappa_1 + \kappa_2$. Hence
\begin{equation}
(2\kappa_1\kappa_2 -1)^2 = (\kappa_1 + \kappa_2 -1)^2 = \kappa_1^2 + \kappa_2^2 + 1 + 2\kappa_1 \kappa_2 - 2\kappa_1 - 2\kappa_2 = (\kappa_1 - \kappa_2)^2 + 1 \geq 1 \label{kappa inequality}
\end{equation}
for all $p\in \Sigma$. Since there exist points $p$ such that $\kappa_1(p)\kappa_2(p)>0$ and since $\Sigma$ is connected, it follows that $\kappa_1(p) \kappa_2(p) \geq 1$ for all $p\in \Sigma$. 
By the assumption that $\mc A (f) = 4\pi = \int_\Sigma K\dif \vol_g$, we see that $K(p) = 1$ for all $p\in \Sigma$ so equality must hold in \eqref{kappa inequality}. Hence every point is umbilic and $f$ must immerse a round sphere.
\end{proof}
We are now ready to prove the following lemma. This lemma is a variation of \cite[Lemma 3.10]{ScharrerPhD}.
\begin{lemma}[First variation of $\mc T$]\label{lem:First variation of T}
Suppose $f:\Sigma \to \R^3$ is a smoothly immersed, closed and connected surface such that $f$ does not immerse a round sphere. Suppose $q\in \Sigma$ is arbitrary. Then, there exists $p \in \Sigma$, $p\neq q,$ with the following property.\newline
For each neighborhood $U$ of $p$, there exists a smooth normal vector field $\xi : \Sigma \to \R^3$ compactly supported in $U$ such that for $f_t \cqq f + t\xi$ with $t\in \R$, the function $t\mapsto \mc T(f_t)$ is differentiable at $t=0$, and
\[\dv{}{t}\bigg\vert _{t=0} \mc T(f_t) \neq 0.\]
Moreover,
\[\mc W(f_t) = \mc W(f) + \mc O(t) \quad \text{as $t\to 0.$}\]

\end{lemma}
\begin{proof}
First, we calculate the variation of $\mc T$. Let $\mb I$ be the second fundamental form of $f$ as defined in \eqref{second fundamental form}. We denote $g^t$ the metric tensor of $f_t$. Then, by \cite[Lemma 2.3]{Rupp2023volume}, we know that
\begin{align}
\partial_t (\dif\vol_{g^t}) &= -H \xi \dif\vol_{g},\label{variation volume form}\\ 
\partial_t H &= \Delta \xi + |\mb I|^2 \xi. \label{variation H}
\end{align}
Thus, we have
\[\partial_t \bigg \vert _{t=0} \int _\Sigma H_t \dif \vol_{g^t}  = \int_\Sigma \Delta \xi  + |\mb I|^2 \xi \dif \vol_g  -\int_\Sigma H^2 \xi \dif \vol_g.\]
By the divergence theorem, it holds that $\int_\Sigma \Delta \xi \dif \vol_g=0$. As $|\mb I|^2 = H^2-2K$, we obtain
\[\partial_t \bigg \vert _{t=0} \int _\Sigma H_t \dif \vol_{g^t}  = -2\int_\Sigma K \xi \dif \vol_g .\]
On the other hand,
\[\partial_t \bigg \vert _{t=0}\mc A(f_t) = -\int _\Sigma H\xi \dif \vol_g.\]
Thus, the derivative of $\mc T(f_t)$ is by the quotient rule
\begin{align}
\partial_t \bigg \vert _{t=0} \mc T(f_t) &= \frac{-2\int_\Sigma K\xi \dif \vol_g \sqrt{\mc A (f)} + \int_\Sigma H \dif \vol_g \frac{\int _\Sigma H \xi \dif  \vol_g }{2\sqrt{\mc A (f)}}}{\mc A(f)}\notag\\
&= \frac{\int_\Sigma \left (-4K\mc A (f)+   \int_\Sigma H \dif \vol_g  H\right )\xi \dif \vol_g  }{2\mc A (f)^{3/2}} \label{T variation}\\
&= \langle \nabla \mc T(f), \xi\rangle _{L^2(\dif \vol_g)} \label{T gradient},
\end{align}
where $\nabla \mc T(f) \cqq \frac{\mc T(f) H}{2\mc A(f)} - \frac{2K}{\sqrt{\mc A(f)}}$ is the $L^2(\dif \vol_g)$-gradient of $\mc T$. 
By Proposition \ref{prop: Alexandrov-corollary}, we find a point $p\neq q \in \Sigma$ such that
\begin{equation}
4K(p)  \mc A (f) \neq H(p) \int _\Sigma H\dif \vol_g . \label{K-H relation}
\end{equation}
By continuity, there exists a neighborhood $V$ of $p$ such that \eqref{K-H relation} holds in $V$. For an open neighborhood $U$ of $p$, define $\xi: \Sigma \to \R$ smooth and compactly supported in $U\cap V$ such that  the right side of \eqref{T variation} is non-zero. This exists by the fundamental theorem of variational calculus.\newline
Finally, using the first variation formula for the Willmore energy \cite[Lemma 2.4.]{Rupp2023volume}, we see that $t\mapsto \mc W(f_t)$ is differentiable in $t=0$, which implies the conclusion.
 \end{proof}
 \begin{cor}\label{cor:usc}
 The map $R \mapsto \beta_0(R)$ is upper semicontinuous in $\R\setminus \{ \mc T(\mb S^2)\}$ and the maps $R \mapsto \beta_g(R)$ for $g\geq 1$ are upper semicontinuous on $\R$.
 \end{cor}
 \begin{proof}
We will prove that \begin{equation}
\limsup _{n\to \infty} \beta_g(R_n)\leq \beta_g(R) \label{usc}
\end{equation}
for any $R \in \R$ and $g\in \N_0$ and any converging sequence $R_n \to R$ if $R \neq \mc T(\mb S^2)$ or $g\neq 0$. Let $f$ be an immersion with $\mc T(f) = R$. By Lemma \ref{lem:First variation of T}, we find immersions $\psi_{n,f}$ such that $\mc T(\psi_{n,f})= R_n$ and $\mc W(\psi_{n,f}) \to \mc W(f)$. This implies \eqref{usc} as for any $f \in \mc S_g$ with $\mc T(f) =R$
\[\limsup_{n\to \infty} \inf _{\substack{\psi \in \mc S_g \\\mc T(\psi) = R_n}} \mc W(\psi) \leq \limsup_{n\to \infty} \mc W(\psi_{n,f}) = \mc W(f)\]
and by taking the infimum over all $f$, we obtain \[\limsup _{n\to \infty} \beta_g(R_n) \leq \beta_g(R).\qedhere\]
 \end{proof}
\section{Existence of minimizers of the Willmore energy under the total mean curvature constraint}\label{sec:ExistenceOfMinimizers}
In this chapter, we want to investigate the existence and smoothness of minimizers for the Willmore energy under the constraint $\mc T$. We will prove that if the total mean curvature constraint lies in a certain interval, for which the infimum is bounded by $8\pi$, minimizers are attained and smooth.
\subsection{Weak immersions}\label{subsec:WeakImmersions}

Our ultimate goal is to prove the existence of minimizers of the constrained problem. We will use the direct method in the calculus of variations and consider a minimizing sequence of the Willmore energy. Ideally, this minimizing sequence of immersions is precompact in some topology. However, the $C^\infty$-topology is too strong to expect any compactness. To address this challenge, we will define the class of weak immersions, originally introduced independently by \cite{KuwertLi} and \cite{RiviereVariationalPrinciples}, expanding our class of potential competitors to ensure that the minimizing sequence converges to another weak competitor in the limit. 

Consider $\Sigma$ a smooth, closed and oriented surface and let $g_0$ be a smooth reference metric on it. Let $u\in C^\infty(\Sigma, \R^n)$ and denote by $\nabla ^k u $ the $k$-th covariant derivative of $u$. Let $|\nabla ^k u|_{g_0}$ be the norm of $\nabla ^k u$ defined in a local chart by
\[|\nabla ^k u|^2_{g_0} = g_0^{i_1j_1}\cdots g_0^{i_kj_k} (\nabla ^k u_l)_{i_1\ldots i_k} (\nabla ^k u_l)_{j_1\ldots j_k}.\]
Let $p\in [1,\infty)$. For $u\in C^\infty(\Sigma)$, we consider the norm
\[\|u\|_{W^{k,p}} = \sum _{j=0}^k \left ( \int_\Sigma |\nabla ^j u|_{g_0}^p \dif \vol _{g_0}\right )^{\frac{1}{p}}.\]
Then, the Sobolev space $W^{k,p}(\Sigma, \R^n)$ is defined to be the completion of $C^\infty(\Sigma, \R^n)$ with respect to $\|\cdot \| _{W^{k,p}}$. One can view this space as a subspace of $L^p(\Sigma, \R^n)$, see \cite{Hebey} for an introduction. For $p=\infty$, we define 
\[W^{1,\infty}(\Sigma,\R^n) \cqq \{f:\Sigma \to \R^n, \; \text{$f$ is Lipschitz}\}.\]
Notice that for $f\in W^{1,\infty}(\Sigma,\R^n)$, for any chart $(x,U)$ the map $f\circ x^{-1}\vert _{x(V)}$, where $V\subset\subset U$, is Lipschitz and so $f\circ x^{-1}\in W^{1,\infty}(x(V), \R^n)$ and, thus, admits a weak derivative $D(f\circ x^{-1})$. Notice also that since $\Sigma$ is compact, these definitions do not depend on the choice of the metric $g_0$.

We will be interested in Lipschitz immersions $f \in W^{1,\infty}(\Sigma, \R^3)$ such that
\begin{equation}
\exists C \geq 1 \text{ such that } C^{-1} g_0(X,X)\leq |Df(X)|^2 \leq C g_0(X,X) \quad \forall X \in T\Sigma. \label{Lipschitz condition}
\end{equation}
This will ensure that the metric does not degenerate.
\begin{defi}
We define the \emph{space of weak immersions with finite total curvature} as
\[\mc E _\Sigma \cqq \left \{f \in W^{1,\infty}(\Sigma, \R^3)\cap W^{2,2}(\Sigma, \R^3), \; \eqref{Lipschitz condition} \text{ holds for some $C\geq 1$}\right \}.\]
\end{defi}
For $f \in \mc E_\Sigma$, we can define $\vec H, \vec n$ and all other quantities introduced in Section \ref{subsec:DefinitionsFromDifferentialGeometry}, possibly up to nullsets. The main goal will be to minimize $\mc W$ in the space $\mc E_\Sigma $ and find smooth minimizers. Thus, we introduce 
\begin{equation}
\beta_g(R) \cqq \inf \{\mc W(f), \; f\in \mc E_\Sigma,\; \mc T(f) = R\}.\label{beta g tau for weak surfaces}
\end{equation}
Here, $R \in \R$ and $\Sigma$ is a closed, oriented and connected surface of genus $g$. Notice that $\mc S_g\subset \mc E_\Sigma$. A priori, \eqref{beta g tau for weak surfaces} is an abuse of notation as we already introduced $\beta_g(R)$ as the constrained minimum of the Willmore energy in the class of smooth immersions. However, we will see that these two infima are equal as Theorem \ref{thm:KMR} shows that the infimum in $\mc E_\Sigma$ is attained by a smooth embedding.  

Notice that the Willmore energy is invariant under diffeomorphisms of the domain, i.e., $\mc W(f) =\mc W(f \circ \psi)$ if $\psi$ is a diffeomorphism of $\Sigma$. This is problematic as the group of diffeomorphisms of $\Sigma$ is not locally compact. One way to circumvent this is to work with conformal immersions. 
\begin{defi}\label{Weak conformality}
 A map $f\in \mc E_\Sigma$ is called \emph{weakly conformal} if and only if it holds
\begin{equation}
\begin{cases}
\langle\parpar{x^1}{f}, \parpar{x^2}{f}\rangle = 0 &\text{ a.e. in }\Sigma\\
|\parpar{x_1}{f}| = |\parpar{x_2}{f}| &\text{ a.e. in }\Sigma
\end{cases} \label{eq:weak conformality}
\end{equation}
in every conformal\footnote{In the sense of Definition \ref{Conformality}, i.e., as an immersion between the Riemannian manifolds $(\Sigma, g_0)$ and $(\R^2, \langle \cdot ,\cdot \rangle)$.} chart $x$ of $\Sigma$. By \eqref{eq:weak conformality}, there is some map $\lambda$, called the \emph{conformal factor of $f$} such that
\[f^*\langle \cdot, \cdot \rangle   = e^{2\lambda} g_{0}.\]
Furthermore, by \eqref{Lipschitz condition}, $\lambda \in L^\infty (\Sigma)$.
\end{defi}
\subsection{Existence of minimizers for spherical immersions}\label{subsec:g=0}
We will begin by considering immersions from the 2-sphere $\mb S^2$. We will show that for $R \in (0, \sqrt{2}\mc T(\mb S^2))$, minimizers are attained and smooth. In order to do so, we will employ the following compactness result, which was obtained by  \textsc{Mondino} and \textsc{Rivière} \cite[Theorem  1.5]{MondinoRiviere}, see also \cite[Theorem 3.2]{KMR}. The additional convergence of the total mean curvature was shown by \textsc{Mondino} and \textsc{Scharrer} \cite[Theorem 3.3]{MondinoScharrer}.
\begin{thm} \label{thm: Weak compactness}
Suppose $(f_n)_n\subset \mc E_{\mb S^2}$ is a sequence of weak immersions and that
\[\limsup _{n\to\infty} \mc A (f_n) < \infty, \quad \liminf _{n\to  \infty} \diam f_n (\mb S^2) >0\quad\text{and}\quad \limsup _{n\to \infty}\mc W(f_n) < 8\pi - \delta\]
for some $\delta > 0$. Then, there exists a subsequence that we still denote $(f_n)_n$, a sequence $(\psi_n)_n$ of conformal diffeomorphisms of $\mb S^2$ and finitely many points $\{a_1,\ldots, a_n\}$ such that
\[f_n \circ \psi_n \wto f_\infty \quad \text{weakly in $W^{2,2}_\loc (\mb S^2 \setminus \{a_1, \ldots, a_n\})$,}\]
where $f_\infty \in \mc E_{\mb S^2}$ is weakly conformal. In addition, by lower semicontinuity of $\mc W$ under weak $W^{2,2}$-convergence, we have
\[\mc W(f_\infty) \leq \liminf _{n\to \infty} \mc W(f_n) <8\pi-\delta.\]
So $f_\infty$ is a (weak) embedding, and moreover
\[\mc A (f_n) \to \mc A (f_\infty),\quad \mc V (f_n) \to \mc V(f_\infty)\quad \text{ and } \int _{\mb S^2} H_{f_n} \dif \vol _{f_n} \to \int _{\mb S^2} H _{f_\infty} \dif \vol _{f_\infty}\]
as $n\to \infty$.
\end{thm}
With this compactness result, we can show the existence of smooth minimizers for $g=0$.
\begin{thm}[The genus 0 case]\label{thm: genus 0 case}
For all $R \in (0, \sqrt{2} \mc T(\mb S^2))$, there exists a smooth embedding $f_R^0\in \mc S_0$ which minimizes the Willmore energy subject to $\mc T(f_R^0)=R$. Furthermore, $\beta_0$ is continuous on $I$, where $I$ is defined as
\begin{equation}
I \cqq (0,\sqrt{2} \mc T(\mb S^2))\setminus \{\mc T(\mb S^2)\} = (0, \sqrt{32\pi})\setminus \{4\sqrt{\pi}\}.\label{I}
\end{equation}
\end{thm}
\begin{proof}
We claim that for $R \in (0, \sqrt{2} \mc T(\mb S^2))$, it holds that $\beta_0(R)<8\pi$. We will postpone the proof of the claim until Section \ref{sec:8piconstruction}, Proposition \ref{prop:8piconstruction}. Let $R \in (0, \sqrt{2} \mc T(\mb S^2))$. Then, we find a minimizing sequence of weak immersions $f_n : \mb S^2 \to \R^3$ with $\mc T(f_n) = R$. By rescaling, we may assume that $\mc A (f_n) = 1$. By Simons inequality, Lemma \ref{lem:Simon}, we have that $\diam(f_n(\mb S^2))>C>0$ independent of $n$. Thus, we can apply Theorem \ref{thm: Weak compactness} to find $f_R^0\in \mc E_{\mb S^2}$ such that, by the convergence of the area and the total mean curvature, we see that $\mc T(f_R^0) = R$. By the lower semicontinuity of $\mc W$ under $W^{2,2}$-convergence, we obtain $\mc W(f_R^0) = \beta_0(R) $. In particular, by the Li-Yau inequality, $f_R^0$ is actually an embedding. 

Since $\mc W$, the area and the total mean curvature are Fréchet-differentiable, we can apply the method of Lagrange multipliers. Thus, if we set
\[\mc H_{\alpha}^{c_0}(f) \cqq \int _{\mb S^2} \left (\frac{H_{f}}{2} - c_0\right )^2 \dif \vol _{g_f} +\alpha \mc A  (f),\]
we find suitable $c_0, \alpha \in \R$ such that
\[\dv{}{t}\bigg\vert _{t=0} \mc H_\alpha ^{c_0} (f_R^0 +t\psi) =0 \quad \text{for all $\psi \in C^\infty(\mb S^2, \R^3)$}.\]
This means that $f_R^0$ is a weak Canham-Helfrich-immersion, see \cite[(4.1)]{MondinoScharrer}, for which the smoothness was shown in \cite[Theorem 4.3]{MondinoScharrer}. Finally, we prove continuity of $\beta_0\vert _I$. The upper semicontinuity of $\beta_0\vert _I$ follows by Corollary \ref{cor:usc}. To prove the lower semicontinuity of $\beta_0\vert _I$, we consider a sequence $R_n\in I$ such that $R_n \to R \in I$. Then as before, there exist constrained minimizers $f^0_{R_n}\in \mc S_0$ satisfying $\beta_0(R_n) = \mc W(f^0_{R_n})$, $\mc T(f^0_{R_n})=R_n$. Again by Theorem \ref{thm: Weak compactness}, the $f^0_{R_n}$ weakly converge to some $\psi\in \mc E_{\mb S^2}$ with $\mc T(\psi) = \limn \mc T(f^0_{R_n}) = R$ and 
\[\liminf _{n\to \infty} \beta_0(R_n)=\liminf _{n\to \infty} \mc W(f^0_{R_n})\geq \mc W(\psi)\geq \beta_0(R).\]
This proves lower semicontinuity and so $\beta_0$ is continuous in ${(0, \sqrt{2} \mc T(\mb S^2))\setminus \{\mc T(\mb S^2)\}}$ by Corollary \ref{cor:usc}. 
\end{proof}
\begin{remark}
The definition of a weak Canham-Helfrich immersion in \cite[(4.1)]{MondinoScharrer} assumes that $\alpha \geq 0$. However, the regularity result works despite the theorem being stated only for $\alpha \geq 0$ as the proof remains unchanged.	
\end{remark}

\subsection{The general case}\label{subsec:GeneralCase}
Let us state the main theorem, which for completeness also includes the case $g=0$ covered in Theorem \ref{thm: genus 0 case}.
\begin{thm}[Total mean curvature constrained minimizers of $\mc W$ for arbitrary genus] \label{thm:KMR}
Suppose that $\Sigma$ is a closed, orientable and connected surface of genus $g\in \N_0$. 

Then, for any $R\in I$, where $I$ is the interval defined in \eqref{I}, there exists a smooth embedding $f^g_R$ of $\Sigma$ into $\R^3$, with $\mc T(f_R^g) = R$ and 
\[\mc W(f^g_R) = \inf _{\substack{f \in \mc E_\Sigma \\ \mc T(f) = R}} \mc W(f),\] i.e., $f^g_R$ minimizes the Willmore energy among all Lipschitz immersions of $\Sigma$ with second fundamental form bounded in $L^2$ and fixed total mean curvature constraint equal to $R$. Furthermore, the minimizers $f^g_R$ are smooth embedded surfaces of genus $g$. Finally, $\beta_g \vert _I$ is continuous.
\end{thm}
\begin{remark}[The interval $I$] \label{remark: I}
The set $I$ is chosen in such a way that for $R\in I$, we have
\begin{equation}
\inf _{\substack{f \in \mc E_{\Sigma}\\ \mc T(f)=R}} \mc W(f) < \min \{8\pi, \omega_g, \beta_g + \beta_0(R) - 4\pi\}. \label{explicit interval}
\end{equation}
Here,
\[\omega_g   \cqq \min \left \{4\pi + \sum _{i=1}^p (\beta _{g_i} - 4\pi) , \;  g= g_1 + \ldots + g_p,\;  1\leq g_i < g\right \},\]
where $\omega_1  = \infty$. The proof of \eqref{explicit interval} for $R\in I$ is the content of Sections \ref{sec:BauerKuwert} and \ref{sec:8piconstruction}. For the current section, however, we will assume that \eqref{explicit interval} holds. We want to quickly summarize why we need these inequalities. 

By the Li-Yau inequality, the $8\pi$ bound ensures that any minimizer is actually embedded which makes it a natural requirement to work with. Furthermore, the $8\pi$ bound, along with the $\omega_g $ bound, are a technical assumption to ensure that the conformal structures of the minimizing sequence do not degenerate in the moduli space, see \cite[Theorem 5.1]{CompactnessKuwertSchatzle}. It is already known by the work of \cite{MarquesNeves} that $\omega_g  \geq 8\pi$ for $g\geq 1$, which means that we effectively only need to prove the $8\pi$ bound and the $\beta_g  + \beta_0(R) - 4\pi$ bound. The $\beta_g  + \beta_0(R) - 4\pi$ bound is used to exclude a blow-up phenomenon of the surface  under which the topology is lost in the limit. In this case, the minimizing sequence appears as a spherical bubble resembling a minimizer of $\beta_0(R)$, to which a small region carrying the $g$ handles is glued. We will show that when such a bubble arises, the energy is at least $\beta_g  + \beta_0(R) - 4\pi$. 
\end{remark}
\subsection{Proof of Theorem \ref{thm:KMR}}\label{subsec:ProofOfKMR}
As the case of $g=0$ is already treated in Theorem \ref{thm: genus 0 case}, we will assume from now on that $g\geq 1$. We will follow closely the work of \textsc{Keller}, \textsc{Mondino}, and \textsc{Rivière} \cite{KMR}, who proved a similar result for the isoperimetric ratio instead of our total mean curvature constraint, see the discussion from Section \ref{subsec:IntroductionIsoConstrainedProblem}. Let us state what the main ideas for the proof will be. 
%

Suppose $R\in I $. Consider a minimizing sequence $(f_k)_{k}\subset \mc E_\Sigma $ of the Willmore energy with a fixed total mean curvature ratio, i.e., $\lim _{k\to \infty} \mc W(f_k) = \beta_g(R)$  and $\mc T(f_k) = R$. We want to show that this sequence is weakly compact in $W^{2,2}(\Sigma, \R^3)$ with a limit $f \in \mc E_\Sigma$ which also satisfies $\mc T(f) = R$. Then, by the lower semicontinuity of $\mc W$ under weak $W^{2,2}$-convergence, it follows that 
\[\beta _g(R) = \mc W(f) \leq \liminf _{k\to \infty} \mc W(f_k) = \beta_g(R).\]
Because the Willmore energy and the total mean curvature ratio are scaling-invariant, we may assume that $\mc A (f_k)=1$ is fixed. We will initially deal with the problem of the right choice of coordinates to work around the invariance of the Willmore energy and the total mean curvature ratio under diffeomorphisms of the domain. For this reason, we will choose coordinates such that $f_k$ becomes a weakly conformal map, see Definition \ref{Weak conformality}. This is done in Section \ref{subsubsection:Conformality}. In Section \ref{subsub:Energy concentration}, we use a Besicovitch covering of $\Sigma$ to obtain at most finitely many points $\{a_1,\ldots, a_n\}$ where the Willmore energy accumulates. After that, we show that we can estimate the conformal factors $\lambda_k$ of the immersion $f_k$ away from the points of energy concentration. We then face two cases:
\begin{enumerate}[label=Case \arabic*)]
\item  We first assume that the conformal factors stay bounded. Then, we see that the $f_k$ remain bounded in $W^{2,2}$ away from the points of energy concentration so we can use weak compactness to obtain a subsequence which weakly converges in $W^{2,2}$ away from these points. Using a lemma from \textsc{Rivière} \cite{RiviereVariationalPrinciples}, we can show that there are no points where the energy concentrates, which allows us to conclude that the immersions converge weakly in $W^{2,2}$ to an immersion $f \in \mc E_{\Sigma}$. Finally, we check that the total mean curvature constraint is preserved in the limit and that $f$ satisfies a certain Euler-Lagrange equation which allows us to conclude that $f$ is smooth.
\item In this case, we assume that the conformal factors diverge to $-\infty$. One can show that each point of area concentration has at least $4\pi$ Willmore energy, and by bounds on the Willmore energy, we conclude that there must be exactly one point of concentration of area and energy. Finally, we perform a ``Cut and Fill`` operation in which we cut out the blow-up point, fill in a disk with almost no Willmore energy and glue a minimizer of the genus $0$ problem with total mean curvature ratio equal to $R$ to the problematic area to estimate the Willmore energy. This will then contradict the $\beta_g + \beta_0(R) -4\pi$ bound.
\end{enumerate}
\subsubsection{Conformality of the minimizing sequence}\label{subsubsection:Conformality}
We will use the following \cite[Corollary 4.4]{RiviereLectureNotes}:
\begin{thm}
Let $\Sigma$ be a closed smooth $2$-dimensional manifold. Then, any $f \in \mc E_\Sigma$ defines a smooth conformal structure on $\Sigma$. In particular, there exists a constant curvature metric $h$ of unit area on $\Sigma$ and a Lipschitz diffeomorphism $\Psi$ of $\Sigma$ such that $f \circ \Psi$ realizes a weak conformal immersion of the Riemann surface $(\Sigma,h)$ and $h$ and $g_{f}$ are conformally equivalent.
\end{thm}
Thus, we can replace the $f_k$ by $f_k \circ \Psi_k$ where the $\Psi_k$ are Lipschitz diffeomorphisms on $\Sigma$, such that $f_k \circ \Psi_k$ is a weak conformal immersion and we do not change the conformal class. 
\subsubsection{Points of energy concentration}\label{subsub:Energy concentration}
Let us denote by $h_k$ the metric of constant scalar curvature which lies in the same conformal class $c_k$ as $g_k = g_{f_k}$. By our assumption $\sup _{k\in \N}\mc W(f_k) < \min\{8\pi, \omega_g \}$, see Remark \ref{remark: I}. Hence, by \cite[Theorem 5.1]{CompactnessKuwertSchatzle}, the conformal classes $c_k$ lie in a precompact set in the moduli space\footnote{For a definition of the moduli space and its topology, see \cite{jost,imayoshi}.}. Thus, up to subsequences, we may assume that $c_k \to c_\infty$, which implies that $h_k $ converge to $h_\infty$, the constant scalar curvature metric associated to $c_\infty$, where this convergence holds in $C^s(\Sigma)$ for all $s$.  

Consider $x\in \Sigma$ and consider balls in $(\Sigma,h_k)$ around $x$ and denote by 
 \[{\tilde \rho}_k^x \cqq \inf \left \{\rho_k,\; \int_{B_{\rho_k(x)}} |\dif \vec n _{f_k}|^2 _{h_k} \dif  \vol _{h_k} \geq \frac{8\pi}{3}\right \}.\]
 Here, $B_{\rho_k}(x)$ is the geodesic ball in $(\Sigma,h_k)$ around $x$ with radius $\rho_k$. The reason why we want the integral to be bounded by $\frac{8\pi}{3}$ is to use Hélein's moving frame technique in the proof of Lemma \ref{lem:lambdakminusCk} as done in \cite[Lemma 4.1]{KMR}. Next, consider the convexity radius $r_0$ in $(\Sigma, h_\infty)$, i.e., the largest radius such that for $r\leq r_0$, the balls $B_r(x)\subset \Sigma$ are strictly geodesically convex. This is positive as $\Sigma$ is compact, see \cite[Proposition 95]{berger}. Now, set $\rho_k^x \cqq \min\{r, {\tilde \rho}_k^x\}$ for some fixed $0<r\leq r_0$. Then,  $\{\overline{B_{\rho_k^x}(x)}\} _{x\in \Sigma}$ is a cover of $\Sigma$. We choose $r$ small enough such that we can apply the Besicovitch covering theorem \cite[Theorem 2.8.14]{GMT} to obtain a subcover $\{\overline{B_{\rho_k^i}(x^i_k)}\} _{i\in I_k}$ where $I_k$ is finite and each point in $\Sigma$ is covered by at most $N\in \N$ balls, where $N$ is independent of $k$ by $C^\infty$-convergence of the metrics $h_k$ to $h_\infty$. In $(\Sigma, h_\infty)$, the cardinality of any family of disjoint balls in $\Sigma$ with radius $r$ is at most $M = M(h_\infty)$\footnote{By the continuity of $h_\infty$, every ball of radius $r$ has at least area $c r^2$ for some $c>0$.}. This also implies
 \[|\{i \in I_k, \; \rho_k^i = r\}| \leq NM\]
 for sufficiently large $k$ again by the convergence $h_k \to h_\infty$. To show that the number of balls with $\rho_k^i = \tilde{\rho}_k^i $ is bounded, we use that 
 \[\int_{ B_{\rho_k^i}(x^i_k)} | \dif \vec n_{f_k}|^2_{h_k} \dif  \vol  _{h_k} = \frac{8\pi}{3},\]
 so that by the Besicovitch covering and the fact that the Willmore energy $\mc W(f_k)$ is uniformly bounded from above, we obtain, also using \cite[(2.30)]{RiviereLectureNotes}, that
 \[|\{i\in I_k,\; \rho_k^i  = {\tilde \rho}_k^i \}| \leq N  \frac{\sup _{k\in \N}\int _{\Sigma } |\dif \vec n _{f_k}| ^2 _{h_k} \dif  \vol  _{h_k}}{ \frac{8\pi}{3}} <\infty.\]
Thus, up to a subsequence, we get a covering $\{\overline{B_{\rho_k^i}(x^i_k)}\}_{i \in I}$ such that 
\[I \text{ does not depend on $k$}, \quad x_k^i \to x_\infty ^i \quad \text{and}\quad\rho_k^i \to \rho_\infty ^i.\]
For $I_0 = \{i \in I, \rho_\infty ^i = 0\}$ and $I_1  = I\setminus I_0$, we obtain by the convergence of the metric that
\[\Sigma \subset \bigcup _{i\in I_1}\overline{B_{\rho_\infty ^i}(x^i _\infty)}.\]
As $\rho_\infty ^i \leq r_0$, the balls are strictly geodesically convex and thus the number of intersections of the boundaries of two such balls is at most 2. In other words,
\[\Sigma \setminus \bigcup _{i\in I_1}B_{\rho_\infty ^i}(x^i _\infty) = \{a_1,\ldots, a_n\}\]
is a finite union of points. These are our points of energy concentration. 

 Away from the $a_i$, the conformal factors of the immersions $f_k$ are well-behaved. This is captured in the following \cite[Lemma 4.1]{KMR}.
 \begin{lemma}\label{lem:lambdakminusCk}
 Let $\eps>0$. Then, denoting by $\lambda_k$ again the conformal factor (with respect to the reference metric $g_0$), see Definition \ref{Weak conformality}, there exist constants $C_k$ and a constant $C_\eps$ such that up to subsequences, we have
 \[\|\lambda_k - C_k\|_{L^\infty (\Sigma \setminus \bigcup_{i=1}^n B_\eps(a_i))} \leq C_\eps.\]
 
 \end{lemma}
It is important to stress that the constants $C_k$ do \emph{not} depend on $\eps$. We now face two possibilities. Either the $C_k$ remain bounded or they diverge to $\pm \infty$. In the next section, we assume that $\sup _{k\in \N} |C_k| <\infty$, i.e., for each $\eps>0$ the conformal factors $\lambda_k$ remain bounded on $\Sigma \setminus \bigcup_{i=1}^n B_\eps(a_i)$.
 \subsubsection{Case 1: Boundedness of the conformal factors}\label{subsubsec:Case1}
 We now assume that the $C_k$ remain bounded. By the previous lemma, this immediately implies that also the conformal factors $\lambda_k$ remain bounded on $\Sigma \setminus \bigcup_{i=1}^n B_\eps(a_i)$. Because $\Delta f_k =  e^{2 \lambda_k} \vec H_{f_k}$ in a conformal chart\footnote{Here, $\Delta$ denotes the flat Laplacian $\Delta \cqq \partial _{x^1}^2 + \partial ^2_{x^2}$ in the chart $x$.} and by assumption $\limsup _{k\to \infty}\mc W(f_k) < \infty$, we obtain
 \begin{align*}
 \|\Delta f_k\| _{L^2\left(\Sigma \setminus \bigcup_{i=1}^n B_\eps(a_i)\right)} &< C(\eps)\\
 \|\log |\nabla f_k|\|_{L^\infty\left(\Sigma \setminus \bigcup_{i=1}^n B_\eps(a_i)\right)} &<C(\eps).
 \end{align*}
By usual regularity theory, this also shows that $\|\nabla ^2 f_k\|_{L^2\left(\Sigma \setminus \bigcup_{i=1}^n B_\eps(a_i)\right)} < C(\eps)$. Also recall Simons lemma \ref{lem:Simon} from the Appendix, by which we obtain a bound of the diameter in terms of the area and the Willmore energy. As the area is fixed and the Willmore energy is bounded, the diameters of our immersions are bounded uniformly. Thus, after a possible shift of our immersions, we obtain that $f_k$ are bounded in $W^{2,2}\left(\Sigma \setminus \bigcup_{i=1}^n B_\eps(a_i)\right)$ and hence weakly converge to $f_\infty$ (up to subsequences). By Rellich-Kondrachov, we also obtain strong convergence of the gradients, i.e., $\nabla f_k \to \nabla f_\infty$ in $L^p\left(\Sigma \setminus \bigcup_{i=1}^n B_\eps(a_i)\right)$ for all $p<\infty$. Up to another subsequence, they also converge pointwise almost everywhere. By the strong $L^2$-convergence of the gradients, we also obtain
\[\mc A \left (f_k \vert _{\Sigma \setminus \bigcup_{i=1}^n B_\eps(a_i)}\right ) \to \mc A  \left (f_{\infty}\vert _{\Sigma \setminus \bigcup_{i=1}^n B_\eps(a_i)}\right ).\]
By the pointwise convergence of the gradients, it also follows that the conformality condition is preserved in the limit away from the points of energy concentration, i.e., $f_\infty \in \mc E _{\Sigma \setminus \bigcup_{i=1}^n B_\eps(a_i)}$.  

In \cite[Section 4.2.3.]{KMR}, it is shown that under these assumptions, $f_\infty$ does not have any branch points. The idea is to use the following lemma \cite[Lemma A.5]{RiviereVariationalPrinciples}:
\begin{lemma}
Let $\xi$ be a conformal immersion of $D^2\setminus \{0\}$ into $\R^3$ in $W^{2,2}(D^2\setminus \{0\}, \R^3)$ and such that $\log |\nabla \xi| \in L^\infty _{\mathrm{loc}} (D^2\setminus \{0\})$. Assume that $\xi$ extends to a map in $W^{1,2}(D^2)$ and that the corresponding Gauss map $n_{\xi}$ also extends to a map in $W^{1,2} (D^2,\mb S^2)$. Then, $\xi$ realizes a Lipschitz conformal immersion of the whole disc $D^2$ and there exists a positive integer $n$ and a constant $C$ such that
\[(C-o(1)) |z|^{n-1} \leq\left | \parpar{z}{\xi}\right | \leq (C+o(1)) |z|^{n-1}.\]
\end{lemma}
We apply this lemma, where $\xi = f_\infty \circ \psi$, where $\psi:D^2\setminus \{0\}\to \Sigma$ is some conformal chart around $a_i$. One can show, see \cite[Section 4.2.3.]{KMR}, that the 2-dimensional lower density of $f_\infty(\Sigma)$ at $f_\infty(a_i)$ is at least $n$ and the Li-Yau inequality implies that $n<2$. Hence, $n=1$, the branch point is removable and we conclude that $f_\infty \in \mc E_\Sigma$.

It remains to show that the total mean curvature constraint is preserved in the limit. In \cite[Section 4.2.5.]{KMR}, it is already shown that $\mc A (f_k) \to \mc A (f_\infty)$ and that
\[\liminf _{\eps \to 0} \liminf _{k\to \infty} \mc A (f_k\vert _{B_\eps(a_i)}) = 0.\]
It follows by Cauchy-Schwarz that
\[\liminf _{\eps \to 0} \liminf _{k\to \infty}\left |\int _{B_\eps(a_i)} H_{f_k} \dif \vol_{g_k} \right |\leq \liminf _{\eps \to 0} \liminf _{k\to \infty}2\sqrt{\mc A (f_k\vert _{B_\eps(a_i)})} \sqrt{\mc W(f_k)} = 0.\]
Furthermore, since $f_k \wto f_\infty $ in $W^{2,2}(\Sigma \setminus \bigcup _{i=1}^n B_\eps(a_i))$, in particular it follows that 
\[\int_{\Sigma \setminus \bigcup _{i=1}^n B_\eps(a_i)} H_{f_k} \dif \vol _{g_k} \to \int_{\Sigma \setminus \bigcup _{i=1}^n B_\eps(a_i)} H_{f_\infty} \dif \vol _{g_\infty}.\]
Thus, the total mean curvature converges as well. Hence $\mc T(f_\infty) = R$. We use the lower semicontinuity of the Willmore energy under weak $W^{2,2}$ convergence to obtain $\mc W(f_\infty)\leq \liminf _{n\to \infty} \mc W(f_k)$. Finally, the smoothness of $f_\infty$ follows by \cite[Theorem 4.3]{MondinoScharrer} and we denote by $f^g _R$ the map $f_\infty$.
\subsubsection{Case 2: Divergence of the conformal factors}\label{subsubsec:Case2}
Up until now, we assumed that the conformal factors remain bounded. In this section, we look at the case where the conformal factors diverge. Since we assumed that the area is fixed, we can exclude the case $C_k \to \infty$. However, we cannot exclude the case $C_k \to -\infty$ right now because the area might concentrate around the points of energy concentration. 

In \cite[Lemma 4.4]{KMR}, it was shown that any such point of area concentration contains at least $4\pi$ Willmore energy and thus, at most one such point can exist. This follows by the $8\pi - \delta$ bound of the Willmore energy.
\begin{lemma}\emph{\cite[Lemma 4.4]{KMR}}\label{lem:area->willmore>=4pi}
Denote again by $(f_k)_k$ our minimizing sequence of weak immersions, where $a_i$, $ i=1,\ldots, n$ are the points of energy concentration. Assume that there exists an index $i$ such that
\[\liminf_{\eps \to 0} \liminf_{k\to \infty} \mc A (f_k(B_\eps(a_i))) > 0.\]
Then, 
\[\liminf _{\eps \to 0} \liminf_{k\to \infty} \mc W(f_k(B_\eps(a_i))) \geq 4\pi.\]
\end{lemma}
Furthermore, since the $a_i$ are the only points where area can concentrate, there must be at least one $a_i$ for which the assumptions of Lemma \ref{lem:area->willmore>=4pi} hold. Consequently, there exists exactly one point, denoted as $a^*$, where both area and energy concentrations occur.  

The behavior that we observe is that the region around $a^*$ blows up, while the remaining surface is blown down. In other words, $f_k(B_\eps(a^*))$ contains almost all of the area, while $f_k(\Sigma \setminus B_\eps(a^*))$ is the shrinking region that contains the topological information. We want to isolate these two parts and complete each of them to obtain new closed surfaces. For these new surfaces, we estimate their Willmore energy and bring this to a contradiction. This procedure is done by the following lemma, see also \cite[Lemma 5.2]{MondinoRiviere} and \cite[Lemma 4.9]{KMR} for a similar statement. 
%
\begin{lemma}[Cut and Fill]\label{Cut and Fill Lemma}

Let $(f_k)_k\subset \mc E_\Sigma$ be a sequence of conformal weak immersions mapping into $\R^3$. Assume that
\begin{equation}
\limsup _{k\to \infty} \left (\mc A (f_k) + \mc W(f_k)\right ) < \infty. \label{area and dn estimate}
\end{equation}
Let $a\in \Sigma$ and $s_k,t_k\searrow 0$ such that
\[\frac{t_k}{s_k}\to 0\]
and suppose that
\begin{equation}
\lim _{k\to \infty} \int _{B_{s_k}(a)\setminus B_{t_k}(a)}[1+ |\mb I_{ f_k}|_{g_k}^2]\dif  \vol  _{g_k} =0. \label{No area and Willmore in annulus}
\end{equation}
We fix $x_0 \in \mb S^2$ and denote by $\pi: \mb S^2\setminus \{x_0\} \to \C$ the stereographic projection from $x_0$ onto $\C$ and by $\omega: B_1(0) \subset \C \to \Sigma$, $ \omega(0)=a$ a conformal chart. Then, for sufficiently large $k$ there exist possibly branched weak immersions\footnote{See \cite{MondinoRiviere} for a definition.}
\begin{enumerate}[label=\roman*)]
\item  $\xi_k^1: \mb S^2 \to \R^3$ such that 
\begin{equation}
\xi_k^1 \circ \pi^{-1} = f_k \circ \omega \quad \text{ in }\omega ^{-1}(B_{t_k}(a))\label{xi^1 k}
\end{equation}
and 
\begin{equation}
\lim _{k\to \infty}\int _{\mb S^2 \setminus \pi^{-1}(\omega ^{-1}(B_{t_k}(a)))}|\mb I^0_{\xi^1_k}|_{g_{\xi^1_k}}^2 \dif  \vol _{g_{\xi^1_k}} =0,\label{tracefree form 1}
\end{equation}
\item  $\xi_k^2: \Sigma \to \R^3$ such that 
\begin{equation}
\xi_k^2\vert _{\Sigma \setminus B_{s_k}(a)} = f_k \vert _{\Sigma \setminus B_{s_k}(a)}\label{xi^2 k}
\end{equation}
and 
\begin{equation}
\lim _{k\to \infty}\int _{B_{s_k}(a)}|\mb I^0_{\xi^2_k}|_{g_{\xi^2_k}}^2 \dif  \vol_{_{\xi^2_k}}  =0,\label{tracefree form 2}
\end{equation}
\end{enumerate}

where $\mb I^0$ is the trace-free second fundamental form. The balls $B_{s_k}(a)$ and $B_{t_k}(a)$ are always measured in a reference metric $g_0$.
\end{lemma}
An intuition for this statement is given in Figure \ref{tikz figure for cut and fill}. We give the proof for this statement and follow closely the proof in \cite[Lemma 5.2]{MondinoRiviere}.
\begin{proof}
Consider first a conformal chart $\omega: B_1(0)\to U\subset \Sigma$, where $\omega(0)=a\in U$ and choose $\lambda\geq 1$ such that $\lambda^{-1} |v|^2 \leq |\dif \omega(v)|_{g_0}^2 \leq \lambda |v|^2$ for all $v\in \C$. Then, for $k$ large enough such that $B_{s_k}(a) \subset U$ we have $\omega(B_{s_k/\lambda}(0)) \subset B_{s_k}(a)$ and $\omega(B_{t_k \lambda}(0)) \supset B_{t_k}(a)$. We set $\rho_k \cqq \lambda \sqrt{t_k /s_k} $ and consider the map 
\[\psi_k: B_{1/\rho_k}(0)\setminus B_{\rho_k}(0) \to \R^3,\quad \psi_k = f_k \circ \omega \circ e_{\sqrt{t_k s_k}}, \]
where $e_{r}(z) \cqq rz$ for $r\in \R$. Notice that $\rho_k \to 0$ and that $\psi_k$ is still a Lipschitz conformal immersion. Furthermore, by \cite[(2.29)]{RiviereLectureNotes} and \eqref{No area and Willmore in annulus}
\begin{equation}
\limsup  _{k\to \infty} \int _{B_{1/\rho_k}(0) \setminus B_{\rho_k}(0)}|\nabla \vec n_{\psi_k}|^2 \dif x \leq  \lim _{k\to\infty} \int _{B_{s_k}(a) \setminus B_{t_k}(a)}|\dif \vec n_{f_k}|_{g_k}^2 \dif \vol _{g_k} =0. \label{limit of gradient of normal goes to 0}
\end{equation}
We denote by $\lambda_k = \log |\partial _{x_1} \psi_k| = \log |\partial _{x_2} \psi_k| $ the conformal factor of $\psi_k$ in $B_{1/\rho_k}(0)\setminus B_{\rho_k}(0)$. First, we want to estimate these $\lambda_k$ and show that they tend to $-\infty$. 

 By \cite[Lemma IV.1]{Bernard}, it is possible to extend $\vec n_{\psi_k} \in W^{1,2}(B_{{1/\rho_k}}(0)\setminus B_{\rho_k}(0),\mb S^2) $ to a ${W^{1,2}(B_{1/\rho_k}(0), \mb S^2)}$ map such that 
\[ \int _{B_{1/\rho_k}(0)}|\nabla \vec n_{\psi_k}|^2 \dif x \leq C  \int _{B_{1/\rho_k}(0) \setminus B_{\rho_k}(0)}|\nabla \vec n_{\psi_k}|^2 \dif x  .\]
In particular, for $k$ large enough, the right-hand side is bounded by $8\pi/3$ so that we can use Hélein's energy controlled moving frame theorem \cite[Theorem V.2.1]{Helein} to obtain the existence of an orthonormal frame $(\vec e_{1,k}, \vec e_{2,k})$ on $B_{1/\rho_k}(0)\setminus B_{\rho_k}(0)$ such that
\[\vec e_{1,k} \times \vec e_{2,k} = \vec n_{\psi_k}\]
and
\begin{equation}
\int _{B_{1/\rho_k}(0) \setminus B_{\rho_k}(0) } |\nabla \vec e_{1,k}|^2 + |\nabla \vec e_{2,k}|^2 \dif x \leq C \int _{B_{1/\rho_k}(0)\setminus B_{\rho_k}(0)}|\nabla \vec n_{\psi_k}|^2 \dif x. \label{Heleins frame estimate}
\end{equation}
By \cite[(5.21)]{RiviereLectureNotes}, we know that
\begin{equation}
\Delta \lambda_k = - \nabla ^\perp \vec e_{1,k} \cdot \nabla \vec e_{2,k}, \label{Delta lambda_k}
\end{equation}
where we employ the notation from \cite{RiviereLectureNotes}. \eqref{Delta lambda_k} is interpreted in the usual distributional sense. Let $\alpha>0$ and $k$ large enough such that $ \rho_k < \alpha$ and consider $\mu_k$ which satisfy the same equation as \eqref{Delta lambda_k} but with zero boundary values:
\begin{equation}
\begin{cases}
\Delta \mu_k = -\nabla ^\perp \vec e_{1,k} \cdot \nabla \vec e_{2,k} & \text{ in } B_{1/\alpha}(0) \setminus B_\alpha(0),\\
\mu_k = 0 &\text{ in } \partial B_{1/\alpha}(0) \cup \partial B_\alpha(0).
\end{cases} \label{Delta mu_k}
\end{equation}
We can now employ Wente's inequality \cite[Corollary 2, Theorem 3]{Topping} to obtain $C>0$ independent of $\alpha$ such that
\begin{align}
\|\mu_k\|_{L^\infty(B_{1/\alpha}(0)\setminus B_{\alpha}(0))} + \|\nabla \mu_k \|_{L^2(B_{1/\alpha}(0)\setminus B_\alpha(0))}&\leq C \|\nabla \vec e_{1,k}\|_{L^2} \|\nabla \vec e_{2,k}\|_{L^2}\notag\\
&\stackrel{\mathclap{\eqref{Heleins frame estimate}}}{\leq } C \int _{B_{1/\rho_k}(0)\setminus B_{\rho_k}(0)} |\nabla \vec n _{\psi_k}|^2 \dif x.\label{Wente}
\end{align}
By \cite[Lemma  2.2]{MondinoRiviere}, using \eqref{area and dn estimate} and that the sectional curvature of $\R^3$ is 0, we obtain a uniform estimate of $\|\nabla \lambda_k\|_{L^{2,\infty}(B_{1/\alpha}(0)\setminus B_\alpha(0))}$\footnote{Here $\|\cdot \|_{L^{2,\infty}(\Sigma)}$ is the Lorentz space (quasi)norm where for a measurable function $f:\Sigma \to \R$, it is defined as
\[\|f\|_{L^{2,\infty}(\Sigma)} \cqq (\sup _{t>0} t^2 \vol_k (\{|f| > t\}) )^{1/2}<\infty.\]
Note that $\|f\|_{L^{2,\infty}(\Sigma)}\leq \|f\|_{L^2(\Sigma) }$ by Chebyshev. Also, since $\mc A (\Sigma)< \infty$, we have ${\|f\|_{L^p(\Sigma)}\leq C_p(1+ \|f\|_{L^{2,\infty}(\Sigma)})}$ for any $p<2$.}. We obtain that $\nu_k \cqq \lambda_k - \mu_k$ is harmonic and by the bound on the $L^2$ norm in \eqref{Wente}, we obtain that it satisfies
\[\limsup _{k\to \infty} \|\nabla \nu_k \|_{L^{2,\infty}(B_{1/\alpha}(0)\setminus B_{\alpha}(0))} < \infty.\]
We now follow the argument presented in the proof of \cite[Theorem 5.5]{RiviereLectureNotes} to show that there exist constants $ \overline{\lambda}_k \in \R$ such that for $\alpha'>\alpha$, we have
\begin{equation}
\limsup _{k\to \infty}  \|\lambda_k  - \overline{\lambda}_k\|_{L^\infty (B_{1/\alpha'}(0) \setminus B_{\alpha'}(0))} < \infty. \label{lambda k dont vary too much}
\end{equation}
By the Poincaré inequality, we have for $1\leq p <2$
\begin{align*}
\|\nu_k - \overline{\nu}_k\|_{L^p(B_{1/\alpha}(0)\setminus B_\alpha(0))} & \leq C_p \|\nabla \nu_k \|_{L^p(B_{1/\alpha}(0)\setminus B_\alpha(0))} \\
&\leq C_p(1+ \|\nabla \nu_k\|_{L^{2,\infty}(B_{1/\alpha}(0)\setminus B_{\alpha}(0))}) < C_p.
\end{align*}
Here, $\overline{\nu}_k$ is the average of $\nu_k$ in $B_{1/\alpha}(0)\setminus B_\alpha(0)$. As $W^{1,p}(B_{1/\alpha}(0) \setminus B_{\alpha}(0))$ embeds into $L^1(\partial B_{1/\alpha}(0) \cup \partial B_\alpha(0))$, it follows that
\[\limsup _{k\to \infty} \|\nu_k - \overline{\nu}_k\|_{L^1(\partial B_{1/\alpha}(0) \cup \partial B_\alpha(0))}\leq C.\]
Using Green's representation formula and general bounds on the Poisson kernel, see \cite[Theorem 1]{Krantz}, it holds that
\[\|\nu_k  - \overline{\nu}_k\|_{L^{\infty}(B_{1/\alpha'}(0)\setminus B_{\alpha'}(0))} < C_{\alpha'}\]
for any $\alpha'>\alpha$. Combining this with \eqref{Wente} and \eqref{limit of gradient of normal goes to 0} implies \eqref{lambda k dont vary too much}. 

By the assumption \eqref{No area and Willmore in annulus} we know that 
\[\lim _{k\to \infty} \int _{B_{1/\alpha'}(0) \setminus B_{\alpha'} (0)}e^{2\lambda_k} \dif x =0.\]
It follows that $\overline{\lambda}_k \to - \infty$. We now want to rescale and shift $\psi_k$ such that the area is uniformly bounded from below and above, thus we set (for some fixed $\alpha'<\frac{1}{2}$)
\[\hat \psi_k  \cqq e^{-\overline{\lambda}_k}  (\psi_k - \psi_k  (0,1/2) ).\]
We set
\[\hat \lambda_k \cqq \lambda_k - \overline{\lambda}_k = \log |\parpar{x_1} {\hat \psi_k} | =\log |\parpar{x_2} {\hat \psi_k} |\]
the conformal factor of $\hat \psi_k$. $\hat \lambda_k \in L^\infty(B_{1/\alpha'}(0)\setminus B_{\alpha'}(0))$ by \eqref{lambda k dont vary too much} and hence $\hat \psi_k$ is uniformly Lipschitz. Because $\hat \psi_k(0,1/2 )$ is fixed, we obtain that the image is bounded, i.e.,	
\begin{equation}
\hat \psi_k (B_{1/\alpha'}(0) \setminus B_{\alpha'}(0))\subset B_{R_{\alpha'}}(0)\label{image bounded}
\end{equation}
for some $R_{\alpha'}>0$. By \cite[(2.42)]{RiviereLectureNotes}, it holds that
\[\Delta \hat \psi_k =  e^{2\hat \lambda_k} \vec H_{\hat \psi_k}.\]
The $L^\infty$-boundedness of $\hat \lambda_k$ and \eqref{area and dn estimate} imply that the right-hand side is uniformly bounded in $L^2(B_{1/\alpha'}(0)\setminus B_{\alpha'}(0))$. Thus $\hat \psi_k$ is uniformly bounded in $W^{2,2}(B_{1/\alpha'}(0)\setminus B_{\alpha'}(0))$ and using weak compactness, we obtain that up to a subsequence (which we do not relabel) 
\begin{equation}
\hat \psi_k \wto \hat \psi_\infty \quad \text{ in } W^{2,2}(B_{1/\alpha'}(0)\setminus B_{\alpha'}(0)). \label{W22 convergence of Psi k}
\end{equation}
In particular, $\hat \psi_\infty$ is Lipschitz in $B_{1/\alpha'}(0)\setminus B_{\alpha'}(0)$ and conformal, where the conformality follows by the strong $W^{1,2}$ convergence of $\hat \psi_k$ to $\hat \psi_\infty$. This convergence also implies that the conformal factors converge, i.e.
\begin{equation}
\hat \lambda_k \to \hat \lambda_\infty \cqq \log |\parpar{x_1}{\hat \psi_\infty}| =\log |\parpar{x_2	}{\hat \psi_\infty}| \quad \text{ in }L^2(B_{1/\alpha'}(0)\setminus B_{\alpha'}(0)). \label{convergence of conformal factors}
\end{equation} 
It follows $\hat \lambda_\infty\in L^\infty(B_{1/\alpha'}(0)\setminus B_{\alpha'}(0))$ and $\hat\psi_\infty$ is a weak immersion. Furthermore, \eqref{W22 convergence of Psi k} implies $\nabla \vec n_{\hat \psi_k} \wto  \nabla \vec n_{\hat \psi_\infty} $ in $L^2(B_{1/\alpha'}(0)\setminus B_{\alpha'}(0))$ and by the lower semicontinuity of the norm with respect to weak convergence
\[\| \nabla \vec n _{\hat \psi_\infty}\|_{L^2(B_{1/\alpha'}(0)\setminus B_{\alpha'}(0))}=0.\]
This holds for any $\alpha' < \frac{1}{2}$, and thus $\vec n_{\hat \psi_\infty} = \vec n_0$ on $\C\setminus \{0\}$ for some $\vec n_0 \in \mb S^2$, where we identify $\C$ with $\R^2$. Hence $\hat \psi_\infty$ is a conformal map from $\C \setminus \{0\}$ into a plane $P_0^2 \subset \R^3$ and after a possible affine transformation, we may assume that $P_0^2 = \R^2 \times \{0\}$ and that $\hat \psi_\infty$ is holomorphic. 

Consider $\pi_k$ the stereographic projection from the sphere $\mb S^2_k$ with radius $e^{-\overline{\lambda}_k}$ and center $(0,0, e^{-\overline{\lambda}_k})$ onto $\R^2\times \{0\}$. We will work with $\pi_k^{-1} \circ \hat \psi _\infty$ instead of $\hat \psi_\infty$. This will, on one hand, ensure the boundedness of the image, and on the other hand ensure that the trace-free second fundamental form is 0 in this region. Notice that for any $R>0$ and $l\in \N$, we have
\begin{equation}
\|\pi_k^{-1}  - \id \|_{C^l(B_R(0))} \leq C_{R,l} e^{\overline{\lambda}_k}\label{stereographic projection}
\end{equation}
and so for any choice of $0<\alpha < \alpha'$ we have, using \eqref{image bounded}, that
\[\|\pi^{-1}_k \circ \hat \psi_\infty - \hat \psi_\infty\|_{C^l(B_{1/\alpha'}(0)\setminus B_{\alpha'}(0))}\to 0.\]
Then, applying \eqref{W22 convergence of Psi k}, we obtain
\begin{equation}
\hat \psi_k - \pi^{-1}_k \circ \hat \psi_\infty \wto 0 \text{ in }W^{2,2}(B_{1/\alpha'}(0)\setminus B_{\alpha'}(0)).\label{W22 convergence of Psi k after stereogr. proj.}
\end{equation}
Fix $r\in (1,2)$ such that it holds
\begin{equation}
\limsup _{k\to \infty} \int _{\partial B_r(0)\cup \partial B_{r/2}(0)} |\nabla^2 \hat \psi_k|^2 \dif l < \infty. \label{good radius}
\end{equation}
This exists by Fubini and Chebyshev, using that $\hat \psi_k$ is bounded in $W^{2,2}(B_{1/\alpha'}(0)\setminus B_{\alpha'}(0))$. As the trace operator is bounded, we obtain from \eqref{W22 convergence of Psi k after stereogr. proj.} that
\[\hat \psi_k - \pi^{-1}_k \circ \hat \psi_\infty \wto 0 \text{ in } W^{3/2, 2}(\partial B_r(0)\cup \partial B_{r/2}(0)).\]
Combining this with \eqref{good radius} gives $\|\hat \psi_k - \pi^{-1}_k \circ \hat \psi_\infty\|_{W^{2,2}(\partial B_r(0)\cup \partial B_{r/2}(0))}\leq C$. Using the Gagliardo-Nirenberg inequality \cite{Mironescu} and the fact that $W^{2,2}(\partial B_r(0)\cup \partial B_{r/2}(0))$ compactly embeds into $W^{1,2}(\partial B_r(0)\cup \partial B_{r/2}(0))$ \cite[Theorem 3.85]{Demengel}, we obtain by interpolation that the embedding $W^{2,2}(\partial B_r(0))\embeds W^{3/2,2}(\partial B_r(0))$ is compact. It follows, using the weak compactness of $W^{2,2}$ that
\begin{equation}
\hat \psi_k - \pi^{-1}_k \circ \hat \psi_\infty \to 0 \text{ in } W^{3/2, 2}(\partial B_r(0)\cup \partial B_{r/2}(0)).\label{strong W22 convergence}
\end{equation}
Similarly, we also obtain $\partial_r\hat \psi_k - \partial_r(\pi_k^{-1}\circ \hat \psi_\infty )\to 0$ in $W^{1/2,2}(\partial B_r(0)\cup \partial B_{r/2}(0))$. We now want to glue together the maps $\hat \psi_k$ and $\pi^{-1}_k \circ \hat \psi_\infty$. For this, we use a biharmonic strip. To obtain $\xi^1_k$ and $\xi^2_k$, we do this with two different boundary conditions. We use the biharmonic strip given by Lemma \ref{lem: biharmonic equation}, i.e., the maps $\tilde{\psi}^1_k$, $ \tilde{\psi}^2_k$ in $W^{2,2}(B_r(0)\setminus B_{r/2}(0), \R^3)$ solving
\begin{equation}
\left \{ \begin{alignedat}{2}
\Delta^2 \tilde{\psi}^1_k &=0 &&\text{ in }B_r(0)\setminus B_{r/2}(0)\\
\tilde{\psi}^1_k &= \pi ^{-1}_k \circ \hat \psi_\infty &&\text{ in } \partial B_r(0)\\
\tilde{\psi}^1_k &= \hat \psi_k &&\text{ in } \partial B_{r/2}(0)\\
\partial_r \tilde{\psi}^1_k &= \partial_r (\pi ^{-1}_k \circ \hat \psi_\infty) &&\text{ in } \partial B_r(0)\\
\partial_r \tilde{\psi}^1_k &= \partial _r \hat \psi_k &&\text{ in } \partial B_{r/2}(0),
\end{alignedat}
\right. \left \{ \begin{alignedat}{2}
\Delta^2 \tilde{\psi}^2_k &=0 &&\text{ in }B_r(0)\setminus B_{r/2}(0)\\
\tilde{\psi}^2_k &= \hat \psi_k &&\text{ in } \partial B_r(0)\\
\tilde{\psi}^2_k &= \pi ^{-1}_k \circ \hat \psi_\infty &&\text{ in } \partial B_{r/2}(0)\\
\partial_r \tilde{\psi}^2_k &= \partial _r \hat \psi_k &&\text{ in } \partial B_r(0)\\
\partial_r \tilde{\psi}^2_k &= \partial_r (\pi ^{-1}_k \circ \hat \psi_\infty) &&\text{ in } \partial B_{r/2}(0).
\end{alignedat}
\right. 
\end{equation}
By Lemma \ref{lem: biharmonic equation} and \eqref{strong W22 convergence}, we obtain \emph{strong} convergence of the biharmonic solutions to $\hat \psi_\infty$, i.e.,
\begin{equation}
\tilde{\psi}^{i}_k - \hat\psi_\infty \to 0 \text{ in }W^{2,2}(B_r(0)\setminus B_{r/2}(0)),\quad i=1,2. \label{convergence of biharmonic strip}
\end{equation}
The fact that $\tilde{\psi}^{1}_k$ and $\tilde{\psi}^2_k$ are weak immersions follows from the slightly stronger estimate $\|\tilde{\psi}^{i}_k - \hat\psi_\infty\|_{C^{1,\alpha}(B_r(0)\setminus B_{r/2}(0))}\to 0$ for $\alpha<1/2$, see \cite[(5.37)]{MondinoRiviere}.  We extend $\tilde{\psi}^1_k$ by $\hat \psi_k$ in $B_{r/2}(0)\setminus B_{\rho_k}(0)$ and by $\pi ^{-1}_k \circ \hat \psi_\infty$ in the complement of $B_r(0)$. Similarly, we extend $\tilde{\psi}^2_k$ by $\pi ^{-1}_k \circ \hat \psi_\infty$ in $B_{r/2}(0)$ and by $\hat \psi_k$ in $B_{1/\rho_k}(0)\setminus B_r(0)$. Due to the matching boundary conditions, these maps are in $W^{2,2}(B_{1/\alpha'}(0)\setminus B_{\alpha'}(0))$. Now set 
\[\zeta_k ^{i}(x) = e^{\overline{\lambda}_k} \tilde{\psi}^{i}_k (x) + \hat \psi_k (0,1/2), \quad i=1,2.\]
Then, it holds 
\begin{equation}
\zeta^{1}_k = \psi_k \text{ in } B_{r/2}(0)\setminus B_{\rho_k}(0),\quad \zeta^{2}_k = \psi_k \text{ in } B_{1/\rho_k}(0)\setminus B_{r}(0). \label{matching boundary conditions}
\end{equation}
Finally, we define $\xi^i_k$, $ i=1,2$ by
\begin{equation}
\xi^1_k(x) = \begin{cases}
f_k \circ \omega \circ \pi(x) &, x\in \pi^{-1}(B_{\lambda t_k}(0)),\\
\zeta^1_k \circ e_{\sqrt{t_ks_k}}^{-1 }\circ \pi(x) &, x\in \mb S^2 \setminus \pi^{-1}(B_{\lambda t_k}(0))
\end{cases}
\end{equation}
and
\begin{equation}
\xi^2_k(x) = \begin{cases}
f_k(x) & ,x\in \Sigma \setminus \omega(B_{s_k/\lambda}(0)),\\
\zeta^2_k \circ e_{\sqrt{t_ks_k}}^{-1}\circ \omega ^{-1} &, x \in \omega(B_{s_k/\lambda}(0)).
\end{cases}
\end{equation}
The conditions \eqref{xi^1 k} and \eqref{xi^2 k} are satisfied by definition. \eqref{matching boundary conditions} ensures that $\xi^1_k \in W^{1,2}(\mb S^2)$, $ \xi^2_k \in W^{1,2}(\Sigma)$. Furthermore, \eqref{tracefree form 1} and \eqref{tracefree form 2} are satisfied due to \eqref{convergence of biharmonic strip}, the assumption \eqref{No area and Willmore in annulus}, and the fact that $\pi^{-1}_k \circ \hat \psi_\infty$ maps into a sphere, so the trace-free part of the second fundamental form vanishes. Branch points may occur at $x_0 \in \mb S^2$ and $a\in \Sigma$.
\end{proof}
Now we want to apply this lemma to the minimizing sequence $(f_k)_k$. We need to make sure that we do not create branch points. Our goal is to estimate the Willmore energy of the immersions $\xi^1_k$, $ \xi^2_k$. By the uniform boundedness of the Willmore energy for the immersions $f_k$ and the concentration of the area (see \cite[(4.14)]{KMR}), we find positive $s_k, t_k \to 0$ satisfying $t_k /s_k \to 0$ such that \eqref{No area and Willmore in annulus} is satisfied and
\begin{equation}
\liminf _{k\to \infty} \mc W(f_k( B_{t_k}(a^*)))\geq 4\pi. \label{4pi willmore energy in tk ball}
\end{equation} 

To prove this, choose $\tilde{t}_k\to 0$ such that 
\[\lim _{k\to \infty} \mc A (f_k(\Sigma \setminus B_{\tilde{t}_k}(a^*))) = 0\quad \text{and}\quad \liminf _{k\to \infty} \mc W(f_k( B_{\tilde{t}_k}(a^*)))\geq 4\pi.\]
This is possible by \cite[Lemma 4.4, (4.11)]{KMR}. After choosing a subsequence, we may assume $\tilde{t}_k < k^{-k-1}$. Define $t_{k,i} = k^i \tilde{t}_k$, $ i =0,\ldots, k$. Then, there is some $i_k \in \{0,\ldots, k-1\}$ such that \[\int_{B_{t_{k,i_k+1}}(a^*)\setminus B_{t_{k,i_k}}(a^*)} |\mb I_{f_k}|^2 \dif \vol_{g_k} < \frac{C}{k}.\] Then, define $t_k = t_{k,i_k}$, $ s_k = t_{k, i_k+1}$. 

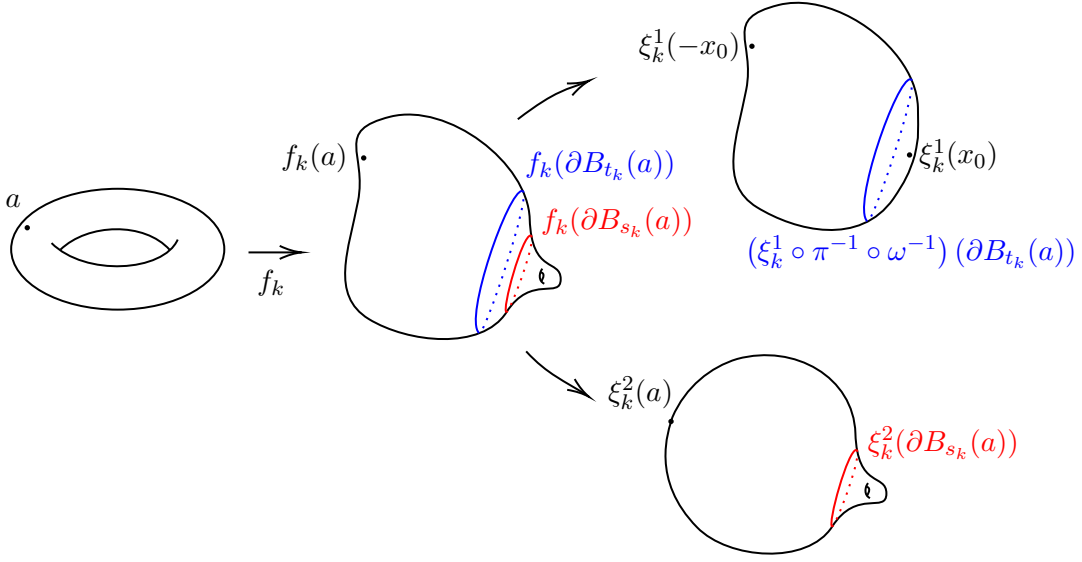
\begin{figure}[H]
\centering
\tikzset{every picture/.style={line width=0.75pt}} 

\begin{tikzpicture}[x=0.75pt,y=0.75pt,yscale=-1,xscale=1]

\draw   (234.15,97.92) .. controls (238.8,73.74) and (221.77,62.52) .. (250.4,55.36) .. controls (279.03,48.21) and (318.96,78.56) .. (319.86,106.8) .. controls (320.76,135.04) and (335.33,126.66) .. (335.33,137.1) .. controls (335.33,147.55) and (320.61,134.3) .. (308.25,153.74) .. controls (295.89,173.17) and (259.02,169.18) .. (239.52,157.15) .. controls (220.03,145.11) and (229.51,122.1) .. (234.15,97.92) -- cycle ;
\draw  [color={rgb, 255:red, 255; green, 0; blue, 0 }  ,draw opacity=1]  (308.25,153.74) .. controls (305.72,152.85) and (316.16,114.55) .. (320.42,115.13) ;
\draw   [color={rgb, 255:red, 0; green, 0; blue, 255 }  ,draw opacity=1] (294.55,164.16) .. controls (286.82,156.81) and (309.2,92.31) .. (316.16,92.7) ;
\draw  [fill={rgb, 255:red, 0; green, 0; blue, 0 }  ,fill opacity=1 ] (236.16,76.14) .. controls (236.16,75.68) and (236.54,75.31) .. (237,75.31) .. controls (237.46,75.31) and (237.83,75.68) .. (237.83,76.14) .. controls (237.83,76.6) and (237.46,76.97) .. (237,76.97) .. controls (236.54,76.97) and (236.16,76.6) .. (236.16,76.14) -- cycle ;
\draw  [color={rgb, 255:red, 0; green, 0; blue, 255 }  ,draw opacity=1, dash pattern={on 0.84pt off 2.51pt}]  (316.16,92.7) .. controls (321.45,93.45) and (296.43,168.37) .. (294.55,164.16) ;
\draw  [color={rgb, 255:red, 255; green, 0; blue, 0}  ,draw opacity=1,dash pattern={on 0.84pt off 2.51pt}]  (320.42,115.13) .. controls (323.54,116.94) and (311.03,151.5) .. (308.25,153.74) ;
\draw    (326.18,131.84) .. controls (324.51,133.91) and (324.06,138.03) .. (327.15,139.39) ;
\draw    (324.96,133.64) .. controls (327.73,133.78) and (326.76,136.49) .. (325.6,138.23) ;
\draw   (427.95,41.73) .. controls (432.59,17.55) and (415.57,6.33) .. (444.2,-0.82) .. controls (464.62,-5.93) and (490.72,7.56) .. (504.38,26.32) .. controls (509.87,33.86) and (513.26,41.78) .. (513.65,50.61) .. controls (513.67,64.89) and (514.42,63.63) .. (512.03,74.99) .. controls (509.34,83.86) and (507.14,88.07) .. (501.4,96.63) .. controls (484.56,120.71) and (452.81,112.99) .. (433.32,100.96) .. controls (413.83,88.92) and (423.31,65.91) .. (427.95,41.73) -- cycle ;
\draw  [color={rgb, 255:red, 0; green, 0; blue, 255 }  ,draw opacity=1]  (488.35,107.98) .. controls (480.61,100.63) and (503,36.12) .. (509.96,36.51) ;
\draw  [fill={rgb, 255:red, 0; green, 0; blue, 0 }  ,fill opacity=1 ] (429.96,19.95) .. controls (429.96,19.49) and (430.33,19.12) .. (430.79,19.12) .. controls (431.25,19.12) and (431.62,19.49) .. (431.62,19.95) .. controls (431.62,20.41) and (431.25,20.78) .. (430.79,20.78) .. controls (430.33,20.78) and (429.96,20.41) .. (429.96,19.95) -- cycle ;
\draw  [color={rgb, 255:red, 0; green, 0; blue, 255 }  ,draw opacity=1,dash pattern={on 0.84pt off 2.51pt}]  (509.96,36.51) .. controls (515.25,37.26) and (490.23,112.18) .. (488.35,107.98) ;
\draw   (391.09,206.58) .. controls (397.68,191.15) and (410.27,181.27) .. (423.41,177.25) .. controls (452.18,169.44) and (481.54,186.41) .. (482.44,214.65) .. controls (483.34,242.89) and (497.91,234.5) .. (497.91,244.95) .. controls (497.91,255.39) and (483.19,242.15) .. (470.83,261.58) .. controls (458.47,281.02) and (419.18,278.2) .. (403.44,261.97) .. controls (380.21,238.15) and (388.78,213.94) .. (391.09,206.58) -- cycle ;
\draw  [color={rgb, 255:red, 255; green, 0; blue, 0 }  ,draw opacity=1]  (470.83,261.58) .. controls (468.3,260.69) and (478.74,222.4) .. (483,222.98) ;
\draw  [color={rgb, 255:red, 255; green, 0; blue, 0 }  ,draw opacity=1,dash pattern={on 0.84pt off 2.51pt}]  (483,222.98) .. controls (486.12,224.79) and (473.61,259.35) .. (470.83,261.58) ;
\draw    (488.76,239.69) .. controls (487.09,241.76) and (486.64,245.88) .. (489.73,247.23) ;
\draw    (487.54,241.49) .. controls (490.31,241.62) and (489.34,244.33) .. (488.18,246.07) ;
\draw  [fill={rgb, 255:red, 0; green, 0; blue, 0 }  ,fill opacity=1 ] (389.57,208.68) .. controls (389.57,208.22) and (389.94,207.85) .. (390.4,207.85) .. controls (390.86,207.85) and (391.24,208.22) .. (391.24,208.68) .. controls (391.24,209.14) and (390.86,209.51) .. (390.4,209.51) .. controls (389.94,209.51) and (389.57,209.14) .. (389.57,208.68) -- cycle ;
\draw  [fill={rgb, 255:red, 0; green, 0; blue, 0 }  ,fill opacity=1 ] (508.63,74.63) .. controls (508.63,74.17) and (509,73.8) .. (509.46,73.8) .. controls (509.92,73.8) and (510.3,74.17) .. (510.3,74.63) .. controls (510.3,75.09) and (509.92,75.46) .. (509.46,75.46) .. controls (509,75.46) and (508.63,75.09) .. (508.63,74.63) -- cycle ;
\draw    (313.91,56.67) .. controls (333.64,41.47) and (338.51,41.89) .. (348.04,37.93) ;
\draw [shift={(349.75,37.19)}, rotate = 155.9] [color={rgb, 255:red, 0; green, 0; blue, 0 }  ][line width=0.75]    (10.93,-3.29) .. controls (6.95,-1.4) and (3.31,-0.3) .. (0,0) .. controls (3.31,0.3) and (6.95,1.4) .. (10.93,3.29)   ;
\draw    (318.09,173.42) .. controls (328.82,185.21) and (337.5,188.35) .. (347.78,195.17) ;
\draw [shift={(349.42,196.28)}, rotate = 214.7] [color={rgb, 255:red, 0; green, 0; blue, 0 }  ][line width=0.75]    (10.93,-3.29) .. controls (6.95,-1.4) and (3.31,-0.3) .. (0,0) .. controls (3.31,0.3) and (6.95,1.4) .. (10.93,3.29)   ;
\draw   (61.02,122.04) .. controls (61.02,105.25) and (84.83,91.64) .. (114.21,91.64) .. controls (143.59,91.64) and (167.41,105.25) .. (167.41,122.04) .. controls (167.41,138.82) and (143.59,152.43) .. (114.21,152.43) .. controls (84.83,152.43) and (61.02,138.82) .. (61.02,122.04) -- cycle ;
\draw    (81.18,118.95) .. controls (97.94,135.88) and (136.92,133.59) .. (144.16,117.19) ;
\draw    (85.59,122.3) .. controls (99.35,108.9) and (127.58,108.19) .. (140.16,122.37) ;
\draw  [fill={rgb, 255:red, 0; green, 0; blue, 0 }  ,fill opacity=1 ] (68.23,111.24) .. controls (68.23,110.78) and (68.6,110.41) .. (69.06,110.41) .. controls (69.52,110.41) and (69.89,110.78) .. (69.89,111.24) .. controls (69.89,111.7) and (69.52,112.07) .. (69.06,112.07) .. controls (68.6,112.07) and (68.23,111.7) .. (68.23,111.24) -- cycle ;
\draw    (179.5,123.65) -- (204,123.65) ;
\draw [shift={(206,123.65)}, rotate = 180] [color={rgb, 255:red, 0; green, 0; blue, 0 }  ][line width=0.75]    (10.93,-3.29) .. controls (6.95,-1.4) and (3.31,-0.3) .. (0,0) .. controls (3.31,0.3) and (6.95,1.4) .. (10.93,3.29)   ;

\draw (194.36,67.17) node [anchor=north west][inner sep=0.75pt]    {$f _{k}( a)$};
\draw (316.16,70.8) node [color={rgb, 255:red, 0; green, 0; blue, 255 }  ,draw opacity=1,anchor=north west][inner sep=0.75pt]    {$f _{k}( \partial B_{t_{k}}( a))$};
\draw (324.02,99.35) node [color={rgb, 255:red, 255; green, 0; blue, 0 }  ,draw opacity=1,anchor=north west][inner sep=0.75pt]    {$f _{k}( \partial B_{s_{k}}( a))$};
\draw (372.9,9.97) node [anchor=north west][inner sep=0.75pt]    {$\xi _{k}^{1}( -x_{0})$};
\draw (426.03,113.61) node [color={rgb, 255:red, 0; green, 0; blue, 255 }  ,draw opacity=1,anchor=north west][inner sep=0.75pt]    {$\left( \xi _{k}^{1} \circ \pi ^{-1} \circ \omega ^{-1}\right)( \partial B_{t_{k}}( a))$};
\draw (487.1,209.69) node [color={rgb, 255:red, 255; green, 0; blue, 0 }  ,draw opacity=1,anchor=north west][inner sep=0.75pt]    {$\xi _{k}^{2}( \partial B_{s_{k}}( a))$};
\draw (357.5,184.72) node [anchor=north west][inner sep=0.75pt]    {$\xi _{k}^{2}( a)$};
\draw (512.79,65.24) node [anchor=north west][inner sep=0.75pt]    {$\xi _{k}^{1}( x_{0})$};
\draw (56.77,92.88) node [anchor=north west][inner sep=0.75pt]    {$a$};
\draw (182.36,131.11) node [anchor=north west][inner sep=0.75pt]    {$f _{k}$};

\end{tikzpicture}
\caption{Intuitively, the map $\xi^1_k$ cuts away the portion of the  surface with the topology and glues in a portion of a sphere, while $\xi^2_k$ cuts away the bubble and glues in an almost complete sphere while keeping the topology.}\label{tikz figure for cut and fill}
\end{figure}
\begin{lemma}
Assume that $s_k$, $ t_k$ are chosen as above. Denote by $\xi^1_k$ and $\xi^2_k$ the immersions obtained from Lemma \ref{Cut and Fill Lemma} applied to our minimizing sequence $(f_k)_k$. Then $\xi^1_k$ respectively $\xi^2_k$ are unbranched immersions in $\mc E_{\mb S^2}$ respectively $\mc E_{\Sigma}$. Furthermore, $\mc A(\xi^1_k) = \mc A(f_k) + o(1)$ as $k\to \infty$ and $\mc T(\xi^1_k) \to R$ where $R = \mc T(f_k)$.
\end{lemma}
\begin{proof}
Let us first show that $\hat \psi_\infty$ is injective. Assume that this is not the case so that we find $x_1,x_2 \in \C\setminus \{0\}$ with $x_1\neq x_2$ and $\hat \psi_\infty(x_1) = \hat \psi_\infty(x_2) = y$. Since $\hat \psi_\infty$ is a smooth immersion, we find a fixed radius $0<s<|x_1-x_2|/2$ and constants $c_1, c_2 > 0$ 
\begin{equation}
c_1 |z-z'| \geq |\hat \psi_\infty(z) - \hat \psi_\infty(z')| \geq c_2 |z-z'| \text{ for all }z, z' \in B_s(x_i) \text{ and }i=1,2.\label{psi infty lipschitz estimate}
\end{equation}
We also fix $n$ such that $n> c_1/c_2$ and define
\[g_{k,i}:B_s(x_i)\setminus B_{s/n}(x_i)\to \R,\quad g_{k,i}(z)\cqq |\hat \psi_k(z) - y|.\]
Because the $\hat \psi_k$ are uniformly Lipschitz on precompact sets by \eqref{lambda k dont vary too much}, there is $c_3$ such that
\[c_3 |z-z'| \geq |\hat \psi_k(z) - \hat \psi_k(z')| \text{ for all }z, z' \in B_s(x_i) \text{ and }i=1,2.\]
$\hat \psi_k$ is bounded by \eqref{image bounded}, thus we get up to a subsequence by the use of the Arzela-Ascoli theorem that 
\begin{equation}
\|\hat \psi_k - \hat \psi_\infty\|_{L^{\infty}(B_s(x_1)\cup B_s(x_2))}\to 0.\label{c0 convergence}
\end{equation}
Let $\eps > 0$ and $k$ large enough such that $\|\hat \psi_k - \hat \psi_\infty\|_{L^{\infty}(B_s(x_1)\cup B_s(x_2))}<\eps$. Then, it follows by \eqref{psi infty lipschitz estimate} that
\begin{equation}
c_1 |z-x_i| +\eps \geq g_{k,i}(z) \geq c_2|z-x_i| - \eps.\label{gki estimate}
\end{equation}
As $|\nabla g_{k,i}|\leq |\nabla \hat \psi_k|$, we obtain, using the coarea formula
\begin{align*}
\int _{B_s(x_i)\setminus B_{s/n}(x_i)} |\nabla \hat \psi_k| \dif z &\geq \int _{B_s(x_i)\setminus B_{s/n}(x_i)} |\nabla g_{k,i}| \dif z = \int_0 ^\infty \mc H^1(g^{-1}_{k,i}(t)) \dif t \\
&\geq \int _{\frac{c_1 s}{n}+\eps} ^{c_2 s -\eps} \mc H^1(g^{-1}_{k,i}(t)) \dif t\geq \frac{1}{c_3} \int _{\frac{c_1 s}{n}+\eps} ^{c_2 s -\eps} \mc H^1(\hat \psi_k(g^{-1}_{k,i}(t))) \dif t \\
&= \frac{1}{c_3} \int _{\frac{c_1 s}{n}+\eps} ^{c_2 s -\eps} \mc H^1(\partial B_t(y) \cap \hat \psi_k(B_s(x_i))) \dif t. \intertext{The last equality follows by \eqref{gki estimate} so that $\hat \psi_k(g_{k,i}^{-1}(t)) = \partial B_t(y) \cap \hat \psi_k (B_s(x_i))$ and $\partial (\hat \psi_k(B_s(x_i))\cap B_t(y))\subset \partial B_t(y)$ for $t\in (\frac{c_1 s}{n}+\eps, c_2 s -\eps)$. Suppose that $\beta\in \R$ is such that the integrand on the right-hand side is pointwise bounded from below by $\beta t$. Then, the right-hand side is bounded from below by}
&\geq \frac{\beta}{2c_3}\left (\left (c_2s-\eps\right )^2 -\left (\frac{c_1 s}{n} + \eps\right )^2\right ).
\end{align*}
Because the left-hand side is bounded in $k$ by the strong $L^2$ convergence of the gradients, we get a uniform bound for $\beta$  from above (for $\eps$ sufficiently small). Hence, we fix $\beta$ large enough and $\eps$ sufficiently small such that there is a set of positive measure of radii $\rho_\eps \in (c_1 s/n , c_2 s)$ such that
\[ \mc H^1(\partial B_{\rho_\eps}(y) \cap \hat \psi_k(B_s(x_i))) \leq \beta \rho_\eps, \quad i=1,2.\] 
We choose $\rho_\eps$ such that the divergence theorem, Lemma \ref{Divergence Theorem}, holds for $I_0\circ (f_k-y)$. We now set $\Sigma_i =  B_{\rho_\eps}(y) \cap \hat \psi_k(B_s(x_i))$. Recalling that the $\hat\psi_k$ are injective, we see that the assumptions of Lemma \ref{lem:Simon 8pi - C estimate} are satisfied and we obtain
\[\mc W(f_k) \geq 8\pi - C \beta \frac{\eps}{\rho_\eps},\]
which converges to $8\pi$ as $k\to \infty$ and $\eps \to 0$ as $\rho_\eps$ is uniformly bounded by $c_1s/n$. This is a contradiction to the $8\pi -\delta$ bound for the Willmore energy of $f_k$. Consequently, $\hat \psi_\infty$ is injective.

Let us now show that $\hat \psi_\infty$ cannot be bounded at the origin. We argue by contradiction, so assume that $\hat \psi_\infty$ is bounded around $0$ so that one can extend it to a holomorphic map in $\C$. We then claim that $\limsup _{k\to \infty} \mc W(\xi^2_k) \leq 4\pi-\delta$, which is impossible by the Willmore inequality\footnote{This still holds in this case, even though $\xi^2_k$ might be branched.}.	\medskip

\noindent\emph{Proof of the claim.} By \eqref{4pi willmore energy in tk ball} and the fact that $\limsup _{k\to \infty} \mc W(f_k)\leq 8\pi- \delta$, we have that
\[\limsup _{k\to \infty}\mc W(\xi^2_k \vert _{\Sigma \setminus \omega (B_{r\sqrt{s_kt_k}}(0))}) \leq  4\pi - \delta .\]
By the strong convergence in \eqref{convergence of biharmonic strip}, we know that the biharmonic strip does not carry Willmore energy in the limit, i.e.,
\[\limsup _{k\to \infty}\mc W(\xi^2_k \vert _{\omega (B_{r\sqrt{s_kt_k}}(0) \setminus B_{\frac{r\sqrt{s_kt_k}}{2}}(0))})=0 .\]
 Using \eqref{stereographic projection} and the assumption that $\hat \psi_\infty \vert _{B_{r/2}(0)}$ is bounded, we have 
\[\lim_{k\to \infty}\mc W(\xi^2_k \vert _{\omega(B_{\frac{r\sqrt{s_kt_k}}{2}}(0))}) =0.\]
Together, this implies the claim. 

Using the injectivity of $\hat \psi_\infty$, we must have that the pole at $0$ is of order 1. Furthermore, by the injectivity we have that $\hat \psi_\infty$ is bounded at $\infty$ so that $z \mapsto \hat \psi_\infty(1/z)$ is an injective entire map, and thus of the form $az+b$ for some $a,b\in \C$. It follows $b=0$ and $\hat \psi_\infty(z) = \frac{1}{az}$. This also implies directly that $\xi^1_k$ and $\xi^2_k$ are unbranched which finishes the proof. \qedhere
\end{proof}


We now bring the case of diverging conformal factors to a contradiction.  

 Notice that due to the pole of order 1, we have
\[\lim _{k\to \infty} \mc W( \pi^{-1}_k\circ \hat \psi_\infty \vert _{B_{r/2}(0)}) = 4\pi.\]
It follows, estimating the Willmore energy of $\xi^2_k$ from below by $\beta_g  $, that
\begin{equation}
\mc W(\xi^2_k) = 4\pi + o(1) + \mc W(f_k \vert _{\Sigma \setminus \omega(B_{s_k/\lambda}(0))})\geq \beta_g  . \label{Willmore energy xi 2 k}
\end{equation}
Notice that $\xi^1_k \vert _{\mb S^2 \setminus \pi^{-1}(B_{r \sqrt{t_ks_k}}(0))}$ immerses a segment of a round sphere of radius 1 with vanishing area\footnote{Due to the fact that $\hat \psi_\infty$ is bounded near infinity.}, which implies that the Willmore energy of this disk is vanishing. Because $\xi^1_k \in \mc E_{\mb S^2}$, we can estimate its Willmore energy:
\begin{equation}
\mc W(\xi^1_k) = o(1) + \mc W(f_k \vert _{\omega(B_{s_k/\lambda}(0))}) \geq \beta_0(\mc T(\xi^1_k)). \label{Willmore energy xi 1 k}
\end{equation}
Summing \eqref{Willmore energy xi 2 k} and \eqref{Willmore energy xi 1 k} we obtain
\[\mc W(f_k) \geq \beta_g  + \beta_0(\mc T(\xi^1_k)) + o(1) - 4\pi.\]
Furthermore, it holds that $\mc A (\xi^1_k) = \mc A (f_k)+o(1)$ as $\mc A (\xi^1_k \vert _{\mb S^2 \setminus \pi^{-1}(B_{r \sqrt{t_k s_k}}(0)}) \to 0$, we obtain by Cauchy-Schwarz that $\mc T(\xi^1_k) -\mc T(f_k) \to 0$ as $k\to \infty$. Using the continuity of $\beta_0(R)$ given by Theorem \ref{thm: genus 0 case}, we obtain that
\[\liminf _{k\to \infty} \mc W(f_k) \geq \beta_g +\beta_0(R) - 4\pi.\]
However, this is a contradiction to our choice of $R$, see \eqref{I} and Remark \ref{remark: I}. 

In conclusion, this shows that the conformal factors of the minimizing sequence do not diverge. As we have already shown the existence of a smooth minimizer under the constraint $\mc T$ in the case of bounded conformal factors, this proves the existence of minimizers of the total mean curvature constrained problem.

Finally, we adress the continuity of $\beta_g \vert _I$. The upper semicontinuity follows by Corollary \ref{cor:usc}. The lower semicontinuity of $\beta_g\vert _I$ follows by using the maps $f_{R_n}^g$ instead of $f_k$ in the previous arguments, where $\limn R_n = R\in I$. Then, as before, we show that $f_{R_n}^g \wto f_\infty$ in $W^{2,2}(\Sigma,\R^3)$ where $f_\infty$ is another smooth minimizer of the Willmore energy under the constraint $\mc T(f_\infty) = R$, and the lower semicontinuity of $\mc W$ under weak $W^{2,2}$-convergence shows lower semicontinuity of $\beta_g\vert _I$. This finishes the proof of Theorem \ref{thm:KMR} under the assumption stated in Remark \ref{remark: I}.

\section{A Willmore energy estimate for the connected sum under a total mean curvature constraint} \label{sec:BauerKuwert}
In the last chapter, we proved Theorem \ref{thm:KMR} which stated that for a fixed parameter $R \in I$, where $I$ was defined in \eqref{I}, there exists a smooth minimizer of the Willmore energy under the total mean curvature constraint. In Remark \ref{remark: I} however, we assumed certain energy bounds to hold for $R\in I$. The goal of this chapter and Chapter \ref{sec:8piconstruction} is to prove the assumptions on the set $I$ from Theorem \ref{thm:KMR}. More precisely, we want to show that for $R \in (0,\sqrt{2}\mc T (\mb S^2))\setminus\{\mc T(\mb S^2)\}$ 
\begingroup
\allowdisplaybreaks
\begin{align}
\inf _{\substack{f \in \mc E_\Sigma \\ \mc T(f) = R}} \mc W(f)&< \beta_g   + \mc W(f_R^0 )- 4\pi,  \label{beta_g  bound}\\
\inf _{\substack{f \in \mc E_\Sigma \\ \mc T(f) = R}} \mc W(f) &< 8\pi, \label{8pi bound}\\
\inf _{\substack{f \in \mc E_\Sigma \\ \mc T(f) = R}} \mc W(f) &< \omega_g  . \label{omega_g  bound}
\end{align}
\endgroup
Here $f_R^0$ is the minimizer of the genus 0 case given by Theorem \ref{thm: genus 0 case}. Recall that in \eqref{intro:betag}, $\beta_g $ was defined as the infimum of $\mc W(f)$ where $f$ is a smooth immersion of a closed surface of genus $g$ and that in \eqref{intro:omegag}, $\omega_g $ was defined by
\[\omega_g   \cqq \min \left \{4\pi + \sum _{i=1}^p (\beta _{g_i} - 4\pi) , \;  g= g_1 + \ldots + g_p,\; 1\leq g_i < g\right \}.\]
The inequality \eqref{omega_g  bound} follows from \eqref{8pi bound} by the work of \textsc{Marques} and \textsc{Neves} \cite[Theorem B]{MarquesNeves}, who showed that $\beta_g  \geq 2\pi^2$ for $g\geq 1$ which implies $\omega_g  \geq 8\pi$ for $g\geq 1$. In this chapter, we will show the first inequality \eqref{beta_g  bound}. In Chapter \ref{sec:8piconstruction}, we will prove \eqref{8pi bound}. \textsc{Simon}\cite{Simon} and \textsc{Bauer} and \textsc{Kuwert} \cite{BauerKuwert} proved the existence of smooth minimizers of the unconstrained Willmore problem, i.e., there exists a smooth embedding of a genus $g$ surface $\Sigma$, $f_g:\Sigma\to \R^3$ such that 
\[\mc W(f_g) = \beta_g .\]
We will construct an immersion of $f:\Sigma \to \R^3$ that satisfies the constraint $\mc T(f) = R$ and furthermore $\mc W(f) < \mc W(f_g) + \mc W(f_R^0) - 4\pi = \beta_g  + \beta_0(R) - 4\pi$. Then, this implies \eqref{beta_g  bound}.
The goal of this section is the following. Given $f_1:\Sigma_1 \to \R^3$ a smooth embedding and a smooth immersion $f_2: \Sigma_2 \to \R^3$, both of which are not immersed round spheres, we want to construct an immersion $f:\Sigma_1 \# \Sigma_2\to \R^3$ of the connected sum $\Sigma_1 \# \Sigma_2$ such that
\[\mc W(f) < \mc W(f_1) + \mc W(f_2) - 4\pi \]
and $\mc T(f) = \mc T(f_2)$. We will use the construction done by \textsc{Bauer} and \textsc{Kuwert} \cite{BauerKuwert}, who inverted $f_1$ at a nonumbilic point in the image and glued the inverted surface into a large copy of $f_2$. The gluing was done using a biharmonic graph. We will adapt this construction and show that under slight modifications, we can ensure that $\mc T(f) = \mc T(f_2)$ still holds. This will be done by either  gluing the surfaces from the correct side and applying a suitable Möbius transformation to the glued surface, or in the case that the total mean curvature of the inverted surface vanishes, we will apply Lemma \ref{lem:First variation of T}. Let us first recall the construction done in \cite[Section 3]{BauerKuwert}.
\subsection{The original construction}
Let $f_i:\Sigma_i \to \R^3$, $ i=1,2$ be smooth immersions and $\Sigma_i$ are closed surfaces. Assume $p_i \in \Sigma_i$, $ i=1,2$ and
\[f_i^{-1}(0) = \{p_i\}, \quad \im Df_i(p_i) = \R^2 \times \{0\}\quad \text{for}\quad  i=1,2.\]
We can choose local graph representations $(z,u(z))$ for $f_1$ and $(z,v(z))$ for $f_2$ where
\begin{align*}
u(z) &= p(z) + \phi(z) \text{ on } D_\rho (0), \quad\text{where }p(z)= \frac{1}{2} P(z,z), \\
v(z) &= q(z) + \psi(z) \text{ on } D_\rho (0), \quad\text{where }q(z)= \frac{1}{2} Q(z,z)
\end{align*}
for some $\rho >0$. $P = \mb I_{f_1}(p_1)$ respectively $Q = \mb I_{f_2}(p_2)$ are the bilinear forms associated to the second fundamental forms of $f_1$ and $f_2$ at $p_1$ respectively $p_2$. The Taylor expansion gives the estimates
\begin{align}
\begin{split}
|z|^{-3} |\phi(z)| + |z|^{-2}|D\phi(z)| + |z|^{-1}|D^2\phi(z)| \leq C \quad \forall z\in D_\rho(0)\setminus \{0\},\label{Error term estimates}\\
|z|^{-3} |\psi(z)| + |z|^{-2}|D\psi(z)| + |z|^{-1}|D^2\psi(z)| \leq C \quad \forall z\in D_\rho(0)\setminus \{0\}.
\end{split}
\end{align}
We decompose $P = \frac{1}{2}\tr (P) \id + P^\circ$ into its pure trace part and its trace-free part $P^\circ$, such that it holds 
\[P(z,z) = \frac{1}{2} (\tr P)|z|^2 + P^\circ (z,z).\]
We may assume that the points $f_i(p_i)$ are nonumbilic by the assumption that $f_1$ and $f_2$ both do not immerse round spheres. Due to \cite[Lemma 4.5]{BauerKuwert}, after applying an orthogonal transformation, we can also assume that 
\begin{equation}
\langle P^\circ, Q^\circ \rangle >0,\label{BK:Frobenius Inner Product Positive}
\end{equation} where $\langle \cdot,\cdot \rangle$ is the Frobenius inner product. Consider the immersion 
\[f_1^\circ:\Sigma_1 \setminus \{p_1\} \to \R^3,\quad f_1^\circ (p) =   I\circ f_1 (p) - \frac{1}{2} \tr(P) e_3\]
where $I(x) = \frac{x}{|x|^2}$ for $0\neq x \in \R^3$.
For some sufficiently large $R<\infty$, the surface $f_1^\circ$ has a graph representation $(\zeta, u^\circ (\zeta))$ at infinity, see \cite[Lemma 2.3]{BauerKuwert}, given by
\[u^\circ (\zeta) = p^\circ (\zeta) + \phi^\circ (\zeta) \text{ on } \R^2 \setminus D_R(0), \quad \text{where}\quad p^\circ(\zeta) = \frac{1}{2} P^\circ \left (\frac{\zeta}{|\zeta|}, \frac{\zeta}{|\zeta|}\right ) \quad \text{and }\]
\[|\zeta| |\phi^\circ (\zeta)| + |\zeta|^2 |D\phi^\circ (\zeta)| + |\zeta|^3 |D^2\phi^\circ (\zeta)| \leq C \quad \forall \zeta \in \R^2 \setminus D_R(0).\]
The idea is to rescale both $f_1^\circ$ and $f_2$ by appropriate factors. To fix notation, we will define for $w:\Omega \to \R^k$, $ \Omega \subset \R^2$ and $\lambda > 0$ the rescaled function $w_\lambda$ 
\[w_\lambda:\Omega_\lambda \cqq \{z\in \R^2,\; z/\lambda \in \Omega \} \to \R^k, \quad w_\lambda(z) \cqq \lambda w(z/\lambda).\]
We rescale $f_1^\circ$ by some small $\alpha>0$ and $f_2$ by $1/\beta = 1/(t\alpha)$ for some fixed $t>0$. The local graph representation of $f_{1,\alpha}^\circ$ respectively $f_{2,1/\beta}$ are given by $u_\alpha^\circ$ respectively $v_{1/\beta}$. We also choose $\gamma$ such that $0<\alpha, \beta \ll \gamma \ll 1$. For given $\alpha >0$, fix a function $\eta \in C^\infty(\R)$ with the properties
\begin{equation}
\eta(s) = \begin{cases}
0& \text{ for $s\leq (1/4) \sqrt{\alpha}$},\\
1& \text{ for $s \geq (3/4) \sqrt{\alpha}$},
\end{cases} \label{BK:eta}
\end{equation}
\[|\eta| + \sqrt{\alpha} |\eta'| + \alpha |\eta''| \leq C \quad \text{with $C$ independent of $\alpha$}.\]
We will modify  $u^\circ_\alpha$ and $v_{1/\beta}$ by defining, for $r=|z|$
\begin{equation}
w(z) \cqq\begin{cases}
p^\circ _\alpha(z) + \eta(\gamma-r)\phi^\circ_\alpha (z) & \text{for $\alpha R < r < \gamma$},\\
q_{1/\beta}(z) + \eta(r-1)\psi_{1/\beta}(z)& \text{for $1\leq r < \frac{\rho}{\beta}$}.
\end{cases} \label{BK:w}
\end{equation}
It holds that $w= u_\alpha^\circ$ for $r\leq \gamma - (3/4) \sqrt{\alpha} $ and $w= v_{1/\beta}$ for $r\geq 1 + (3/4) \sqrt{\alpha}$ by \eqref{BK:eta}. On the remaining annulus $D_1(0)\setminus D_\gamma(0)$, we define the interpolation $w$ as the unique solution of the equation
\[\Delta^2 w = 0 \quad \text{on $D_1(0)\setminus D_\gamma(0)$}\]
subject to the Dirichlet boundary conditions
\begin{alignat*}{3}
w&= p^\circ_\alpha,& \partial_r w&= \partial_r p^\circ_\alpha  &\text{ on } |z|&=\gamma,\\
w&= q_{1/\beta} ,\quad  &\partial_r w&= \partial_r q_{1/\beta} &\quad\text{ on } |z|&=1.
\end{alignat*}
Now, let us define the glued surface. We let $U$ be the complement in $\Sigma_1$ of the preimage of the set $\{z\in \R^2,\; \gamma - \sqrt{\alpha} \leq |z| < \infty\}$ under the map $\alpha \pi \circ f_1^\circ$. Here, $\pi:\R^3\to \R^2$ is the orthogonal projection map onto the first two coordinates. Furthermore, $V$ is defined to be the complement in $\Sigma_2$ of the preimage of the set $\{z\in \R^2, \;|z|\leq 1+\sqrt{\alpha}\}$ under the map $(1/\beta) \pi \circ f_2$. Finally, we set $W\cqq D_{\rho/\beta}(0)\setminus D_{\alpha R}(0)$. Then, the connected sum $\Sigma$ is defined as
\[\Sigma \cqq (U\cup V \cup W)/ \sim,\]
where $\sim$ is the equivalence relation 
\begin{equation}
  \begin{alignedat}{2}
    p\sim z &= \alpha \pi(f_1^\circ (p)),  \quad &&\text{where}\quad  p\in U, z\in W\\
   q\sim z &= (1/\beta) \pi(f_2(q)), \quad &&\text{where}\quad  q\in V, z\in W.
  \end{alignedat}
\label{BK:sim}
\end{equation}
Topologically, $\Sigma$ is the connected sum $\Sigma_1 \# \Sigma_2$. Finally, the immersion $f:\Sigma \to \R^3$ can be defined by
\[f:\Sigma \to \R^3, \quad f(x) = \begin{cases}
\alpha f_1^\circ (p) &\text{for } x=p\in U \subset \Sigma_1,\\
(1/\beta) f_2(q)&\text{for } x=q \in V \subset \Sigma_2,\\
(z,w(z)) &\text{for } x=z\in W=D_{\rho/\beta}(0) \setminus D_{\alpha R}(0).
\end{cases}\]
Notice that by \eqref{BK:eta}, \eqref{BK:w} and \eqref{BK:sim}, this is well-defined. 
\begin{figure}[H]
    \centering
    \tikzset{every picture/.style={line width=0.75pt}} 
\begin{tikzpicture}[x=0.75pt,y=0.75pt,yscale=-1,xscale=1]

\draw[-stealth] (100,128) -- (574,128) node[pos=1, above right] {$r$};
\draw    (100,125) -- (100,131) node [above, inner sep=0.75pt, yshift=10pt] {$0$};
\draw    (133.6,125) -- (133.6,131)node [above, inner sep=0.75pt, yshift=10pt] {$\alpha R$} ;
\draw    (196.6,125) -- (196.6,131)node [above, inner sep=0.75pt, yshift=10pt] {$\gamma-\sqrt{\alpha}$} ;
\draw    (241.5,125) -- (241.5,131) node [above, inner sep=0.75pt, yshift=10pt] {$\gamma$} ;
\draw    (296.4,125) -- (296.4,131) node [above, inner sep=0.75pt, yshift=10pt] {$1$};
\draw    (341.6,125) -- (341.6,131)node [above, inner sep=0.75pt, yshift=10pt] {$1+\sqrt{\alpha}$} ;
\draw    (483.2,125) -- (483.2,131)node [above, inner sep=0.75pt, yshift=10pt] {$\rho/\beta$} ;

\draw    (133.6,249.15) .. controls (186.5,244.15) and (195.35,220.68) .. (247.85,200.93) ;
\draw  [dash pattern={on 0.84pt off 2.51pt}]  (191.6,223.68) .. controls (234.6,185.43) and (473,134.6) .. (570,133.6) ;
\draw  [dash pattern={on 0.84pt off 2.51pt}]  (100,128) .. controls (200.57,127.7) and (308.6,164.43) .. (346.6,186.18) ;
\draw    (290.88,167.68) .. controls (335.68,177.68) and (439.28,249.68) .. (484.48,278.48) ;
\draw [dashed, color={rgb, 255:red, 255; green, 0; blue, 0 }  ,draw opacity=1  ]   (196.6,227.26) .. controls (228.71,208.48) and (211.38,209.26) .. (241.5,198.43) ;
\draw [dashed, color={rgb, 255:red, 255; green, 0; blue, 0 }  ,draw opacity=1  ]   (296.4,163.97) .. controls (323.99,175.58) and (315.85,177.18) .. (341.6,189.18) ;
\draw [dashed, color={rgb, 255:red, 0; green, 0; blue, 255 }  ,draw opacity=1 ]   (241.5,198.43) .. controls (259.64,190.98) and (267.74,153.33) .. (296.4,163.97) ;
\draw (156.5,222.4) node [anchor=north west][inner sep=0.75pt]    {$u_{\alpha }^{\circ }$};
\draw (482.33,145.07) node [anchor=north west][inner sep=0.75pt]    {$p_{\alpha }^{\circ }$};
\draw (158.67,137.93) node [anchor=north west][inner sep=0.75pt]    {$q_{1/\beta }$};
\draw (437.67,235.77) node [anchor=north west][inner sep=0.75pt]    {$v_{1/\beta }$};
\draw (247.5,163.6) node [anchor=north west][inner sep=0.75pt]  [color={rgb, 255:red, 95; green, 175; blue, 253 }  ,opacity=1 ]  {$w$};

\end{tikzpicture}
\caption{The graphic (not to scale) illustrates the 1-dimensional profile curve of the connected sum construction from this section. The dotted curves denote $q_{1/\beta}$ and $p_{\alpha}^\circ$. In the region $\gamma-\sqrt{\alpha}<r<\gamma$ respectively $1<r<1+\sqrt{\alpha}$, the interpolation between $u_\alpha^\circ$ and $p_\alpha^\circ$ respectively $q_{1/\beta} $ and $v_{1/\beta}$ is depicted in red. Finally, in the region $\gamma<r<1$, the biharmonic solution $w$ with the corresponding boundary condition is shown in blue.}
\end{figure}
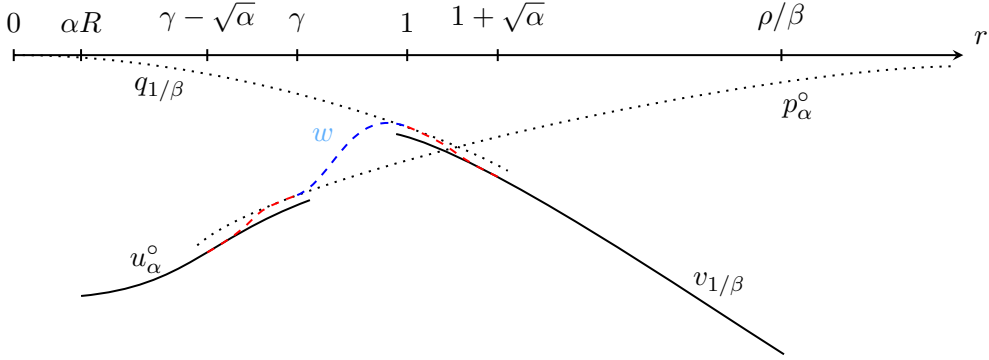

\subsection{The total mean curvature ratio of the connected sum}
From \cite[Lemma 2.1]{BauerKuwert}, we see that for $f:\Omega \to \R^3, f(z) = (z,u(z))$ a 2-dimensional graph such that $|Du|\leq 1$, the following Lemma holds.
\begin{lemma}[Mean curvature versus Laplacian]\label{mean curvature vs laplacian}
Let $f:\Omega \to \R^3$, $ f(z) = (z,u(z))$ be a 2-dimensional graph in $\Omega\times \R \subset\R^3$ with first fundamental form $G= (g_{ij})_{1\leq i,j\leq 2}$ and scalar mean curvature $H$. If $|Du|\leq 1$ in $\Omega$, then we have
\[|H \sqrt{\det G} - \Delta u| \leq C |Du| |D^2u|.\]
\end{lemma}
\begin{proof}
From \cite[(2.13),(2.14)]{BauerKuwert}, we have a decomposition $H = H^{\text{hor}} + H^{\text{ver}}$ such that by \cite[(2.15)]{BauerKuwert}
\[|H^{\text{hor}} | \leq C |Du| |D^2u|\]
and \cite[(2.17)]{BauerKuwert}
\[|H^{\text{ver}} - \Delta u |\leq C|Du|^2 |D^2u| \leq C|Du| |D^2u|.\]
It follows that $|H - \Delta u| \leq C |Du||D^2u|$. Because $1\leq \sqrt{\det G} \leq 1 + |Du|^2$ by \cite[(2.9)]{BauerKuwert}, we have 
\begin{align*}
|H \sqrt{\det G} - \Delta u| &\leq |H - \Delta u|\sqrt{\det G}  +  |\Delta u| (\sqrt{\det G} - 1)\\
&\leq C |Du||D^2u|.\tag*\qedhere
\end{align*}
\end{proof}
Next, we want to estimate the difference of $\mc T(f)$ and $\mc T(f_2)$. For that, we need to find asymptotic estimates of the total mean curvature of the region of $f_2$ that we cut out and the biharmonic strip which we insert. 
\begin{lemma}[Behavior of removed disk of $f_2$ and biharmonic strip] \label{lem:Behavior of removed disks}
Let $\beta = t \alpha$ for fixed $t>0$. Define $e\cqq \frac{1}{2}(Q_{11} + Q_{22}) $. We have as $\gamma,\alpha \to 0$ that
\begin{align}
\int _{D_{1+\sqrt{\alpha}}\setminus D_{1}} \Delta w \dif x  +\int _{D_{\gamma}\setminus D_{\gamma-\sqrt{\alpha}}} \Delta w \dif x&= \mc O_{t, \gamma} (\alpha ^{3/2}),\label{BK:Integral in interpolation strip}\\
\int _{D_{1}\setminus D_{\gamma}} \Delta w \dif x   &=  2\beta \pi e + \mc O_t(\gamma^2 \alpha)\label{BK:Integral in biharmonic function},\\
\int_{D_{1+\sqrt{\alpha}}} \Delta v _{\frac{1}{\beta}} \dif x &= 2\beta \pi e  + \mc O_t(\alpha^{3/2})\label{BK:Integral of removed disk},
\end{align}
where the constants in $\mc O_{t,\gamma}$ respectively $\mc O_{t}$ depend on $t$ and $\gamma$ respectively $t$.
\end{lemma}
\begin{proof}
First, we estimate the contribution of the interpolation strip. We have by equations \cite[(4.21),(4.22)]{BauerKuwert} for $0 < \alpha \ll \gamma$
\begin{alignat*}{2}
 |D^2 w(z)| &\leq C(\gamma) \alpha \quad &&\text{for $\gamma - \sqrt{\alpha}\leq r \leq \gamma$,}\\
|D^2 w(z)| &\leq C(t) \alpha \quad &&\text{for $1\leq r \leq 1+\sqrt{\alpha}$}.
\end{alignat*}
Integration in the domains $D_{1+\sqrt{\alpha}}\setminus D_1$ and $D_{\gamma}\setminus D_{\gamma - \sqrt{\alpha}}$ yields \eqref{BK:Integral in interpolation strip}. 

To prove the second equality, we write $w$ as a sum of biharmonic functions as in \cite[(3.29),(3.30)]{BauerKuwert} with coefficients as in \cite[(3.35)]{BauerKuwert}, i.e.
\[w(r,\theta) = g(r) \cos(2\theta) + h(r) \sin(2\theta) + k(r)\]
for certain functions $g(r),h(r), k(r)$. Because of
\[\int_{\gamma}^1 \int_0^{2\pi} g(r) \cos (2\theta) r \dif \theta \dif r = \int_{\gamma}^1 \int_0^{2\pi} h(r) \sin (2\theta) r \dif \theta \dif r = 0\]
and \cite[(3.32)]{BauerKuwert}, we see that only $k(r)$ contributes to the integral, namely by \cite[(3.33)]{BauerKuwert}
\[\int_{D_1\setminus D_\gamma} \Delta w \dif x = \int _\gamma ^1 2\pi r (C_2 + C_4 + 2 C_4 \log r)\dif r = C_2 \pi (1-\gamma^2) - C_4 2\gamma^2 \pi \log(\gamma).\]
As $\gamma,\alpha \to 0$, we have that $C_2 - 2\beta e = \mc O_t(\gamma^2 \alpha)$ and $C_4=\mc O_t(\gamma \alpha)$, see \cite[(3.35)]{BauerKuwert}, which implies \eqref{BK:Integral in interpolation strip}.  

To prove the last equation, we notice that 
\begin{align*}\int _{D_{1+\sqrt{ \alpha}}} \Delta (v_{\frac{1}{\beta}}) \dif x &=  \frac{\beta}{\beta^2}\int_{D_{\beta(1+\sqrt{\alpha})}} \Delta v \dif x.\intertext{Using the differentiability of $\Delta v$, this equals}
&=\beta \pi  \Delta v(0) + \mc O_t(\alpha^{3/2}).
\end{align*}
Now by definition of $Q$, we have $\Delta v(0) = Q_{11} + Q_{22} = 2e,$ which proves \eqref{BK:Integral of removed disk} and finishes the proof. 
\end{proof}
\begin{lemma} \label{Difference of H-integrals}
Let 
\begin{align*}
U_\alpha  &= \Sigma_1 \setminus (\pi \circ f_1^\circ )^{-1} \left \{z\in \R^2,\; \frac{\gamma -\sqrt{\alpha}}{\alpha}<|z|<\infty\right \}.
\end{align*}
Then, as $\alpha \to 0$, $\gamma \to 0$, and again denoting $\beta = t\alpha$ for some fixed $t>0$, it holds that
\begin{equation}
\mc T(f) -\mc T(f_2) = \alpha \beta \mc A(f_2)^{-1/2}\left (\int _{U_\alpha } H_{f_1^\circ} \dif \vol_{f_1^\circ} + \mc O_t(\gamma^2)\right ) + \mc O_{t,\gamma}(\alpha ^{5/2}).\label{BK:Estimate T}
\end{equation}
\end{lemma}
\begin{proof}
By \cite[(4.13), (4.21)-(4.26)]{BauerKuwert}, we can apply Lemma \ref{mean curvature vs laplacian} and Lemma \ref{lem:Behavior of removed disks}, as $\alpha \to 0,\gamma \to 0$ to obtain
\begin{align}
&\quad\,\int _{\Sigma} H_f \dif \vol_f - \int _{\Sigma_2} H _{f_{2,\frac{1}{\beta}}} \dif \vol _{f_{2,\frac{1}{\beta}}} \\
&= \int _{D_{1+\sqrt{\alpha}}\setminus D_{\gamma-\sqrt{\alpha}}} H_f \dif \vol_f + \alpha \int _{U_\alpha } H_{f_1^\circ} \dif \vol_{f_1^\circ}  - \int _{\Sigma_2 \setminus V} H_{f_{2,\frac{1}{\beta}}}\dif \vol_{f_{2,\frac{1}{\beta}}}\notag\\
 &= \int _{D_{1}\setminus D_{\gamma}} \Delta w \dif x + \alpha  \int _{U_\alpha } H_{f_1^\circ} \dif \vol _{f_1^\circ}\notag\\
&\quad  \,  -\int _{D_{1+\sqrt{\alpha}}} \Delta v_{\frac{1}{\beta}} \dif x + \mc O_{t,\gamma}(\alpha^{3/2}) \notag\\
 &= \alpha \left (\int _{U_\alpha } H_{f_1^\circ} \dif \vol_{f_1^\circ} + \mc O_t(\gamma^2)\right ) + \mc O_{t,\gamma}(\alpha^{3/2}).\label{BK:Int H estimate}
\end{align}
By \cite[(3.18)]{MondinoScharrerInequality}, we know that
\[\sqrt{\mc A(f)} = \frac{1}{\beta}\left (\sqrt{\mc A(f_2)} + \mc O_{t,\gamma}(\alpha ^{5/2})\right ).\]
Thus
\begin{align*}
\mc T(f) &= \frac{\frac{1}{\beta}\left (\int _{\Sigma_2} H_{f_2} \dif \vol_{f_2} + \alpha \beta \left (\int _{U_\alpha } H_{f_1^\circ} \dif \vol_{f_1^\circ} + \mc O_t(\gamma^2)\right ) + \mc O_{t,\gamma}(\alpha^{5/2})\right )}{\frac{1}{\beta}\left (\sqrt{\mc A(f_2)} + \mc O_{t,\gamma}(\alpha ^{5/2})\right )}\\
&= \mc T(f_2) + \alpha \beta \mc A(f_2)^{-1/2} \left (\int _{U_\alpha } H_{f_1^\circ} \dif \vol_{f_1^\circ} + \mc O_t(\gamma^2)\right ) + \mc O_{t,\gamma}(\alpha ^{5/2}).\tag*\qedhere
\end{align*}
\end{proof}
The Willmore energy of the connected sum immersion was estimated in \cite[Lemma 4.4]{BauerKuwert}.
\begin{lemma} \label{lem: Bauer-Kuwert Willmore estimate}
Taking $\beta = t\alpha$ for any $t>0$, and letting $\alpha$ tend to 0, there holds
\begin{equation}
\begin{split}
&\quad \, \mc W(f) - (\mc W(f_1) + \mc W(f_2) - 4\pi) \\
&= \pi \alpha^2( |P^\circ |^2 - t\langle P^\circ, Q^\circ \rangle + \mc O _t(\gamma^2 \log(\gamma)^2) + \mc O_{t,\gamma}(\alpha^{1/2})),
\end{split}\label{Bauer Kuwert Willmore estimate }
\end{equation}
where the constants in $\mc O_t$ respectively $\mc O_{t,\gamma}$ depend on $t$ respectively $t$ and $\gamma$.
\end{lemma}
With this lemma at hand, we can prove the main theorem of this section.
\begin{thm} \label{thm:connectedsum}
Suppose that $\Sigma_1$, $ \Sigma_2$ are two closed, oriented surfaces, $f_1:\Sigma_1 \to \R^3$ is a smooth embedding, $f_2:\Sigma_2 \to \R^3$ is a smooth immersion and neither $f_1$ nor $f_2$ parametrize a round sphere and also $\mc T(f_2) \neq \mc T(\mb S^2)$. Denote with $\Sigma$ the connected sum $\Sigma_1 \# \Sigma_2$. Then, there exists a smooth immersion $f:\Sigma \to \R^3$ such that
\begin{equation}
\mc T(f) = \mc T(f_2) \label{BK:Tf = Tf2}
\end{equation}
and 
\begin{equation}
\mc W(f) < \mc W(f_1) + \mc W(f_2) - 4\pi.\label{BK:strict inequality}
\end{equation}
\end{thm}
\begin{proof}
Let us first outline the strategy. We assume without loss of generality that $\mc T(f_2) > \mc T(\mb S^2)$. The proof for the case $\mc T(f_2) < \mc T(\mb S^2)$ is identical up to changes in the inequality signs. We will be concerned with three cases. In Case 1 and Case 2, we assume that there is some sequence $\alpha_n \to 0$ such that $\limn \int _{U_{\alpha_n} } H_{f_1^\circ} \dif  \vol _{f_1^\circ}\neq 0$. Then, we show that the glued surface $f:\Sigma \to \R^3$ with specific choices of coefficients $t, \gamma$ and $\alpha$ satisfies
\begin{equation}
\mc T(f) > \mc T(f_2)> \mc T(\mb S^2)\quad \text{and}\quad\mc W (f) < \mc W(f_1) + \mc W(f_2) - 4\pi. \label{BK:Tf greater than Tf2}
\end{equation} Once this is established, we apply Proposition \ref{prop: Willmore invariance}, Lemma \ref{lem: Möbius blow down} and Lemma \ref{lem: Möbius blowup}. We have that
\[\lim _{|a|\to \infty} \mc T(I_a \circ f) = \mc T(f), \quad \lim _{a\to \im f} \mc T(I_a \circ f) = \mc T(\mb S^2) \]
and $\mc W(I_a\circ f) = \mc W(f)$ for $a\not \in \im f$. So by continuity, we find some $a\in \R^3\setminus \im f$ such that $\mc T(I_a \circ f) = \mc T(f_2)$ and $\mc W(I_a \circ f) = \mc W(f)$ which finishes the proof. In the remaining Case 3, we will instead directly apply Lemma \ref{lem:First variation of T}.\medbreak
\textbf{Case 1: $\omega \cqq \limsup_{\alpha \to 0}\int _{U_\alpha } H_{f_1^\circ} \dif \vol_{f_1^\circ} > 0$.} By \eqref{BK:Frobenius Inner Product Positive}, we have that $\langle P^\circ, Q^\circ\rangle >0$. Thus, it is possible to choose $t>0$ large enough such that 
\begin{equation}
|P^\circ|^2 - t \langle P^\circ, Q^\circ \rangle < 0.\label{BK: t estimate}
\end{equation} 
Following from \eqref{BK: t estimate}, we choose $\gamma$ small enough such that 
\begin{equation}
|P^\circ |^2 - t\langle P^\circ, Q^\circ \rangle + \mc O _t(\gamma^2 \log(\gamma)^2)<0\label{BK:gamma estimate}
\end{equation}
which is possible as $\lim _{\gamma \to 0^+} \gamma^2 \log(\gamma)^2 =0$. Possibly decreasing $\gamma$, we also require that $\omega + \mc O_t(\gamma^2)> \omega/2 > 0$. Then, choose $\alpha > 0$ small enough such that 
\begin{equation}
\int _{U_\alpha } H_{f_1^\circ} \dif \vol_{f_1^\circ} + \mc O_t(\gamma^2) > \omega/4,\label{BK:alpha estimate 1}
\end{equation}
which is possible by definition of $\omega$. Possibly making $\alpha$ smaller, we also require that the right-hand side of \eqref{BK:Estimate T} is positive, and furthermore that the right-hand side of \eqref{Bauer Kuwert Willmore estimate } is still negative, which is possible due to \eqref{BK:gamma estimate}. Then, \eqref{Bauer Kuwert Willmore estimate } implies \eqref{BK:strict inequality} and together with \eqref{BK:Estimate T}, we deduce \eqref{BK:Tf greater than Tf2}. The statements from the theorem, \eqref{BK:Tf = Tf2} and \eqref{BK:strict inequality}, then follow as mentioned before by an application of Proposition \ref{prop: Willmore invariance}, Lemma \ref{lem: Möbius blow down} and Lemma \ref{lem: Möbius blowup}.  Notice also that this case includes the possibility that $\omega =+\infty$. \medbreak
\textbf{Case 2: $\omega \cqq \liminf_{\alpha \to 0}\int _{U_\alpha } H_{f_1^\circ} \dif \vol_{f_1^\circ} < 0$.} Denote by $A \in O(3)$ the matrix 
\[A = \begin{pmatrix}
0 & -1 & 0\\ 
1 & 0 & 0 \\ 
0 & 0 & -1
\end{pmatrix}. \] Instead of the immersion $f_1$, we consider $f_1' \cqq A\circ f_1$, which corresponds to mirroring $f_1$ along the $xy$-plane and then rotating it $90^\circ$ around the $z$-axis. Then, 
\[{\int _{U_\alpha } H_{{f_1'}^\circ} \dif \vol_{{f_1'}^\circ} = - \int _{U_\alpha } H_{f_1^\circ} \dif \vol_{f_1^\circ}}.\]
Furthermore, if $P'$ denotes the corresponding second fundamental form at the origin of $f_1'$, a simple computation shows $(P')^\circ=P^\circ$ and thus \eqref{BK:Frobenius Inner Product Positive} implies $\langle (P')^\circ, Q^\circ \rangle > 0$. We can now apply Case 1.\medbreak
\textbf{Case 3: $\lim_{\alpha \to 0}\int _{U_\alpha } H_{f_1^\circ} \dif \vol_{f_1^\circ} = 0$.} In this case, Lemma \ref{Difference of H-integrals} implies
\begin{align}
\mc T(f) = \mc T(f_2) +\mc O_t(\gamma^2) \alpha^2 + o_{t,\gamma}(\alpha^2). \label{BK:Case 3 estimate T}
\end{align}
Since $f_2$ is not a round sphere, we can apply Lemma \ref{lem:First variation of T} to $f_2$ to obtain $f_{2,s} \cqq f_2 + s\xi$ where $\xi$ is some smooth normal vector field supported away from $p_2$\footnote{Remember that $p_2$ was the nonumbilic point of $\Sigma_2$ around which we pasted the inverted surface $f_1^\circ$.} such that 
\begin{equation}
\dv{}{s}\bigg\vert _{s=0} \mc T(f_{2,s}) > 0\label{BK:Derivative T}
\end{equation}
where we possibly replaced $\xi$ by $-\xi$. We apply the connected sum construction to $f_1$ and $f_{2,s}$ around $p_1$ and $p_2$ to obtain a glued surface $f'_{s,\alpha,\gamma,t}:\Sigma \to \R^3$. For this glued surface, notice that the right-hand side of \eqref{Bauer Kuwert Willmore estimate } with $f_2$ replaced by $f_{2,s}$ does not depend on $s$ since we applied the variation supported away from the gluing region.  Consider next $f_{s,\alpha,\gamma,t} \cqq \beta f'_{s,\alpha,\gamma,t}$ and $f = f_{0,\alpha,\gamma,t}$, where the rescaling by $\beta$ changes neither the total mean curvature ratio nor the Willmore energy. For fixed $\gamma$ and $t$, it holds by \cite[(3.13),(3.17)]{MondinoScharrerInequality} that
\begin{equation*}
\mc A(f_{s,\alpha,\gamma, t}\vert _{\Sigma \setminus V}) = \mc A(f_{0,\alpha,\gamma, t}\vert _{\Sigma \setminus V}) \to 0 \qquad \text{as}\quad \alpha \to 0
\end{equation*}
and, by a computation similar as \eqref{BK:Int H estimate},
\begin{equation*}
\int _{\Sigma \setminus V} H_{f_{s,\alpha,\gamma, t}} \dif \vol _{f_{s,\alpha,\gamma, t}} = \int _{\Sigma \setminus V} H_{f_{0,\alpha,\gamma, t}} \dif \vol _{f_{0,\alpha,\gamma, t}} \to 0\qquad \text{as}\quad \alpha \to 0.
\end{equation*}
Thus, for fixed $\gamma, t$ and $\alpha$ sufficiently small,  \eqref{BK:Derivative T} together with $f_{s,\alpha,\gamma,t}\vert _{V} = f_{2,s}\vert _{V}$ implies
\begin{equation}
 \frac{\mc T(f_{s,\alpha,\gamma,t})- \mc T(f)}{s} > C_1 + o(1)\quad \text{as}\quad s\to 0, \label{derivative estimate}
\end{equation}
where the constant $C_1>0$ and $o(1)$ in \eqref{derivative estimate} do not depend on $\alpha,\gamma$ and $t$. Choose $C_2> 0$ such that 
\begin{equation}
|\mc W(f_{s,\alpha,\gamma,t}) - \mc W(f)| < C_2 |s|  \quad \text{as}\quad s\to 0,\label{Willmore difference}
\end{equation}
where $C_2$ and $o(s)$ in \eqref{Willmore difference} are independent of $\alpha, \gamma$ and $t$. Now, let us choose the coefficients $t, \gamma,\alpha$ and $s$ appropriately. First, choose $t$ large enough such that
\[|P^\circ|^2 - t \langle P^\circ, Q^\circ \rangle < 0,\]
which is again possible by \eqref{BK:Frobenius Inner Product Positive}. Second, choose $\gamma$ small enough such that 
\[-\eps \cqq |P^\circ |^2 - t\langle P^\circ, Q^\circ \rangle + \mc O _t(\gamma^2 \log(\gamma)^2) < 0,\]
Choosing $\gamma$ possibly even smaller, we also require that $ \left |\mc O_t(\gamma^2)\right |$ from \eqref{BK:Case 3 estimate T} satisfies $  {\left |\mc O_t(\gamma^2)\right |<\frac{1}{4}C_1 C_2^{-1}\pi\eps}$. This guarantees that, third, $\alpha>0$ can be taken sufficiently small such that 
\begin{equation}
|\mc O_t(\gamma^2)|\alpha^2 + |o_{t,\gamma}(\alpha^2)| < 
-\frac{1}{2}C_1 C_2^{-1}\pi \alpha^2( -\eps + \mc O_{t,\gamma}(\alpha^{1/2})). \label{alpha estimate}
\end{equation}
Finally, let $s_0=2C_1^{-1} (|\mc O_t(\gamma^2)|\alpha^2 + |o_{t,\gamma}(\alpha^2)|)$. Now, set $s=s_0$ and notice that for $\alpha$ sufficiently small, it holds that
\begin{align*}
&\quad \, \mc T(f_{s,\alpha,\gamma,t}) -\mc T(f_2) \\
& =\mc T(f_{s,\alpha,\gamma,t}) - \mc T(f) + \mc T(f) -\mc T(f_2) \\
\text{(by \eqref{BK:Case 3 estimate T} and \eqref{derivative estimate})}\quad &> C_1 s + o(s)-| \mc O_t(\gamma^2)|\alpha^2 - |o_{t,\gamma}(\alpha^2)|\\
\text{(by definition of $s_0$)}\quad&= o(s) + |\mc O_t(\gamma^2)|\alpha^2 + |o_{t,\gamma}(\alpha^2)| \\
\text{($o(s) = o_{t,\gamma}(\alpha^2)$)}\quad &> 0  
\end{align*}
by the choice of $s$. For $s =-s_0$ we analogously obtain $\mc T(f_{s,\alpha,\gamma,t}) -\mc T(f_2)  < 0$. By continuity, we find some $|s| \leq s_0$ such that $\mc T(f_{s,\alpha,\gamma,t}) = \mc T(f_2)$. For this choice of $t,\gamma,\alpha$ and $s$ we finally obtain, using the previous estimates, that
\begin{align*}
&\quad \, \mc W(f_{s,\alpha, \gamma,t}) - (\mc W(f_1) + \mc W(f_2) - 4\pi) \\
\text{(\eqref{Bauer Kuwert Willmore estimate } and \eqref{Willmore difference})}\quad &\leq C_2| s |+ \pi \alpha^2( -\eps + \mc O_{t,\gamma}(\alpha^{1/2}))\\
\text{(choice of $s_0$)}\quad &< 
2C_2 C_1^{-1}(|\mc O_t(\gamma^2)|\alpha^2 + |o_{t,\gamma}(\alpha^2)|) +\pi \alpha^2( -\eps + \mc O_{t,\gamma}(\alpha^{1/2}))\\
\text{(\eqref{alpha estimate})}\quad&<0.\tag*\qedhere
\end{align*}
\end{proof}
\section{An explicit construction of surfaces with less than \texorpdfstring{$8\pi$}{} Willmore energy and prescribed mean curvature constraint}\label{sec:8piconstruction}
In this section, we want to show that for an arbitrary genus $g$ and each $R \in I$, there exists a smooth embedded surface $\Sigma \subset \R^3$ of genus $g$ such that $\mc W(\Sigma)< 8\pi$ and $\mc T(\Sigma) = R$. This will confirm \eqref{8pi bound} and finish the proof of the constrained minimization problem, see also Remark \ref{remark: I}. Using Lemma \ref{lem: Möbius blow down} and  \ref{lem: Möbius blowup}, it is enough to construct genus $g$ surfaces $\Sigma^{1,g}_n$ and $\Sigma ^{2,g}_n$ with $\mc T(\Sigma ^{1,g}_n) \to 0$, $\mc T(\Sigma^{2,g}_n) \to \sqrt{2} \mc T(\mb S^2)$ and $\mc W(\Sigma^{1,g}_n)<8\pi$, $ \mc W(\Sigma ^{2,g}_n) < 8\pi$. We will use the construction from \textsc{Ndiaye} and \textsc{Schätzle} \cite[Proposition D.1.]{Ndiaye}, which we will recall in the following proposition, while also showing that the total mean curvature ratios converge. This construction can also be used to show the $8\pi$ bound for the isoperimetrically constrained problem \eqref{intro:KMR8pi}, see Remark \ref{rem:iso constrained 8pi}, providing an alternative proof to \cite{KusnerMcGrath}.
\begin{prop}\label{prop:8piconstruction}
There exist smooth genus $g$ surfaces $\Sigma^{1,g}_n,\Sigma^{2,g}_n \subset \R^3$ such that $\mc W(\Sigma^{1,g}_n)< 8\pi$, $ \mc W(\Sigma^{2,g}_n) < 8\pi$ and
\[\mc T(\Sigma^{1,g}_n) \to 0, \quad \mc T(\Sigma^{2,g}_n) \to\sqrt{2} \mc T(\mb S^2) \quad \text{as $n\to \infty$}.\]
\end{prop}
\begin{proof}
We denote by $\gamma: \mb S^1 \times \R \to \R^3$ the immersion 
\[\gamma(\theta, t)\cqq  \cosh t (\cos \theta, \sin \theta, 0) + t e_3,\]
where $e_3=(0,0,1)$. The image of this immersion is called \emph{catenoid}, a minimal surface in $\R^3$, i.e., $H_{\gamma} =0$. For now, the construction we are about to do is rotationally symmetric around the $z$-axis, so we simply work in the $xz$-plane. Let $S(t)$ be the sphere with center on the $z$-axis which touches the catenoid tangentially from above in $(\pm \cosh t, t)$ for $t\geq 0$. If we denote by $r(t)$ the radius, then by a simple computation, we obtain $r(t) = \cosh(t)^2$ and that the center is at $C(t) \cqq (0, \xi(t))$, where $\xi(t) = \sinh t \cosh t + t$. We denote by $\frac{\alpha(t)}{2}$ the angle between $z$-axis and the line between $C(t)$ and $(\cosh(t),t)$, such that
\[\cos \left (\frac{\alpha(t)}{2}\right ) = \frac{\sinh t \cosh t + t - t}{\cosh(t)^2} = \tanh t. \]
Similarly, the sphere $-S(t)$ tangentially touches the catenoid from below in the points $(\pm \cosh(t), -t)$. We denote by $\Capp(t)$ the spherical cap of $S(t)$ below $z=t$ and similarly $-\Capp(t)$ the spherical cap of $-S(t)$ above $z=-t$. We consider the surface $\Gamma_t$ obtained by gluing together $S(t)$ and $-S(t)$ along the segments $(\pm \cosh t, \pm t)$ with the segment $\cat(t) \cqq \gamma(\mb S^1 \times [-t,t])$ of the catenoid which lies in between both segments, while removing $\Capp(t)$ from $S(t)$ and $-\Capp(t)$ from $-S(t)$. Because the spheres touch the catenoid tangentially, this results in an orientable $C^{1,1}$- surface. Since the catenoidal segment has no Willmore energy, the Willmore energy of this surface is $8\pi$ minus the Willmore energy of $\pm\Capp(t)$. We can compute $\mc A(\Capp(t)) = 2\pi r(t)^2 (1-\tanh t)$. Since the mean curvature of a sphere of radius $r$ is $\frac{2}{r}$, we obtain that $\Gamma_t$ has a Willmore energy of
\[\mc W(\Gamma_t) = 8 \pi - 4\pi(1 - \tanh t).\]
We consider $t$ large enough such that $S(t)\cap -S(t) =\emptyset$, that is
\begin{equation}
1<\frac{\xi(t)}{r(t)} = \tanh t + \frac{t}{\cosh(t)^2}, \label{ratio of center and radius}
\end{equation}
which is true for sufficiently large $t$.  

We now consider the stereographic projection in $\R^3$ of the unit sphere centered at $e_3$ given by
\[T(p) \cqq e_3 +2 \frac{p-e_3}{|p- e_3|^2}.\]
Then, $T = T^{-1}$ and 
\[T(\partial B_1(0)) = \{z=0\}\cup \{\infty\}.\]
Furthermore, for $\sigma\neq 1$, the sphere $\partial B_\sigma(0)$ is mapped to another sphere via $T$, given by
\begin{equation}
T(\partial B_\sigma(0)) = B_{\rho(\sigma)}(\zeta(\sigma)e_3), \quad \zeta(\sigma) = -\frac{1+\sigma^2}{1-\sigma^2}, \quad \rho(\sigma) = \left | \frac{2\sigma}{1-\sigma^2}\right |.  \label{def of rho and sigma}
\end{equation}
It follows
\[T(\partial B_{1/\sigma}(0)) = -T(\partial B_{\sigma}(0)).\]
The map $\sigma \mapsto \frac{\rho(\sigma)}{|\zeta(\sigma)|} = \frac{2\sigma}{1+\sigma^2}$ is bijective on the interval $(0,1)$. Hence, after rescaling $S(t)$ appropriately by a factor of $\lambda(t)$ and using the fact that $S(t)$ and $-S(t)$ are disjoint, we find $\sigma = \sigma(t)\in (0,1)$ such that
\[-\lambda(t) S(t) = T(\partial B_{\sigma(t)}(0)), \quad \lambda(t) S(t) = T(\partial B_{1/\sigma(t)}(0)).\]
Then, it follows that
\[\tanh t + \frac{t}{\cosh(t)^2} = \frac{1+ \sigma^2}{2\sigma}, \quad \lambda(t) \cosh(t)^2 = \frac{2\sigma}{1-\sigma^2}.\]
The first equation implies $\sigma(t) \to 1$ as $t\to \infty$. Then, using \eqref{def of rho and sigma} results in
\[\rho(\sigma(t))^2 \left (\tanh t + \frac{t}{\cosh(t)^2}-1\right ) \to \frac{1}{2}\quad \text{as $t\to \infty$}, \quad \rho(\sigma(t))^2 te^{-2t} \to \frac{1}{8} \quad \text{as $t\to \infty$}.\]
This also implies $\lambda(t) = \rho(\sigma(t))/r(t) = \mc O(t^{-1/2} e^{-t})$. We now define the Möbius transformation $\Phi_t$ by $\Phi_t(p) \cqq T(\lambda(t)p)$ so that we obtain
\[\Phi_t(-S(t)) = \partial B_{\sigma(t)}(0),\quad \Phi_t(S(t))  = \partial B_{1/\sigma(t)}(0).\]
Furthermore, $\Sigma_t\cqq \Phi_t(\Gamma_t)$ is still a closed, orientable $C^{1,1}$-surface consisting of the two spheres $\partial B_{\sigma(t)}(0)$ and $\partial B_{1/\sigma(t)}(0)$ with the spherical caps 
\[\Capp^-(t) \cqq \Phi_t(-\Capp(t))\subset \partial B_{\sigma(t)}(0), \quad \Capp^+(t)\cqq \Phi_t(\Capp(t))\subset \partial B_{1/\sigma(t)}(0) \]
removed and replaced by $\Phi_t(\cat(t))$. If we denote by $I$ the inversion at the unit sphere, i.e., $I(w) = w/|w|^2$, then $T\circ I\circ T \circ (-\id) $ is the reflection around the $z$-axis\footnote{Rotational symmetry is preserved, so it suffices to show this for the $xz$-plane. Then, identifying the $xz$-plane with $\C$, this follows by a simple computation using $I(x+iy) = \frac{1}{x-iy}$ and $T(z) = i+\frac{2}{x+(1-y)i}$.}. It follows that \[I(\Capp^+(t)) = I (T(\lambda (t) \Capp(t))) = T(-\lambda (t) \Capp(t)) =  \Capp^-(t)\]
and both have the same opening angle $\beta(t)$. The Willmore energy stays invariant under this inversion by Proposition \ref{prop: Willmore invariance} so that 
\begin{equation}
\mc W(\Sigma_t) = \mc W(\Gamma_t)= 8\pi - 4\pi(1-\tanh t).\label{catenoidal bridge saves energy}
\end{equation}
To estimate the opening angle $\beta(t)$, we estimate the area of the rescaled caps $\Capp(t)$. Using that $\lambda(t) = \mc O(t^{-1/2} e^{-t})$, $\mc H^2(\Capp(t)) = \mc O(e^{2t})$ and $\diam(\Capp(t)) = \mc O(e^t)$, we obtain that $\mc H^2(\lambda(t)\Capp(t)) \to 0$ and $\diam(\lambda(t) \Capp(t))\to 0$ as $t\to \infty$. Furthermore, for sufficiently large $t$, ${\pm\lambda(t)\Capp(t)\subset B_{1/2}(0)}$ and since $DT$ is bounded in $B_{1/2}(0)$, it follows from the area formula that
\[\mc H^2(\Capp^\pm (t)) = \mc H^2(T(\pm \lambda(t) \Capp(t))) \to 0 \text{ and } \diam(\Capp^ \pm (t)) \to 0 \text{ as $t\to \infty$}.\]
Therefore, $\beta(t)\to 0$ as $t\to \infty$. 
By the same arguments, we obtain $\mc H^2(\Phi_t(\cat(t)))\to 0$ and $\diam(\Phi_t(\cat(t))) \to 0$. Thus, we have constructed two spheres of radius $\sigma(t)$ and $ 1/\sigma(t)$ connected by a handle of vanishing area and diameter in the limit. As the two spherical segments have a Willmore energy of almost $8\pi$, the Willmore energy of the catenoidal segment converges to 0. Furthermore, by \eqref{catenoidal bridge saves energy}, we see that the catenoidal bridge has less energy than the two removed caps. We can copy and paste this handle $g+1$ times at different points $p_0=-e_3,\ldots, p_g$ on $\partial B_1(0)$ to obtain a genus $g$ surface. More precisely, for $\eps>0$ sufficiently small such that the balls $B_\eps(p_i)$ are pairwise disjoint and $t$ sufficiently large, we define the surface
\[\Sigma^{1,g}_t \cqq \left (\bigcup_{i=0}^g R_i(\Sigma_t \cap B_{\eps}(p_0))\right )\cup \left (\Bigl(\partial B_{1/\sigma(t)}(0) \cup \partial B_{\sigma(t)}(0)\Bigr) \setminus \bigcup _{i=0}^g B_\eps(p_i)\right ).\]
Here, $R_i\in O(3)$ are rotations such that $R_i(p_0)=p_i$. This surface is a closed, connected and oriented $C^{1,1}$-surface as long as the catenoidal segment of $\Sigma_t$ is contained in $B_\eps(p_0)$, which is satisfied for large $t$ as $\beta(t)\to 0$ for $t\to \infty$. We can also choose the $p_i$ such that $\|e_3 - p_i\| > 1$ and define 
\[\Sigma^{2,g}_t \cqq T(\Sigma^{1,g}_t).\]
Then, it holds that $\mc W(\Sigma^{1,g}_t)= \mc W(\Sigma^{2,g}_t) = 8\pi - 4\pi (g+1)(1-\tanh t) < 8\pi$. 

Next, we show that $\mc T(\Sigma ^{1,g}_t) \to 0$ and $\mc T(\Sigma ^{2,g}_t) \to \sqrt{2}\mc T(\mb S^2)$ for $t\to \infty$. $\Sigma^{1,g}_t$ consists of two spheres with radius $\sigma(t)$ and $1/\sigma(t)$, where $g+1$ spherical caps of vanishing area are removed and replaced by $g+1$ catenoidal handles with vanishing area and vanishing Willmore energy. We observe that by the Cauchy-Schwarz inequality, the total mean curvature of the catenoid and the spherical caps goes to 0. We thus obtain\footnote{Notice here that the normal vectors of the two spheres point in opposite directions, resulting in the sign change.}
\begin{equation}
\mc A (\Sigma^{1,g}_t) \to 8\pi, \quad \int _{\Sigma^{1,g}_t} H_{\Sigma ^{1,g}_t} \dif \vol_g \to \int _{\partial B_1(0)} 2 \dif \vol _g - \int _{\partial B_1(0)} 2 \dif \vol _g =0.\label{NS:area estimate 2}
\end{equation}
Thus $\mc T(\Sigma^{1,g}_t) \to 0$. However, $\Sigma^{2,g}_t$ consists of two spheres with diverging radii $\lambda(t) r(t)$ and $g$ spherical caps of bounded area\footnote{By the assumption that the $p_i$ have positive distance to $e_3$, $DT$ is bounded.} removed and catenoidal handles of bounded area inserted. Hence
\begin{equation}
\mc A (\Sigma^{2,g}_t) = 8\pi (\lambda(t)r(t))^2 + \mc O(1), \quad \int _{\Sigma^{2,g}_t} H_{\Sigma^{2,g}_t}\dif \vol_g =8 \pi (\lambda(t)r(t))^2 \frac{2}{\lambda(t)r(t)} + \mc O(1)\label{NS:area estimate}
\end{equation}
as $t\to \infty$. Thus
\[\mc T(\Sigma^{2,g}_t) = \frac{8 \pi (\lambda(t)r(t))^2 \frac{2}{\lambda(t)r(t)} + \mc O(1)}{\sqrt{8\pi (\lambda(t)r(t))^2 + \mc O(1)}} \to \sqrt{2} \mc T(\mb S^2).\]
Finally, we can approximate this surface by smooth surfaces by writing the gluing regions in a local graph representation and using standard mollifiers $\rho_\delta$ to obtain smooth surfaces $\Sigma^{1,g}_{t,\delta}$, $\Sigma^{2,g}_{t,\delta}$ for which 
\[\lim _{\delta \to 0}\mc W(\Sigma ^{i,g}_{t,\delta}) \to \mc W(\Sigma ^{i,g}_t),\quad \lim _{\delta \to 0} \mc T(\Sigma ^{i,g}_{t,\delta}) \to \mc T(\Sigma ^{i,g}_t).\]
The proof is finished by taking a suitable diagonal sequence in $t$ and $\delta$.
\end{proof}
\begin{remark} \label{rem:iso constrained 8pi}
Similar to the estimate for the total mean curvature ratio, we can also estimate the isoperimetric ratio of the constructed surfaces $\Sigma^{1,g}_t$, $ \Sigma^{2,g}_t$. Since the radii of the two spheres $\partial B_{\sigma(t)}(0)$, $\partial B_{1/\sigma(t)}(0)$ converge to $1$ and the diameter of the catenoidal handles converge to 0, we obtain that
\[\mc V(\Sigma^{1,g}_t) = o(1)\quad\text{as }t\to 0.\]
For $\Sigma^{2,g}_t$, we recover, using the same argument
\[\mc V(\Sigma^{2,g}_t) = \frac{8\pi}{3} (\lambda(t)r(t))^3 + \mc O(1).\]
Together with the area estimate \eqref{NS:area estimate 2} and \eqref{NS:area estimate}, it follows that
\[\iso(\Sigma^{1,g}_t) \to \infty,\quad \iso(\Sigma^{2,g}_t) \to \sqrt[3]{72\pi}.\]
Notice that this proves \eqref{intro:KMR8pi} using \cite[Lemma 3.7, Lemma 3.8]{ScharrerPhD}.
\end{remark}
\section{Minimizers close to round spheres}\label{sec:CloseToSphere}
In 1903, \textsc{Minkowski} \cite{Minkowski} showed that for closed, convex surfaces $\Sigma \subset \R^3$, i.e., $\Sigma$ being the boundary of a convex set in $\R^3$, it holds that
\[\mc T(\Sigma)\geq \mc T(\mb S^2)\]
with equality if and only if $\Sigma$ is a round sphere. In particular, if one considers small $C^2$-variations of round spheres, these are still convex by \cite[Theorem 1.8]{DajczerTojeiro} and so they increase the total mean curvature ratio. As the Willmore energy varies continuously under these perturbations, it follows that $\beta_0$ is right-continuous at $\mc T(\mb S^2)$ and one might expect the left-continuity to fail. However, the opposite was shown in \cite{ChodoshEichmairKoerber}. For $u\in C^\infty(\mb S^2)$, we denote by $\Sigma(u)$ the graph of $u$ over $\mb S^2$, i.e.,
\[\Sigma(u)\cqq \{(1+u(x))x, \; x\in \mb S^2\}.\]
Notice that $\mc W(\Sigma(u)) \to 4\pi$ as $\|u\|_{W^{2,2}(\mb S^2)}\to 0$. \textsc{Chodosh}, \textsc{Eichmair}, and \textsc{Koerber} \cite{ChodoshEichmairKoerber} showed that there are $u_n\in C^\infty(\mb S^2)$ such that $\|u_n\|_{W^{2,2}(\mb S^2)}\to 0$ as $n\to \infty$ and $\mc T(\Sigma(u_n))<\mc T(\mb S^2)$. In particular, $\beta_0$ is continuous in $(0, \sqrt{2}\mc T(\mb S^2))$. 
\begin{cor} \label{cor:Continuity at sphere}
For $g\in \N_0$, $\beta_g \vert _{(0,\sqrt{2}\mc T(\mb S^2))}$ is continuous and $\beta_g(\mc T(\mb S^2)) = \beta_g$, where $\beta_g$ is the global minimum of the Willmore energy  for genus $g$ surfaces defined in \eqref{beta g for smooth surfaces}.
\end{cor}
\begin{proof}
	Using that $\beta_0(\tau)\to 4\pi$ as $\tau\to \mc T(\mb S^2)$, it follows by Theorem \ref{thm:connectedsum} that $\beta_g(\tau)\to \beta_g$ as $\tau \to \mc T(\mb S^2)$. It remains to show that $\beta_g(\mc T(\mb S^2)) = \beta_g$. For that, let $\eps>0$ and $\phi \in C^\infty(\mb S^2)$ be such that $\|\phi\|_{W^{2,2}(\mb S^2)} < \eps$, $\Sigma(\phi)$ is convex and such that $\Sigma_{n,t}\cqq \Sigma((1-t)u_n + t\phi)$ is not a round sphere\footnote{In \cite{ChodoshEichmairKoerber}, the $u_n$ are chosen rotationally symmetric, so it suffices to choose $\phi$ rotationally symmetric in one half of $\mb S^2$ and not rotationally symmetric in the other half.} for any $n\in \N$ and $t\in [0,1]$. As $\mc T(\Sigma_{n,t})$ varies continuously in $t$ and $\mc T(\Sigma_{n,0})<\mc T(\mb S^2)<\mc T(\Sigma_{n,1})$, it follows that there exists $t_n \in [0,1]$ such that $\mc T(\Sigma_{n, t_n }) = \mc T(\mb S^2)$.  Choose $n $ sufficiently large such that $\|u_n\|_{W^{2,2}(\mb S^2)}<\eps$. Then $\|(1-t_n ) u_n  + t_n \phi\|_{W^{2,2}(\mb S^2)}<\eps $ and passing to the limit yields
	\[\limsup_{\eps \to 0} \mc W(\Sigma_{n, t_n})=4\pi.\]
Finally, because $\Sigma_{n,t_n}$ is not a round sphere, we can apply Theorem \ref{thm:connectedsum} to obtain genus $g$ surfaces $\Sigma ^g_{n,t_n}$ with $\mc T(\Sigma ^{g}_{n,t_n}) = \mc T(\mb S^2)$ and 
\[\limsup_{\eps \to 0} \mc W(\Sigma^g_{n, t_n})\leq \beta_g.\qedhere\]
\end{proof}
By \cite[Proposition E.4]{ChodoshEichmair}, it holds that for for any sequence $u_n\in C^\infty(\mb S^2)$ such that $\|u_n\|_{C^1(\mb S^2)}\to 0$, we have as $n\to \infty$
\[\mc T(\Sigma(u_n)) - \mc T(\mb S^2)\geq - o(1)  (\mc W(\Sigma(u)) - 4\pi).\]
We will extend this result to arbitrary spherical surfaces with a total mean curvature ratio converging to $\mc T(\mb S^2)$. \medskip

Recall that the $L^2(\dif \vol_g)$-gradient of $\mc T$ as in \eqref{T gradient} is given by
\[\nabla \mc T(f) = \frac{\mc T(f) H}{2\mc A(f)} - \frac{2K}{\sqrt{\mc A(f)}}.\]
This gradient is non-zero whenever $f$ does not immerse a round sphere, see Proposition \ref{prop: Alexandrov-corollary}. The $L^2(\dif \vol_g)$-gradient $\nabla \mc W$ of the Willmore functional is given by 
\[\mc W(f) = \Delta H + |\mb I^0 |_g^2 H,\]
where $\mb I^0$ is the trace-free second fundamental form, see e.g. \cite[(7.40)]{Willmore1993}. The Euler-Lagrange equation for minimizers $f^0 _R$ of \eqref{beta g tau for smooth surfaces} from Theorem \ref{thm:KMR} then becomes
\begin{equation}
\nabla \mc W (f^0_R) = \lambda \nabla \mc T(f^0_R),\label{Euler Lagrange Equation}
\end{equation}
for some $\lambda \in \R$ depending on $f^0_R$. Notice that if $\beta_0$ is differentiable at $R$, then $\beta_0'(R) = \lambda$. Furthermore, it is easy to check that $\lambda$ is invariant under rescaling of $f$.
\begin{thm} \label{thm:asymptotic estimate at T(S2)}
Suppose that $\tilde{f_n}:\mb S^2\to \R^3$ are immersions such that $\mc T(\tilde{f_n})\to \mc T(\mb S^2)$. Then as $n\to \infty$
\[\mc T(\tilde{f_n})-\mc T(\mb S^2) \geq  - o(1) (\mc W(\tilde{f_n}) - 4\pi).\]
\end{thm}
\begin{proof}
We may assume that $\mc T(\tilde{f_n})< \mc T(\mb S^2)$ and that the $\tilde{f_n}=f^0 _{\mc T(\tilde{f_n})}$ are the minimizers of the constrained problem \eqref{beta g tau for smooth surfaces} given by Theorem \ref{thm:KMR}. We will proceed by contradiction and assume that after possibly taking a subsequence
\begin{equation}
\mc T(\tilde{f_n}) - \mc T(\mb S^2)\leq -c_0 (\mc W(\tilde{f_n}) - 4\pi) \label{Assumption bounded lagrange multiplier}
\end{equation}
for some $c_0>0$. Notice that since $\beta_0$ is monotonically decreasing on $(-\infty, \mc T(\mb S^2)]$ by Corollary \ref{cor:Monotonicity}, $\beta_0'$ exists a.e. and 
\[\mc W(\tilde{f_n}) - 4\pi \geq \int ^{\mc T(\tilde{f_n})}_{\mc T(\mb S^2)} \beta_0'(t) \dif t.\]
From \eqref{Assumption bounded lagrange multiplier}, we see that there are $R_n \in (\mc T(\tilde{f_n}), \mc T(\mb S^2))$ such that $\beta_0$ is differentiable in $R_n$ and $|\beta_0'(R_n)| \leq \frac{1}{c_0}$. Setting $f_n = f^0 _{R_n}$ then implies that $f_n$ solve the Euler-Lagrange equation \eqref{Euler Lagrange Equation} with a Lagrange multiplier $\lambda_n$ such that $|\lambda_n| \leq \frac{1}{c_0}$. We may rescale that $\mc A(f_n) = 4\pi$. Also notice that since $R_n\to \mc T(\mb S^2)$, $\mc W(f_n)\to 4\pi$ by Corollary \ref{cor:Continuity at sphere}. \medskip

Let $i:\mb S^2\to \R^3$ be the standard isometric embedding. \textsc{De Lellis} and \textsc{Müller} \cite{deLellisMüller} proved the following rigidity estimate, namely that up to a reparametrization of $f_n$ 
\begin{equation}
\|f_n - i\|_{W^{2,2}(\mb S^2)} \to 0\label{strong W22 convergence of fn}
\end{equation}
and 
\begin{equation}
\|\mb I_{f_n} - g\|_{L^2(\mb S^2)} \to 0\label{strong L2 convergence of A}
\end{equation}
for the standard metric $g$ on $\mb S^2$. Let $(x_i, \rho_i)\in \R^3\times \R_{>0}$, $i\in I$ be a finite collection such that
\begin{equation}
i(\mb S^2) \subset \bigcup _{i\in I} B_{\rho_i/2}(x_i)\label{covering of sphere}
\end{equation}
and
\begin{equation}
\int _{i^{-1}(B_{2\rho_i}(x_i))}|\mb I_{i}|_g^2 \dif \vol_g < \frac{\eps_0}{2}.\label{local L^2 estimate for A}
\end{equation}
By the Sobolev embedding \cite[Theorem 2.10]{Aubin} and \eqref{strong W22 convergence of fn}, we obtain for $n$ sufficiently large and all $i\in I$
\begin{equation}
i^{-1}(B_{\rho_i/2}(x_i))\subset f_n^{-1}(B_{\rho_i}(x_i))\subset i^{-1}(B_{2\rho_i}(x_i)). \label{i balls subset fn balls}
\end{equation}
It follows by \eqref{strong L2 convergence of A} and \eqref{local L^2 estimate for A} that for $n$ sufficiently large
\[\int _{{f_n}^{-1}(B_{\rho_i}(x_i))}|\mb I_{f_n}|_g^2 \dif\vol_g < \eps_0,\]
where $\eps_0$ is the constant from \cite[Proposition 3.2]{Rupp2023volume}, see also \cite[Proposition 2.6]{WillmoreFlowSmallInitialEnergy}. For $i\in I$, take $\tilde \gamma_i\in C_c^\infty(\R^3)$ with $\chi _{B_{\rho_i/2}(x_i)}\leq \tilde \gamma_i \leq \chi_{B_{\rho_i}(x_i)}$ such that $\|D\tilde \gamma_i\|_{L^\infty}\leq \Lambda$, $\|D^2\tilde \gamma_i\|_{L^\infty}\leq \Lambda ^2$ for some fixed $\Lambda >0$ and $\gamma_{n,i} \cqq 
\tilde \gamma_i \circ f_n$. By \cite[Proposition 3.2]{Rupp2023volume}, the fact that we have bounded Lagrange multiplier and $\mc W(f_n)<C$, it holds that
\begin{equation}
\int (|\nabla^2 \mb I_{f_n}|_g^2 + |\mb I_{f_n}|_g^2 |\nabla \mb I_{f_n}|_g^2 + |\mb I_{f_n}|_g^6)\gamma_{n,i}^4 \dif \vol_g \leq C\left ( \int |\nabla \mc T(f_n)|_g^2 \gamma_{n,i}^4 \dif\vol_g + 1\right ).\label{local W22 estimate for A}
\end{equation}
Using \eqref{covering of sphere} and \eqref{i balls subset fn balls}, summation over $i$ in \eqref{local W22 estimate for A} yields
\begin{equation}
\int (|\nabla^2 \mb I_{f_n}|_g^2 + |\mb I_{f_n}|_g^2 |\nabla \mb I_{f_n}|_g^2 + |\mb I_{f_n}|_g^6) \dif \vol_g \leq C\left ( \int |\nabla \mc T(f_n)|_g^2 \dif\vol_g + 1\right ). \label{global W22 estimate for A}
\end{equation}
By summing over $i$ in \cite[Lemma 4.3]{GradientFlowWillmoreFunctional} and using \eqref{global W22 estimate for A}, it holds that
\begin{equation}
\|\mb I_{f_n}\|_{L^\infty(\mb S^2)}^4 \leq  C\left ( \int |\nabla \mc T(f_n)|_g^2 \dif\vol_g + 1\right ).\label{first L^infinity bound for A}
\end{equation}
Notice that $|K|\leq |\mb I|_g^2$ and $|H|\leq 2|\mb I|_g$. Hence $\|\nabla \mc T(f_n)\|_{L^2}^2 \leq C\left (\|\mb I_{f_n}\|_{L^\infty(\mb S^2)} + 1\right )$. Together with \eqref{first L^infinity bound for A}, this implies
\begin{equation}
\|\mb I_{f_n}\|_{L^\infty(\mb S^2)} \leq C. \label{second L^infinity bound for A}
\end{equation}
By \cite[Corollary 5.2]{GradientFlowWillmoreFunctional} and \eqref{global W22 estimate for A} it holds that $\|\mb I_{f_n}\|_{W^{2,2}}\leq C$. With this estimate, \eqref{strong L2 convergence of A}, and \eqref{global W22 estimate for A}, we apply again \cite[Lemma 4.3]{GradientFlowWillmoreFunctional} to $\phi = \mb I_{f_n} - g$ and sum over $i\in I$ to see
\begin{equation}
\|\mb I_{f_n} - g\|_{L^\infty}^4 \leq C \|\mb I_{f_n}-g\|_{L^2}^2 \to 0
\end{equation}
as $n\to \infty$. Notice that 
\[\mb I_{f_n}(v,v) = g(v,v) + (\mb I_{f_n}-g)(v,v) \geq |v|_g^2 - \|\mb I_{f_n}-g\|_{L^\infty} |v|_g^2 \geq \frac{1}{2} |v|_g^2\]
for any $v\in T_p \mb S^2$ and $n$ sufficiently large, and thus $\mb I_{f_n}$ is positive definite everywhere. By \cite[Theorem 1.8]{DajczerTojeiro}, $f_n$ immerses the boundary of a convex body in $\R^3$. However, by the Minkowski inequality, this implies $R_n= \mc T(f_n)\geq \mc T(\mb S^2)$, which is a contradiction.
\end{proof}
\section{On the total mean curvature of axisymmetric surfaces}\label{sec:OtherResults}
In \cite{Dalphin2016}, it was conjectured that any $C^{1,1}$-regular axisymmetric surface $\Sigma\subset \R^3$ satisfies
\begin{equation}
\int_\Sigma H \dif \vol_g \geq 0. \label{DalphinClaim}
\end{equation}
Let us first define the notion of an axisymmetric surface, following \cite{Dalphin2016} and focusing on embedded surfaces.
\begin{defi}[Axisymmetric $C^{1,1}$-sphere]
Suppose that $\gamma=(\gamma_1,\gamma_2):[0,T] \to \R^2$ is a $C^{1,1}$-immersion such that $\gamma(0) = (0,a_0)$, $\gamma(T) =  (0,a_T)$, $a_0<a_T$, $\gamma_2'(0) = \gamma_2'(T)=0$ and $\gamma_1(t)> 0$ for all $t\in (0,T)$. An \emph{axisymmetric $C^{1,1}$-sphere} is given as the image of the parametrization
\begin{equation}
\psi: \mb S^1 \times [0,T] \to \R^3, \quad \psi(\phi, s) = (\gamma_1(s) \cos (\phi), \gamma_1(s) \sin (\phi), \gamma_2(s)).\label{reparametrization}
\end{equation}
We parametrize $\gamma$ such that $|\gamma'| = 1$. Furthermore, there exists a Lipschitz map ${\theta:[0,T]\to \R}$ such that
\begin{equation}
(\gamma_1'(s), \gamma_2'(s)) = (\cos(\theta(s)), \sin(\theta(s))).\label{velocity parametrized by angle}
\end{equation}
We may assume $\theta(0)= 0$.
\end{defi}
Notice that all geometric quantities involving second derivatives can still be defined by approximation. By \cite{Dalphin2016}, the principal curvatures $\kappa_1$, $\kappa_2$ are given by
\begin{equation}
\kappa_1(s) = \frac{\sin(\theta(s))}{\gamma_1(s)}, \quad \kappa_2(s) = \theta'(s).\label{principal curvatures for axisymmetric surface}
\end{equation}
As the area element is $ \gamma_1(s) \dif s\dif \phi$, the total mean curvature is
\begin{equation}
\int _{\mb S^2} H \dif \vol_g = 2\pi \int_0^T\sin(\theta(s)) + \gamma_1(s) \theta'(s)\dif s = 2\pi \int _0^T \sin(\theta(s)) - \cos(\theta(s)) \theta(s)\dif s. \label{total mean curvature for spherical surface}
\end{equation}

In this section, we will provide a counterexample to \eqref{DalphinClaim}. Nonetheless, we will prove \eqref{DalphinClaim} in Theorem \ref{rot sym surface} under the additional assumption that $\theta(s) \in \left [-\frac{\pi}{2}, \frac{3\pi}{2}\right ]$ for all $s\in [0,T]$. 
\begin{thm} \label{rot sym surface}
Suppose that $f:\mb S^2\to \R^3$ is a $C^{1,1}$-axisymmetric sphere, which can be reparametrized by \eqref{reparametrization}. Suppose that $\theta(s) \in \left [-\frac{\pi}{2}, \frac{3\pi}{2}\right ]$ for all $s\in [0,T]$. Then \eqref{DalphinClaim} holds.
\end{thm}
\begin{proof}
Notice that by \eqref{velocity parametrized by angle}, it holds that
\begin{equation}
\int_0^T \sin(\theta(s)) \dif s = a_T - a_0 > 0, \quad \int_0^T \cos(\theta(s)) \dif s = 0.\label{int sin and cos of angle}
\end{equation}
Notice further that if $-\cos(\theta(s)) > 0$, then by the assumption on $\theta(s)$, we have $\theta(s)> \frac{\pi}{2}$. Since furthermore $\theta(s)\leq \frac{\pi}{2}$ if $-\cos(\theta(s))< 0$, it follows together with \eqref{int sin and cos of angle} that
\begin{align*}
&\quad \, \int _0^T \sin(\theta(s)) - \cos(\theta(s)) \theta(s)\dif s \\
&> \int _0^T -\cos(\theta(s)) \theta(s) \dif s \\
&= \int _{-\cos(\theta(s))>0} -\cos(\theta(s)) \theta(s) \dif s + \int _{-\cos(\theta(s))< 0}-\cos(\theta(s))\theta(s) \dif s \\
&\geq \int_{-\cos(\theta(s))>0} -\cos(\theta(s)) \frac{\pi}{2} \dif s + \int _{-\cos(\theta(s))< 0} -\cos(\theta(s)) \frac{\pi}{2} \dif s \\
& = 0.\tag*\qedhere
\end{align*}
\end{proof}
\begin{remark}\label{counterexamples for axisymmetric surfaces}
The statement is false once we allow larger images of $\theta$. Consider the following curves, defining axisymmetric surfaces:\medskip

\enlargethispage{1cm}
\noindent\begin{minipage}[t]{0.6\textwidth}
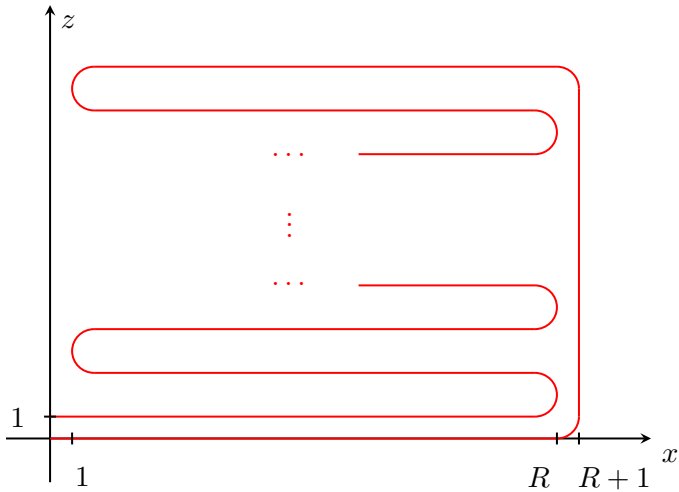
\begin{figure}[H]
\centering
\tikzset{every picture/.style={line width=0.75pt}} 

\begin{tikzpicture}[x=0.75pt,y=0.75pt,yscale=-1,xscale=1]
\draw [-stealth] (110,275) -- (110,35) node[anchor=north west] {$z$};
\draw [-stealth] (88,253) -- (410,253) node[anchor=north west] {$x$};

\draw [color={rgb, 255:red, 255; green, 0; blue, 0 }  ,draw opacity=1 ]   (110,242) -- (352,242) ;
 \draw  [color={rgb, 255:red, 255; green, 0; blue, 0 }  ,draw opacity=1 ] (352,220) .. controls (352,220) and (352,220) .. (352,220) .. controls (352,220) and (352,220) .. (352,220) .. controls (358.08,220) and (363,224.92) .. (363,231) .. controls (363,237.08) and (358.08,242) .. (352,242) ;  
\draw  [color={rgb, 255:red, 255; green, 0; blue, 0 }  ,draw opacity=1 ] (132,220) .. controls (132,220) and (132,220) .. (132,220) .. controls (125.92,220) and (121,215.08) .. (121,209) .. controls (121,202.92) and (125.92,198) .. (132,198) ;  
\draw [color={rgb, 255:red, 255; green, 0; blue, 0 }  ,draw opacity=1 ]   (132,220) -- (352,220) ;
\draw [color={rgb, 255:red, 255; green, 0; blue, 0 }  ,draw opacity=1 ]   (132,198) -- (352,198) ;
\draw  [color={rgb, 255:red, 255; green, 0; blue, 0 }  ,draw opacity=1 ] (352,176) .. controls (352,176) and (352,176) .. (352,176) .. controls (352,176) and (352,176) .. (352,176) .. controls (358.08,176) and (363,180.92) .. (363,187) .. controls (363,193.08) and (358.08,198) .. (352,198) ; 
\draw [color={rgb, 255:red, 255; green, 0; blue, 0 }  ,draw opacity=1 ]   (264,176) -- (352,176) ;
\draw [color={rgb, 255:red, 255; green, 0; blue, 0 }  ,draw opacity=1 ]   (264,110) -- (352,110) ;
 \draw  [color={rgb, 255:red, 255; green, 0; blue, 0 }  ,draw opacity=1 ] (352,88) .. controls (352,88) and (352,88) .. (352,88) .. controls (352,88) and (352,88) .. (352,88) .. controls (358.08,88) and (363,92.92) .. (363,99) .. controls (363,105.08) and (358.08,110) .. (352,110) ;  
\draw [color={rgb, 255:red, 255; green, 0; blue, 0 }  ,draw opacity=1 ]   (352,88) -- (132,88) ;
\draw  [color={rgb, 255:red, 255; green, 0; blue, 0 }  ,draw opacity=1 ] (132,88) .. controls (132,88) and (132,88) .. (132,88) .. controls (125.92,88) and (121,83.08) .. (121,77) .. controls (121,70.92) and (125.92,66) .. (132,66) ; 
\draw [color={rgb, 255:red, 255; green, 0; blue, 0 }  ,draw opacity=1 ]   (132,66) -- (363,66) ;
\draw  [color={rgb, 255:red, 255; green, 0; blue, 0 }  ,draw opacity=1 ] (363,66) .. controls (363,66) and (363,66) .. (363,66) .. controls (363,66) and (363,66) .. (363,66) .. controls (369.08,66) and (374,70.92) .. (374,77) ;  
\draw  [color={rgb, 255:red, 255; green, 0; blue, 0 }  ,draw opacity=1 ] (374,242) .. controls (374,242) and (374,242) .. (374,242) .. controls (374,248.08) and (369.08,253) .. (363,253) ;  
\draw [color={rgb, 255:red, 255; green, 0; blue, 0 }  ,draw opacity=1 ]   (374,77) -- (374,242) ;
\draw [color={rgb, 255:red, 255; green, 0; blue, 0 }  ,draw opacity=1 ]   (110,253) -- (363,253) ;
\draw    (121,250) -- (121,256) ;
\draw    (363,250) -- (363,256) ;
\draw    (374,250) -- (374,256) ;
\draw    (107,242) -- (113,242) ;
\draw (219,173) node [anchor=north west][inner sep=0.75pt]  [color={rgb, 255:red, 255; green, 0; blue, 0 }  ,opacity=1 ] [align=left] {$\ldots$};
\draw (226,130) node [anchor=north west][inner sep=0.75pt]  [color={rgb, 255:red, 255; green, 0; blue, 0 }  ,opacity=1 ] [align=left] {$\vdots$};
\draw (219,108) node [anchor=north west][inner sep=0.75pt]  [color={rgb, 255:red, 255; green, 0; blue, 0 }  ,opacity=1 ] [align=left] {$\ldots$};
\draw (121,266) node [anchor=north west][inner sep=0.75pt]   [align=left] {$1$};
\draw (99,248.5) node [anchor=south east][inner sep=0.75pt]   [align=left] {$1$};
\draw (361,266) node [anchor=north east][inner sep=0.75pt]   [align=left] {$R$};

\draw (372,266) node [anchor=north west][inner sep=0.75pt]   [align=left] {$R+1$};

\end{tikzpicture}
\caption{A curve with $\im \theta = [0,2\pi]$ defining a surface with total mean curvature ratio converging to $-\infty$.}\label{figure counterexample}
\end{figure}
\end{minipage}\hfill
\begin{minipage}[t]{0.35\textwidth}
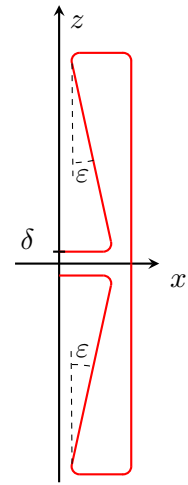
\begin{figure}[H]
\centering
\begin{tikzpicture}[x=0.75pt,y=0.75pt,yscale=-1,xscale=1]

\draw [-stealth, thick] (110,286) -- (110,46) node[anchor=north west] {$z$};
\draw [-stealth, thick] (88,176) -- (160,176) node[anchor=north west] {$x$};

\draw [color={rgb, 255:red, 255; green, 0; blue, 0 }  ,draw opacity=1, thick ]   (110,182) -- (132,182) ;
\draw  [color={rgb, 255:red, 255; green, 0; blue, 0 }  ,draw opacity=1, thick ] (132,182) .. controls (132,182) and (132.01,182) .. (132.01,182) .. controls (134.22,182) and (136.01,183.79) .. (136.01,186) .. controls (136.01,186.35) and (135.97,186.7) .. (135.88,187.03) ;  
\draw [color={rgb, 255:red, 255; green, 0; blue, 0 }  ,draw opacity=1, thick ]   (116.44,75.24) -- (135.88,164.92) ;
\draw  [color={rgb, 255:red, 255; green, 0; blue, 0 }  ,draw opacity=1 , thick] (120,282) .. controls (117.9,281.87) and (116.24,280.13) .. (116.24,278.01) .. controls (116.24,277.57) and (116.31,277.15) .. (116.44,276.76) ;  
\draw [color={rgb, 255:red, 255; green, 0; blue, 0 }  ,draw opacity=1, thick ]   (120,282) -- (142,282) ;
\draw  [color={rgb, 255:red, 255; green, 0; blue, 0 }  ,draw opacity=1 , thick] (146,277.88) .. controls (146,277.92) and (146,277.96) .. (146,278) .. controls (146,280.21) and (144.21,282) .. (142,282) .. controls (141.87,282) and (141.74,281.99) .. (141.61,281.98) ;  
\draw [color={rgb, 255:red, 255; green, 0; blue, 0 }  ,draw opacity=1, thick ]   (146,176) -- (146,278) ;
\draw [color={rgb, 255:red, 255; green, 0; blue, 0 }  ,draw opacity=1, thick ]   (146,74) -- (146,176) ;
\draw [color={rgb, 255:red, 255; green, 0; blue, 0 }  ,draw opacity=1, thick ]   (110,170) -- (132,170) ;
\draw  [color={rgb, 255:red, 255; green, 0; blue, 0 }  ,draw opacity=1, thick ] (132,170) .. controls (132,170) and (132.01,170) .. (132.01,170) .. controls (134.22,170) and (136.01,168.19) .. (136.01,165.95) .. controls (136.01,165.6) and (135.97,165.25) .. (135.88,164.92) ;  
\draw [color={rgb, 255:red, 255; green, 0; blue, 0 }  ,draw opacity=1, thick ]   (120,70) -- (142,70) ;
\draw [color={rgb, 255:red, 255; green, 0; blue, 0 }  ,draw opacity=1, thick ]   (135.88,187.03) -- (116.44,276.76) ;
\draw  [color={rgb, 255:red, 255; green, 0; blue, 0 }  ,draw opacity=1, thick ] (120,70) .. controls (117.9,70.13) and (116.24,71.87) .. (116.24,73.99) .. controls (116.24,74.43) and (116.31,74.85) .. (116.44,75.24) ;  
 \draw  [color={rgb, 255:red, 255; green, 0; blue, 0 }  ,draw opacity=1, thick ] (146,74.12) .. controls (146,74.08) and (146,74.04) .. (146,74) .. controls (146,71.79) and (144.21,70) .. (142,70) .. controls (141.87,70) and (141.74,70.01) .. (141.61,70.02) ;  
\draw [thick](107,170) -- (113,170);
\draw  [dash pattern={on 2.25pt off 2.25pt}]  (116,220) -- (116.44,276.76) ;
\draw  [dash pattern={on 2.25pt off 2.25pt}] (116.44,226.76) .. controls (116.44,226.76) and (116.44,226.76) .. (116.44,226.76) .. controls (120.14,226.76) and (123.74,227.16) .. (127.21,227.92) ;  
\draw  [dash pattern={on 2.25pt off 2.25pt}] (127.21,124.08) .. controls (123.74,124.84) and (120.14,125.24) .. (116.44,125.24) ;  
\draw  [dash pattern={on 2.25pt off 2.25pt}]  (116.44,75.24) -- (116.88,132) ;
\draw (117,215.4) node [anchor=north west][inner sep=0.75pt]    {$\varepsilon $};
\draw (117,127.4) node [anchor=north west][inner sep=0.75pt]    {$\varepsilon $};

\draw (99,170) node [anchor=south east][inner sep=0.75pt]   [align=left] {$\delta$};
\end{tikzpicture}
\caption{A curve with $\im \theta = \left [-\frac{\pi}{2}-\eps, \frac{3\pi}{2}+\eps \right ]$ defining a surface with negative total mean curvature.}\label{figure counterexample2}
\end{figure}
\end{minipage}\medskip

In Figure \ref{figure counterexample} and \ref{figure counterexample2}, the curves are unions of circular segments and straight lines. The curve in Figure \ref{figure counterexample} is such that the straight lines are parallel to the axes and the circular segments have a radius of 1 and we have a total of $n$ humps. Using \eqref{total mean curvature for spherical surface}, we estimate that the total mean curvature of the surface in Figure \ref{figure counterexample} is given by
\[\frac{1}{2\pi}\int_{\mb S^2} H_f \dif \vol_{g_f} = -Rn\pi + \pi R+ n \mc O(1) + \mc O(1) \]
as $n\to \infty$, $R\to \infty$. Furthermore, the area of the resulting surface is estimated by
\[\mc A(f) = \mc O_{R}(n)\]
as $n\to \infty$, where the constant in $\mc O_{R}$ depends on $R$. Hence, choose $R$ sufficiently large such that $ -R\pi + \mc O(1) < 0$ and let $n\to \infty$, we see that
\[\mc T(f) = \frac{-Rn\pi + n \mc O(1) + \mc O_R(1)}{\mc O_{R}(\sqrt{n})} \to -\infty\]
as $n\to \infty$.\medskip

In Figure \ref{figure counterexample2}, the surface consists of circular segments of radius $\delta>0$ and straight lines. All horizontal lines have a length of $1$, while the two sloped lines have a length of $\frac{1}{\eps}$ and a corresponding angle of $\theta = -\frac{\pi}{2}-\eps $ and $\theta = \frac{3\pi}{2}+\eps$ for some $\eps >0$. The length of the vertical line is chosen such that the end points have a distance of $\delta$ to the origin. Then \eqref{total mean curvature for spherical surface} implies that the total mean curvature is given by
\begin{align}
&\quad \,\frac{1}{2\pi}\int_{\mb S^2} H_f \dif \vol_{g_f}\notag\\
 &= -\frac{1}{\eps} \cos \left (-\frac{\pi}{2}-\eps\right )\left (-\frac{\pi}{2}-\eps\right ) + \pi - \frac{1}{\eps}\cos \left (\frac{3\pi}{2}+\eps\right )\left (\frac{3\pi}{2}+\eps\right ) +\pi + \mc O(\delta)\notag \\
&= -2\eps + o(\eps) + \mc O(\delta)\label{total mean curvature of slightly negative surface}
\end{align}
as $\eps\to 0$, $\delta\to 0$. Now first choose $\eps$ and then $\delta$ sufficiently small such that the right-hand side of \eqref{total mean curvature of slightly negative surface} becomes negative.
\end{remark}
Theorem \ref{rot sym surface} also implies the following Corollary.
\begin{cor}\label{cor:Willmore 6pi}
Assume that $f:\mb S^2 \to \R^3$ is a $C^{1,1}$-axisymmetric sphere such that 
\[\int _{\mb S^2} H \dif \vol_g \leq 0.\]
Then, $\mc W(f) \geq 6\pi$.
\end{cor}
\begin{proof}
Notice that $H^2 \geq 2K + 2|K|$. Hence $\mc W(f) \geq 2\pi + \frac{1}{2}\int _{\mb S^2} |K|\dif \vol_g$. By \eqref{principal curvatures for axisymmetric surface}, we get the estimate
\begin{equation}
\mc W(f)\geq 2\pi + \pi \int_0^T |\theta'(s)| |\sin(\theta(s))|\dif s.\label{total variation estimate for W}
\end{equation}
By Theorem \ref{rot sym surface}, either $\frac{3\pi}{2}\in \im \theta$ or $-\frac{\pi}{2}\in \im \theta$ holds. Suppose without loss of generality that $\theta(t) = \frac{3\pi}{2}$ for some $t\in (0,T)$. Let $0<t'<t$ be such that $\theta(t') = \pi$. It follows that
\begin{align*}
&\quad \, \int_0^T |\theta'(s)| |\sin(\theta(s))|\dif s \\
&\geq \left |\int_0^{t'} \theta'(s) \sin(\theta(s)) \dif s\right | + \left |\int _{t'} ^t \theta'(s) \sin(\theta(s)) \dif s\right | +\left | \int_t^T \theta'(s) \sin(\theta(s))\dif s\right | \\
&=\cos(0)-\cos(\pi) + \cos \left (\frac{3\pi}{2}\right ) -\cos(\pi) + \cos \left (\frac{3\pi}{2}\right ) -\cos(\pi) \\
&= 4.\tag*\qedhere
\end{align*}
\end{proof}
\setcounter{secnumdepth}{0}
\stepcounter{section}
\section{Appendix}
Here, we state some results that were needed in the main sections. $\Sigma$ denotes a smooth, closed, oriented and connected surface and $f\in \mc E_\Sigma$ is a weak immersion.
\begin{lemma}\label{lem:Divergence theorem}
Suppose that $f\in \mc E_\Sigma$ is injective. Let $a\in \R^3$ and $r>0$. Denote $\Sigma_r \cqq f^{-1}(f(\Sigma) \cap B_r(a))$. Then, for a.e. $r$, there exists a measurable map $\eta:\partial \Sigma_r\to \mb S^2$ such that for $X\in C^1(\R^3, \R^3)$
\begin{equation}
\int _{\Sigma_r} \Div (X\circ f) \dif\vol_g = - \int_{\Sigma_r} \langle X\circ f,H\rangle \dif\vol_g + \int _{\partial \Sigma_r} \langle \eta, X\circ f\rangle \dif\sigma,\label{Divergence Theorem}
\end{equation}
where $f_\# \sigma = \mc H ^1 \mres{f(\Sigma)\cap \partial B_r(a)}$.
\end{lemma}
\begin{proof}
We view $f(\Sigma)$ as an integer varifold\footnote{See \cite{Allard, Mantegazza, SimonGMT} for a definition and introduction to varifolds.} $V$ in $\R^3$ of density 1 given by
\[V(\phi) \cqq \int _{\R^3} \phi(x, \tau(x))\dif f_\# \vol_g\]
for all $\phi \in C_0( G_{3,2}(\R^3))$, where $G_{3,2}(U) \cqq U \times G_{3,2}$ for any open subset $U$ of $\R^3$, $G_{3,2}$ is the metric space of all 2-dimensional subspaces of $\R^3$ and $\tau(x)$ denotes the orthogonal projection of $\R^3$ onto the tangent plane of $f$ at $x$ which is defined for $f_\# \vol_g$-a.e. $x$. Then, the first variation of $V\mres{G_{3,2}(B_r(a))}$ is given by
\begin{align*}
\delta (V\mres{G_{3,2}(B_r(a))}(X) &= \int _{\Sigma \cap f^{-1}(B_r(a))} \Div (X\circ f) \dif \vol_g\\
&= -\int _{\Sigma \cap f^{-1}(B_r(a))} \langle H,X\circ f\rangle \dif \vol_g - (V \partial B_r(a))(X),
\end{align*} 
where $V\partial B_r(a)$ is the distribution defined in \cite[Definition 5.1]{Menne}. Using the fact that ${h(x)= |x-a|}$ is a Lipschitz map, \cite[Example 8.7]{Menne} implies that $h$ is generalized $V$ weakly differentiable. Then, \cite[Theorem 8.30]{Menne} shows that for a.e. $r>0$, $V\partial B_r(a)$ is representable by integration, see \cite[4.1.5]{GMT} for a definition, and is supported in $f(\Sigma)\cap \partial B_r(a)$. It follows that $V\mres{G_{3,2}(B_r(a))}$ is a curvature varifold with boundary in the sense of \cite[Definition 3.1]{Mantegazza}, see \cite[Lemma 4.1]{Mantegazza}. By \cite[Theorem 7.1]{Mantegazza}, the total variation $|V \partial B_r(a)|$ is absolutely continuous with respect to $\mc H^1$. Finally, by \cite[Corollary 12.2]{Menne}, the 1-dimensional density of $|V \partial  B_r(a)|$ equals 1 $|V \partial B_r(a)|$-a.e. Therefore, it holds that $V\partial B_r(a) = \eta \mc H ^1 \mres{(\partial B_r(a)\cap f(\Sigma))}$ for some Borel map $\eta$ taking values in $\mb S^2$. This implies \eqref{Divergence Theorem}.
\end{proof}
\begin{remark}
If $X$ is merely Lipschitz, an approximation argument by smooth maps shows that Lemma \ref{lem:Divergence theorem} holds for $X$ as well.
\end{remark}
\begin{lemma} \emph{\cite[Lemma 1.4]{Simon}} \label{lem:Simon 8pi - C estimate}
Suppose $f\in \mc E_\Sigma$ is injective, and $\partial B_\rho(0) \cap f(\Sigma)$ contains two disjoint subsets $\Sigma_1,\Sigma_2$ with $\Sigma_j \cap B_{\theta \rho}(0) \neq \emptyset$, $ \partial \Sigma_j\subset \partial B_\rho(0)$ and $\mc H^1(\partial \Sigma_j) \leq \beta \rho$ for $j=1,2$, $\theta \in (0,1/2)$ and $\beta > 0$. Assume further that for $ I_0\circ f \in \mc E_\Sigma$ and $r = \frac{1}{\rho}$, \eqref{Divergence Theorem} holds, where $I_0$ is defined by \eqref{sphere inversion}. Then,
\[\mc W(\Sigma) \geq 8\pi - C\beta \theta,\]
where $C$ does not depend on $\Sigma$, $\beta$ or $\theta$.
\end{lemma}
The proof given in \cite{Simon} assumes that $f(\Sigma)$ is a smooth surface. However, all that is needed is the divergence theorem \cite[(1.5)]{Simon} applied to $(I_0\circ f)(\Sigma)$, which is replaced by  \eqref{Divergence Theorem}.

\begin{lemma} \emph{\cite[Lemma 1.1]{Simon}} \label{lem:Simon}
Suppose that $f\in \mc E_\Sigma$. Then,
\[\sqrt{ \frac{\mc A(\Sigma)}{\mc W(\Sigma)}} \leq \diam(\Sigma) \leq C \sqrt{\mc A(\Sigma) \mc W(\Sigma)}.\]
\end{lemma}
The proof in \cite{Simon} assumes that $f(\Sigma)$ is a smooth surface, but once more, only the divergence theorem is needed.
\begin{lemma}[Existence and uniqueness of biharmonic equation] \label{lem: biharmonic equation}
Suppose that $\Omega\subset \R^n$ is open and bounded with $C^1$ boundary. For $f\in H^{3/2}(\partial \Omega), g\in H^{1/2}(\partial \Omega)$, the equation
\begin{equation}
\begin{cases}
\Delta^2 u &=0 \text{ in } \Omega\\
 u &= f \text{ on } \partial \Omega\\
 \partial _\nu u &= g \text{ on }\partial \Omega
\end{cases}
\end{equation}
has a unique weak solution $u\in H^2(\Omega)$, i.e.
\[\int _\Omega \Delta u \Delta \phi \dif x = 0 \quad \forall \phi \in H_0^2(\Omega)\] satisfying
\[\|u\|_{H^2(\Omega)} \leq C( \|f\|_{H^{3/2}(\Omega)} + \|g\|_{H^{1/2}(\Omega)}).\]
\end{lemma}
\begin{proof}
Consider $v\cqq Z(f,g) \in H^2(\Omega)$ given by the extension operator from \cite[Theorem 5.8]{Necas} such that $v = f\text{ on } \partial \Omega$ and $ \partial_\nu v = g \text{ on } \partial \Omega$. We consider $w = u - v$ so that $w \in H_0^2$ needs to solve
\[\int _\Omega \Delta w \Delta \phi \dif x = - \int_\Omega \Delta v \Delta \phi \dif x\quad \forall \phi \in H_0^2(\Omega).\]
The right-hand side is a bounded functional in $H_0^2(\Omega)$ and so by Riesz representation theorem, there exists a unique solution $w \in H_0^2(\Omega)$. Then, $u \cqq w +v$ is the unique solution to the initial problem. Riesz representation theorem also implies
\[\|u\|_{H^2(\Omega)} \leq C\|v\|_{H^2(\Omega)} \leq C ( \|f\|_{H^{3/2}(\Omega)} + \|g\|_{H^{1/2}(\Omega)}),\]
where the last step follows by the boundedness of the trace extension operator. 
\end{proof}

\addcontentsline{toc}{section}{References}
\bibliographystyle{abbrv}
\bibliography{literatur}

(A. West) \textsc{Institute for applied mathematics, University of Bonn, Endenicher Allee 60, 53115 Bonn, Germany.}\\
\emph{Email address:} \href{mailto:s6arwest@uni-bonn.de}{s6arwest@uni-bonn.de} 
\medskip\newline
(C. Scharrer) \textsc{Institute for applied mathematics, University of Bonn, Endenicher Allee 60, 53115 Bonn, Germany.}\\
\emph{Email address:} \href{mailto:scharrer@iam.uni-bonn.de}{scharrer@iam.uni-bonn.de} 

\end{document}